\documentclass[a4paper, 13pt, reqno]{amsart}
\usepackage[dvipsnames]{xcolor}
\usepackage[T1]{fontenc}
\usepackage{hyperref}
\hypersetup{
colorlinks = True,
linkcolor=Bittersweet,
citecolor=blue,
}
\usepackage{amssymb}
\usepackage{amsthm}
\usepackage{bm}
\usepackage{latexsym}
\usepackage{amsmath}

\usepackage{eufrak}
\usepackage{mathrsfs}
\usepackage{caption}
\usepackage{xparse}
\usepackage{amscd}
\usepackage[all,cmtip]{xy}

\usepackage{placeins}
\usepackage{afterpage}
\usepackage{amscd}
\usepackage{float}
\usepackage{graphics}
\usepackage{tikz}
\usetikzlibrary{graphs}
\usetikzlibrary{decorations.pathreplacing,angles,quotes}
\usetikzlibrary{shapes.geometric}

\usepackage{tikz-cd}
\usepackage{comment}
\usepackage{xspace}
\usepackage{mathtools}
\usepackage{pifont} %
\usepackage{amssymb}
\usepackage{amsthm}
\usepackage{graphicx}
\usepackage{subfig}
\usepackage{enumerate}
\usepackage{hyperref}

\usetikzlibrary{patterns,decorations.pathreplacing}
\usepackage{marginnote}

\usepackage{cancel}

\theoremstyle{plain}

\newtheorem{theorem}{Theorem}[section]
\newtheorem{proposition}[theorem]{Proposition}
\newtheorem{lemma}[theorem]{Lemma}
\newtheorem{corollary}[theorem]{Corollary}

\newtheorem{algorithm}{Algorithm}
\newtheorem{question}{Question}

\theoremstyle{definition}

\newcommand{\appsection}[1]{\let\oldthesection\thesection
\renewcommand{\thesection}{Appendix \oldthesection}
\section{#1}\let\thesection\oldthesection}

\newtheorem{definition}[theorem]{Definition}
\newtheorem{notation}[theorem]{Notation}

\theoremstyle{remark}

\newtheorem{remark}[theorem]{Remark}
\newtheorem{example}[theorem]{Example}

\def\D{{\mathbb{D}}}
\def\R{{\mathbb{R}}}
\def\Z{{\mathbb{Z}}}
\def\F{{\mathbb{F}}}
\def\Q{{\mathbb{Q}}}
\def\C{{\mathbb{C}}}
\def\P{{\mathbb{P}}}
\def\PP{{\mathcal{P}}}

\def\O{{\mathcal{O}}}

\def\Y{{\mathcal{Y}}}

\def\W{{\mathcal{W}}}

\def\bG{{\overline{\Gamma}}}
\def\bE{{\overline{E}}}

\usepackage{xcolor}

\tikzset{every loop/.style={min distance=7mm,in=120,out=60,looseness=10}}

\begin{document}

\title{Classification of Horikawa surfaces with T-singularities}
\author[Vicente Monreal]{Vicente Monreal}
\email{vicente.monreal@gmail.com}
\address{Mathematisches Institut, Heinrich-Heine-Universit\"at D\"usseldorf, Germany.}

\author[Jaime Negrete]{Jaime Negrete}
\email{jaime.negrete@uga.edu}
\address{Department of Mathematics, University of Georgia, Athens, USA.}

\author[Giancarlo Urz\'ua]{Giancarlo Urz\'ua}
\email{gianurzua@gmail.com}
\address{Facultad de Matem\'aticas, Pontificia Universidad Cat\'olica de Chile, Santiago, Chile.}

\date{\today} 

\begin{abstract}
We classify all projective surfaces with only T-singularities, ample canonical class, and $K^2=2p_g-4$. In this way, we identify all surfaces, smoothable or not, with only T-singularities in the Koll\'ar--Shepherd-Barron--Alexeev (KSBA) moduli space of Horikawa surfaces. We also prove that they are not smoothable when $p_g \geq 10$, except for the Lee-Park (Fintushel-Stern) examples, which we show to have only one deformation type unless $p_g=6$ (in which case they have two). This demonstrates that the challenging Horikawa problem cannot be addressed through complex T-degenerations. We propose new questions regarding diffeomorphism types based on our classification. Furthermore, the techniques developed in this paper enable us to classify all KSBA surfaces with only T-singularities and $K^2\leq 2p_g-3$, for example, quintic surfaces and I-surfaces.
\end{abstract}

\maketitle

\section{Introduction} \label{s0}

A \textit{Horikawa surface} is a minimal nonsingular complex projective surface $Y$ of general type satisfying $$K_Y^2=2p_g(Y)-4,$$ where $K_Y^2$ is the self-intersection of the canonical class, and $p_g(Y)$ is the dimension of the global holomorphic 2-forms on $Y$. By Noether's inequality, this is the minimum possible value of $K^2$ among all minimal surfaces of general type. The name comes from Horikawa's classification of all such surfaces in a series of articles in the 70s, which also includes surfaces with invariants close to the Noether line $K^2=2p_g-4$. These surfaces are simply-connected, and their topological type as closed 4-manifolds can be determined using Freedman's theorem. There are spin and nonspin surfaces for $K^2=8t$ with $t$ odd; for all other cases, the surfaces are nonspin. Horikawa also described the connected components and dimensions of their moduli spaces once $K^2$ is fixed. If $K^2$ is divisible by $8$, then there are two connected components of the moduli space, each of the same dimension; otherwise, there is only one component. For $K^2=16t$ it is an open problem to decide whether \textit{surfaces in distinct components are diffeomorphic or not}. All known smooth invariants are equal, and this question has remained open since Horikawa raised it in \cite{H76}. Some places where this problem is discussed are \cite{C93}, \cite[4.101]{K2}, \cite{CM95}, \cite{FS97}, \cite{GS99}, \cite{A06}, \cite{FS09}, \cite{E18}, \cite{RR22}, \cite{CP25}. We call it the \textit{Horikawa problem}.  
\vspace{0.1cm}

One approach to attack this problem involves finding a common degeneration for the two components in question. That is, we would like to find complex deformations of Horikawa surfaces from each of these components that meet on a singular surface. In terms of moduli spaces, this means seeking common points within a certain compactification. A geometric way to do this was introduced by Koll\'ar--Shepherd-Barron in \cite{KSB88}, which, combined with the boundedness result of Alexeev \cite{A94}, gives a compactification. Consequently, there exists a Koll\'ar--Shepherd-Barron--Alexeev (KSBA) coarse moduli space of surfaces with fixed $K^2$ and $\chi$, which is a projective scheme that contains the Gieseker moduli space of canonical surfaces of general type \cite{Gi77}. The corresponding KSBA surfaces have semi-log-canonical singularities and ample canonical class. Unlike the Deligne-Mumford moduli space of curves, a KSBA surface may not be smoothable. Indeed, there are examples of entire components of nonsmoothable KSBA surfaces (see e.g. \cite{RR22}). 
\vspace{0.1cm}

For Horikawa surfaces, an explicit description of all KSBA surfaces remains unknown. Recently, Rana and Rollenske \cite{RR22} found the first example of a common degeneration for these components, but it is not clear whether this degenerated surface could answer the Horikawa problem as there is no control on the resulting surgery on the singularities. On the other hand, Manetti showed that two degenerations into a surface with only \textit{T-singularities}, to be recalled below, induces an oriented diffeomorphism between the nonsingular surfaces from both families \cite[Theorem 1.5]{M01}. Manetti used that result to prove that there are moduli spaces of surfaces of general type with an arbitrarily large number of connected components which parametrize diffeomorphic surfaces. T-singularities are the 2-dimensional quotient singularities which admit a $\Q$-Gorenstein smoothing \cite{KSB88}. They are Du Val singularities or cyclic quotient singularities of type $\frac{1}{dn^2}(1,dna-1)$ with gcd$(a,n)=1$. The most relevant are non-Du Val T-singularities with $d=1$. They are called \textit{Wahl singularities}, as Wahl was the first who studied them as singularities that admit smoothings that preserve the self-intersection of the canonical class \cite{W80,W81}. It is more convenient to think about T-singularities via their \textit{Hirzebruch-Jung continued fraction} $\frac{dn^2}{dna-1}=[e_1,\ldots,e_r]$; we call it \textit{T-chain}. This represents the chain of $\P^1$s which minimally resolves $\frac{1}{dn^2}(1,dna-1)$. If $[e_1, \ldots, e_r]$ is a T-chain, then $[2, e_1, \ldots, e_{r-1},e_r+1]$ and $[e_1+1,e_2, \ldots,e_r,2]$ are T-chains. Every T-chain is obtained by iterating this operation, starting with either $[4]$ (to get all \textit{Wahl chains}) or $[3,2, \ldots,2,3]$ with $(d-2)$ $2$s. Let us call a surface with only T-singularities (not all Du Val) a \textit{T-surface}. 
\vspace{0.1cm}

Therefore, there is a two-step approach to answer the Horikawa problem in the positive: 

\begin{itemize}
\item[\textbf{(I)}] Find all T-surfaces $W$ with $K_W^2=2p_g(W)-4$.
\item[\textbf{(II)}] Study the smoothability of the surfaces classified in \textbf{(I)}.
\end{itemize}

In \cite{ESU24} the authors achieved to answer \textbf{(II)} without knowing \textbf{(I)} for $p_g=3$, which is the smallest possible geometric genus. In general, a difficulty with this approach is that there are no examples of Horikawa T-surfaces for $p_g>3$ besides the Lee-Park examples \cite{LP11}, which were originally discovered by Fintushel-Stern as smooth 4-manifolds in \cite{FS97}. We will describe it here, as it will be relevant to the present work. 

\begin{example}
Let $p_g\geq 3$, and let $S \to \P^1$ be an elliptic fibration with sections of self-intersection $-(p_g+1)$. In this way, $p_g(S)=p_g$. Assume that we can find $2$ disjoint Wahl chains of $\P^1$s whose Hirzebruch-Jung continued fraction is $[p_g+1,2,\ldots,2]$, where the number of $2$s is $p_g-3$ and the first curve is a section. There are plenty of these situations. By Artin's contractibility criteria \cite{A62}, we can contract these two Wahl chains via a birational morphism $S \to W$, where $W$ is a projective surface with two Wahl singularities $\frac{1}{(p_g-1)^2}(1,p_g-2)$. The surface $W$ has $K_W$ big and nef, and $K_W^2=2p_g-4$. We have $p_g(W)=p_g$. We call them \textit{Lee-Park examples} (compare with \cite[Example 8.5.5(b)]{GS99}). In \cite{LP11}, Lee and Park constructed particular such surfaces as an ingenious double cover $W \to W_0$, where $W_0$ has only one $\frac{1}{(p_g-1)^2}(1,p_g-2)$ and the branch locus is disjoint to this singularity. The point is that $W_0$ has $\Q$-Gorenstein smoothings which lift the branch locus, and so this induces $\Q$-Gorenstein smoothings for $W$. The smoothings are Horikawa surfaces, and one can show that they live in the non-spin component when $K^2=8t$. This last part can be done via the birational geometry computation in Theorem \ref{LeePark}. We actually prove in that theorem that these are the only possible smoothings for $K^2>8$. In Example \ref{spinK^2=8} we show that there are spin and nonspin smoothings when $K^2=8$.
\label{leepark}
\end{example}

In this paper, we give a complete answer to \textbf{(I)}, and \textbf{(II)} for $p_g\geq 10$. This shows that the Horikawa problem cannot be solved via complex degenerations with only T-singularities. From our classification, it turns out that, apart from Lee-Park Example \ref{leepark}, all the rest of Horikawa T-surfaces belong to one big family of KSBA surfaces which are completely classified in this paper as well. Here is the definition.

\begin{definition}
Let $W$ be a surface with only non Du Val T-singularities, and $K_W$ big and nef. Let $\phi \colon X \to W$ be its minimal resolution with exceptional (reduced) divisor Exc$(\phi)$, and let $\pi \colon X \to S$ be the composition of blow-ups into the minimal model $S$ of $X$. Assume that the Kodaira dimension of $S$ is $1$. Hence it must admit an elliptic fibration $S \to \P^1$, the base must be indeed $\P^1$. We say that $W$ is a \textit{small surface} if $\pi(\text{Exc}(\phi))$ contains precisely one section of $S \to \P^1$ and components of its singular fibers.     
\end{definition}

\begin{remark}
We require $K_W$ big and nef because otherwise we can reduce to that via MMP. This is the MMP described in \cite{HTU17} but done only on $W$ as we do not have a smoothing necessarily. It could be seen as an MMP of the rational blowdown of the minimal resolution of the M-resolution of $W$ (see Subsection \ref{topo}). Also we do not require $K_W$ ample since the canonical model of $W$ (with $K_W$ big and nef) has only T-singularities. We are considering an M-resolution of that canonical model (see Remarks \ref{canonical} and \ref{notation}).  
\end{remark}

\begin{theorem} [Theorem \ref{classHorikawa}]
Let $W$ be a surface with only non Du Val T-singularities, $K_W$ big and nef, and $K_W^2=2p_g(W)-4$. Then $W$ is either a Lee-Park Example \ref{leepark} or a small surface. In the latter case we have exactly $3$ explicit families for $p_g(W)=3$, $9$ explicit families for $p_g(W)=4$, and $8$ explicit families for every $p_g(W)\geq 5$. 
\label{main1}
\end{theorem}

Each of these families are described in Theorem \ref{classHorikawa}. A consequence of this classification is that we can study potential smoothings explicitly, by considering the particular sets of singularities in small surfaces. This breaks into: involutions of T-singularities \ref{s41}, extending canonical involution of Horikawa surfaces in the family and strong constraints on the number of singularities in the quotient family \ref{s42}, and some birational geometry of the quotient family \ref{s43}. We prove the following.   

\begin{theorem} [Theorem \ref{NOsmoothable}]
For $p_g\geq 10$, the only KSBA degenerations of Horikawa surfaces with only T-singularities (not all Du Val) are the Lee-Park examples, and only for the nonspin component.
\label{main1.1}
\end{theorem}

For $p_g\leq 9$ there is only one moduli space with two components, namely $p_g=6$. In Example \ref{spinK^2=8} we construct Horikawa T-surfaces for both components, and this represents the only smoothing of a Horikawa T-surface in the spin component according to Theorem \ref{nospin}. A comprehensive study of all smoothable Horikawa T-surfaces for $4\leq p_g\leq 9$ will be addressed elsewhere. Additionally, the small surfaces described in Theorem \ref{main1} admit a rational blowdown into closed symplectic 4-manifolds; this is studied in \ref{topo}. Most of them are simply-connected and thus homeomorphic to each other by Freedman's theorem.

\begin{question}
Are rational blowdowns of small Horikawa surfaces diffeomorphic?
\label{newHori}
\end{question}

Since small surfaces turn out to be present for a wide range of $K^2$, we give a classification of all of them. 

\begin{theorem}[Theorems \ref{formulaK^2}, \ref{ClassSmallSurf}, \ref{geogrSmall}]
Let $W$ be a small surface with $l$ singularities, and let $N$ be the number of fibers of $S \to \P^1$ completely contained in $\pi(\text{Exc}(\phi))$.   Then:
\begin{itemize}
    \item[(1)] We have $l \geq \text{max} \{1,N-1\}$ and $K_W^2=p_g(S)-2+N$.

    \item[(2)] The surface $W$ is constructed from one of the $14$ building blocks in the list \ref{blocks} of type S0F, S1F.i, S2F.i, plus a suitable number of FIBs.
    
    \item[(3)] We have $p_g(S) -2\leq K^2_{W}\leq \Big(4+\frac{2}{3} \Big) p_g(S) +\frac{11}{3}$. The equality on the left holds if and only if $W$ is the contraction of one chain $[p_g(S)+1,2,\ldots,2]$ in $S$. The equality on the right holds if and only if $p_g \equiv 2$(mod $3$) and $W$ is obtained via S2F.7 gluing a suitable number of FIBs. Every $K^2$ is realizable by some $W$.
\end{itemize}
\label{main2}
\end{theorem}

Part (1) says that small surfaces tend to have lots of singularities. In particular, Horikawa small surfaces with $K^2=16t$ have at least $8t-1$ singularities. Part (2) precisely constructs all of them, and part (3) states that small surfaces have slopes $K^2/\chi(\O) \in [1,4.\bar{6}]$, and so they are abundant in geography.
\vspace{0.1cm}

\begin{remark}
We prove in Corollary \ref{hori} that any T-surface $W$ with $K_W^2 < 2p_g(W)-4$ must be a small surface. Note that there are no $\Q$-Gorenstein smoothings for them as $K^2$ stays constant, and so it would violate the Noether's inequality. Instead of complex smoothings we can consider the corresponding rational blowdowns (see \ref{topo}). Then we are proving in this paper that every rational blowdown $Y$ from an algebraic surface over algebraic Wahl chains which violates Noether's inequality comes from a small surface, and we classify all of them. In particular, in geographical terms, they are always between the Noether line and half the Noether line, and when they lie on the half Noether line the rational blowdown is over a unique Wahl chain of type $[p_g(S)+1,2,\ldots,2]$ in $S$. 
\label{halph-Noether}
\end{remark}

\vspace{0.1cm}

To prove that Horikawa T-surfaces are either Lee-Park examples or small surfaces, it is key the following general new inequality. It is based on some results from \cite{FRU23}. A baby version is \cite[Corollary 3.13]{FRU23}, but the full version requires new ideas to strongly bound the geometry of the rational configurations $\pi(\text{Exc}(\phi))$ (see Section \ref{s1}). Let us consider the diagram $$ \xymatrix{  & X  \ar[ld]_{\pi} \ar[rd]^{\phi} &  \\ S & & W}$$ where $\phi$ is the minimal resolution of $W$, and $\pi$ is a composition of blow-ups such that $S$ has no $(-1)$-curves. 

\begin{theorem} [Corollary \ref{ineqq}]
Let $W$ be a T-surface with $K_W$ big and nef, and $l$ singularities. Then, $$K_S \cdot \pi(\text{Exc}(\phi)) \leq K_W^2-K_S^2 + \text{min}\{|J|,l\},$$ where $J$ is the set of curves $\Gamma$ in $\pi(\text{Exc}(\phi))$ such that $\Gamma \cdot K_S\neq 0$. Equality holds if and only if $X=S$.
\label{main3}
\end{theorem}

Therefore every KSBA T-surface satisfying $K^2<2p_g-4$ is a small surface (see Corollary \ref{hori}). It also implies the following general structure theorem for T-surfaces with low $K^2$. A \textit{double section} is an irreducible curve in an elliptic fibration whose intersection with fibers is $2$.

\begin{theorem} [Corollary \ref{low}]
Let $W$ be a T-surface with $K_W$ big and nef. Assume $K_W^2 <3p_g(W) -6$, and that the Kodaira dimension of $S$ is $1$. Thus  there is an elliptic fibration $S \to \P^1$. Then $\pi(\text{Exc}(\phi))$ contains exactly either one double section, two sections, or one section, and some components of singular fibers.
\label{main4}
\end{theorem}

Horikawa classified surfaces with $K^2=2p_g-3$ in \cite{H76b}. Using Theorem \ref{main4}, we can classify all T-surfaces $W$ with $K_W^2=2p_g(W)-3$, which of course includes I-surfaces ($p_g=2$) and quintic surfaces ($p_g=4$). We point out that the corresponding KSBA surfaces have been studied e.g. in \cite{CFPR22}, \cite{FPRR22}, \cite{CFPRR23}, \cite{GPSZ24}, \cite{RT24}, \cite{CFPR24} for I-surfaces, and in \cite{R14}, \cite{G19} for quintics. We can prove that any T-surface $W$ with $K_W$ big and nef with $K_W^2=2p_g(W)-3$ must belong to one of the options: 

\begin{itemize}
    \item[(A)] The surface $S$ is a Horikawa surface and we contract one $[4]$ or $[3,2,\ldots,2,3]$ in $S$. Smoothability of these surfaces has recently been studied in \cite{CP25}.
    \item[(B)] There is a minimal elliptic fibration $S \to \P^1$ and a double section $D$ with $D^2=-2p_g$, and $W$ is the contraction of one Wahl chain $[2p_g,2,\ldots,2]$ starting with $D$.   
    \item[(C)] The surface $W$ is constructed from an elliptic fibration $S \to \P^1$ such that $\pi(\text{Exc}(\phi))$ contains precisely $2$ sections and components from fibers.  
    \item[(D)] The surface $W$ is small.
\end{itemize}

The case (C) of two sections brings a new huge world of surfaces which will be studied in another article. The details for case (A) will be given in that article. Case (D) can be classified via Theorem \ref{main2}. As an example, we prove that I-surfaces with T-singularities are precisely the surfaces already studied in \cite{CFPRR23}.

\begin{theorem}
Let $W$ be a surface with only non Du Val T-singularities, $K_W$ big and nef, $p_g=2$ and $K_W^2=1$. Then $W$ is one of the following:
\begin{itemize}
    \item[(i)] Contraction of a double section $D$ with $D^2=-4$ in an elliptic surface $S \to \P^1$.

    \item[(ii)] Contraction of a configuration $[3,2,\ldots,2,3]$ in an elliptic surface $S \to \P^1$ where the two $(-3)$-curves are sections.

    \item[(iii)] A small surface from one S1F.4. Hence we have one T-chain $[2,3,4]$, and $n$ $[4]$ for some $0\leq n \leq 32$. When $n>1$, the chains $[4]$ can be arranged in one $[3,2,\ldots,2,3]$.

    \item[(iv)] A small surface from one S1F.1. Hence we have one $[3,5,2]$ and $n$ $[5,2]$ for some $0\leq n \leq 33$. When $n>1$, the chains $[5,2]$ can be arranged in one $[4,2,\ldots,2,3,2]$.
\end{itemize}
\label{main5}
\end{theorem}

For quintic surfaces we have (A) with a Horikawa surface with $K^2=4$, and (B) using a double section $\Delta$ on $S$ with $\Delta^2=-8$, and contracting a $[8,2,2,2,2]$. The case (C) has lots of geometric situations, and so it will be considered in a sequel to this paper. The case (D) is the next corollary.  

\begin{corollary}
Let $W$ be a small surface with $K_W$ big and nef, $p_g=4$ and $K_W^2=5$. Then $W$ is constructed through one of the following combinations of building blocks:
\begin{itemize}
    \item[(i)] One S0F with $r=11$ and $3$ FIBs.

    \item[(ii)] One S1F.i with $i\in \{2,4\}$ and $r=9$, and $2$ FIBs.

    \item[(iii)] One S2F.i with $i\in \{4,5,6,7,8\}$ and $r=7$, and $1$ FIB.
\end{itemize}
\label{main6}
\end{corollary}

The classification of small surfaces is determined by the classification of P-resolutions of particular cyclic quotient singularities. This is developed in Section \ref{s2} and \ref{app}. P-resolutions was a tool used by Koll\'ar and Shepherd-Barron \cite[Section 3]{KSB88} to study the deformation spaces of quotient singularities. The way P-resolutions appear for small surfaces is through the T-chain in $X$ that contains the relevant section of $S \to \P^1$. In an intermediate blow-up of $S$, we show that this T-chain must be part of a P-resolution of a particular singularity. Similarly for T-chains totally included in fibers, which come, as we prove in this paper, from the unique building block FIB in the list \ref{blocks}. We can verify that discrepancies of exceptional divisors work out for the plumbings of the 15 building blocks via a careful analysis of all possible P-resolutions.    

\subsubsection*{Acknowledgments} We thank Fabrizio Catanese, Igor Dolgachev, Jonny Evans, DongSeon Hwang, Pedro Montero, Matthias Sch\"utt, and Roberto Villaflor for valuable discussions. Special thanks to Juan Pablo Z\'u\~niga for discovering and proving Theorem \ref{JPZ}, and to Makoto Enokizono for pointing out an error in Proposition 5.6 in Section \ref{s4} of the first version. This paper was partially written while the third-named author was at the Freiburg Institute for Advanced Studies under a Marie S. Curie FCFP fellowship (2023-2024). The first-named author was funded by ANID scholarship 22221734, and the third-named author was also supported by FONDECYT regular grant 1230065.

\tableofcontents

\section{An inequality for rational KSBA configurations} \label{s1}

Let us recall first some basic definitions and facts. For $0<q<m$ coprime integers, a \textit{cyclic quotient singularity} (c.q.s.) $\frac{1}{m}(1,q)$ is the surface germ at the image of $(0,0)$ of the quotient of $\C^2$ by $(x,y) \mapsto (\zeta x, \zeta^q y)$, where $\zeta$ is an $m$-th primitive root of $1$. A \textit{T-singularity} is either a Du-Val singularity or a c.q.s. $\frac{1}{dn^2}(1,dna-1)$ with $0<a<n$ coprime and $d\geq 1$. When $d=1$, we call it \textit{Wahl singularity}. A singularity $\frac{1}{m}(1,q)$ can be minimally resolved by a chain of nonsingular rational curves $C_1,\ldots,C_r$ where $C_i^2=-e_i \leq -2$ and $\frac{m}{q}=[e_1,\ldots,e_r]$, which is the \textit{Hirzebruch-Jung continued fraction} of $\frac{m}{q}$. The symbol $[e_1,\ldots,e_r]$ will also refer to these chains of curves. T-singularities (Wahl singularities) are minimally resolved by \textit{T-chains} (\textit{Wahl chains}). As usual, if $\sigma \colon \widetilde{Y} \to Y=\frac{1}{m}(1,q)$ is the minimal resolution, then one can write the numerical equivalence $$K_{\widetilde{Y}} \equiv \sigma^*(K_Y) + \sum_{i=1}^r \delta_i C_i$$ where $\delta_i \in ]-1,0]$ is (by definition) the \textit{discrepancy} at $C_i$. For T-chains we have the well-known properties (see e.g. \cite[Section 1]{FRU23}):

\begin{proposition}
For non-ADE T-singularities $\frac{1}{dn^2}(1,dna-1)$ we have:

\begin{itemize}
    \item[(i)] If $n=2$ then the T-chain is either $[4]$ or $[3,2, \ldots,2,3]$, where the number of $2$s is $d-2$. In this case, all discrepancies are equal to $-\frac{1}{2}$.
    \item[(ii)] If $[e_1,e_2,\ldots,e_r]$ is a T-chain for a given $d$, then $[2,e_1, \ldots, e_{r-1},e_r+1]$ and $[e_1+1,e_2, \ldots,e_r,2]$ are T-chains for the same $d$.
    \item[(iii)] Every T-chain can be obtained by starting with one of the singularities in (i) and iterating the steps described in (ii).
    \item[(iv)] Consider a T-chain $[e_1, \ldots, e_r]=\frac{dn^2}{dna-1}$ with discrepancies $-1+\frac{t_1}{n}, \ldots, -1+\frac{t_r}{n}$ respectively, where $t_1+t_r=n$. Then $[e_1+1,e_2, \ldots,e_r,2]$ has discrepancies $-1+\frac{t_1}{n+t_1},\ldots,-1+\frac{t_r}{n+t_1},-1+\frac{t_1+t_r}{n+t_1}$, and $[2,b_1, \ldots,b_r+1]$ has discrepancies $-1+\frac{t_1+t_r}{n+t_r},$ $-1+\frac{t_1}{n+t_r},\ldots,-1+\frac{t_r}{n+t_r}$, respectively.
    \item[(v)] Given the T-chain $[e_1,\ldots,e_r]$, the discrepancy of an ending $(-2)$-curve is $>-\frac{1}{2}$.
\end{itemize}

\label{T-chain}
\end{proposition}

For a non-ADE T-singularity $\frac{1}{dn^2}(1,dna-1)$, we define its \textit{center} to be the collection of exceptional divisors in the corresponding T-chain which have the lowest discrepancy, that is, the curves in the center all have discrepancy $-\frac{n-1}{n}$. In this way, the center is the collection of curves corresponding to the initial step (i) in the algorithm (i), (ii), (iii) in Proposition \ref{T-chain}.

Let $W$ be a projective surface with $K_W$ big and nef, and only non-Du Val T-singularities. Let $\frac{1}{d_in_i^2}(1,d_in_ia_i-1)$ be its T-singularities, where $i \in \{1,\ldots,l\}$, and so $l$ is the number of singularities in $W$. Let us consider the diagram of morphisms
$$ \xymatrix{  & X  \ar[ld]_{\pi} \ar[rd]^{\phi} &  \\ S &  & W}$$
where $\phi$ is the minimal resolution of $W$, and $\pi$ is a composition of $m$ blow-ups such that $S$ has no $(-1)$-curves. We will use the same notation as in \cite{R14,RU17,FRU23}. Let $E_i$ be the pull-back divisor in $X$ of the $i$-th point blown-up through $\pi$. Therefore, $E_i$ is a connected, possibly non-reduced tree of $\P^1$s, $E_i^2=-1$, and $E_i\cdot E_j=0$ for $i\neq j$. Let $E:=\sum_{i=1}^m E_i$. We have that $$K_X \sim \pi^*(K_S) +E.$$ Let $C=\sum_{j=1}^l C_j=\sum_{j=1}^l \sum_{i=1}^{r_i}C_{j,i}$ be the exceptional (reduced) divisor of $\phi$, where $C_j=\sum_{i=1}^{r_j}C_{j,i}$ is the T-chain of the singularity $\frac{1}{d_j n_j^2}(1,d_j n_j a_j-1)$. We have the formula $ K_S^2-m+\sum_{j=1}^l(r_j-d_j+1)=K_W^2.$

To bound singularities on $W$ for a fixed $K_W^2$, it is key to study the intersection $E \cdot C$. We have the formula (see \cite[Lemma 3.3]{FRU23}) $$E \cdot C = \sum_{j=1}^l (r_j - d_j + 2) - K_S \cdot \pi(C)= K_W^2 -K_S^2 + m + l - K_S \cdot \pi(C).$$ Recall that $E_i \cdot \Big(\sum_{C_{k,j} \subseteq E_i} C_{k,j} \Big)$ is equal to $0$ or $-1$, via the projection formula. It is $-1$ if and only if $E_i$ is the pull-back of a $(-1)$-curve whose proper transform belongs to $C$.

\begin{definition}
Let $S_h$ be the set of $E_i$s such that $E_i \cdot \Big(\sum_{C_{k,j} \nsubseteq E_i} C_{k,j} \Big)=h$. Let $T_h$ be the set of $E_i$s in $S_h$ satisfying $E_i \cdot \Big(\sum_{C_{k,j} \subseteq E_i} C_{k,j} \Big)=0$. Let $s_h=|S_h|$ and $t_h=|T_h|$. \footnote{For a set $T$, its cardinality will be denoted by $|T|$.}
\end{definition}

\begin{definition}
Let us consider the following partition of $C$: 
\begin{itemize}
    \item $M:= \{ C_{i,j} \ \text{contracted by} \ \pi \}$,
    \item $J:= \{ C_{i,j} \ \text{such that} \ \pi(C_{i,j}) \cdot K_S \neq 0 \}$, and 
    \item $J^c:= \{ C_{i,j} \ \text{such that} \ \pi(C_{i,j}) \cdot K_S = 0 \ (\text{and} \ \pi(C_{i,j}) \ \text{a curve in } S)\}$. 
\end{itemize}
\end{definition}

In \cite[Section 3]{FRU23}, it was proved that $s_0=s_1=0$ \footnote{In this paper the hypothesis on $K_W$ is big and nef. In \cite{FRU23} $K_W$ is ample. We will be careful when we use results from \cite{FRU23}. In the case of $s_0=s_1=0$ there are no discrepancies since this is based on \cite[Section 2]{FRU23}, and there $K_W$ can be taken big and nef.} (and so $t_0=t_1=0$). As $\sum_{h \geq 0} s_h =m$, we have $$E \cdot C = \sum_{h \geq 2} \left( \sum_{E_i \in S_h} E_i \right) \cdot C = -|M| + \sum_{h \geq 2} h s_h = -|M| + 2m + \sum_{h\geq 3} (h-2) s_h,$$ and $E \cdot \Big(\sum_{C_{k,j} \subseteq E_i} C_{k,j} \Big)=-|M|$. Note that $\sum_{h\geq 2} t_h= m -|M|$, and so we have $$E \cdot C=m + \sum_{h \geq 2} t_h +\sum_{h\geq 3} (h-2) s_h= m + \sum_{h \geq 2} (h-1)t_h + \sum_{h\geq 3} (h-2) (s_h-t_h).$$

\bigskip 

Let us construct an auxiliary graph $G$. Its vertices are either the $E_i$ in $\cup_{h\geq 2} T_h$, or each of the $l$ T-chains. It will only have edges between the vertices of these $E_i$s and the vertices corresponding to T-chains. For a $E_i$ vertex and a vertex of a T-chain $C_j$, we assign $E_i \cdot C_j$ edges between them. Let $l_c$ be the number of connected components of $G$, and let $\alpha$ be the rank of $H^1(G,\Z)$. Then the Euler characteristic of $G$ gives $$\Big(\sum_{h\geq 2} t_h +l \Big) - \Big(\sum_{h\geq 2} h t_h \Big) = l_c -\alpha.$$ Thus $\sum_{h\geq 2} (h-1)t_h-l+l_c \geq 0$, and it is equal to zero if and only if there are no cycles in $G$. 

\begin{corollary}
We have $K_W^2-K_S^2-K_S\cdot \pi(C)+l_c=\alpha+\sum_{h\geq 3} (h-2) (s_h-t_h)\geq 0$. Thus we have equality if and only if $G$ is a tree and $S_h=T_h$ for $h\geq 3$.
\label{euler}
\end{corollary}

\begin{proof}
This is the equality $$K_W^2 -K_S^2 + m + l - K_S \cdot \pi(C)=E \cdot C=m + \sum_{h \geq 2} (h-1)t_h + \sum_{h\geq 3} (h-2) (s_h-t_h)$$ together with the equation for the Euler characteristic of $G$.
\end{proof}

\begin{lemma}
Assume that $G$ has no cycles. Let $E_i \in \cup_{h\geq 2} T_h$. Then there is no $E_j \in \cup_{h\geq 2} T_h$ with $E_i \subset E_j$, and so $E_i$ is a $(-1)$-curve in $X$. Moreover $E_i$ intersects each T-chain $C_k$ either transversally at one point, or at none.
\label{nocycles}
\end{lemma}
    
\begin{proof}
Let $C_k$ be a T-chain such that $E_i \cdot C_k > 0$. As there are no cycles in $G$, we have that $E_i \cdot C_k=1$. This means that there is a $(-1)$-curve in $E_i$ intersecting $C_k$ transversally at one point. As $K_W$ is nef, there is at least one more $C_{k'}$ such that $E_i \cdot C_{k'}=1$ as well.

Let $E_j \in \cup_{h\geq 2} T_h$ with $E_i \subset E_j$. Then $E_j=E_i+D$ where $D \cdot C_k \geq 0$ by definition of $T_h$. Therefore $E_j \cdot C_k \geq E_i \cdot C_k=1$. Same for $C_{k'}$, but then we obtain a cycle between the vertices in $G$ corresponding to $E_i$, $E_j$, $C_k$ and $C_{k'}$. Therefore there is no such $E_j$.
\end{proof}

\begin{theorem}
Assume that in every connected component of $G$ we have at least one T-chain with a curve in $J$. Then,
$$K_S \cdot \pi(C) \leq K_W^2-K_S^2 + \text{min}\{|J|,l\}.$$
Equality holds if and only if $X=S$. \label{ineq}
\end{theorem}

\begin{proof}
If $l < |J|$, then this is \cite[Corollary 3.13]{FRU23} ($K_W$ big and nef works again in this case). Let us assume $l \geq |J|$. Via Corollary \ref{euler}, we have $$K_W^2-K_S^2 - K_S \cdot \pi(C) +|J| = \alpha + \Big(|J|-l_c \Big) + \Big(\sum_{h\geq 3} (h-2) (s_h-t_h) \Big),$$ and the right side of this equality is greater than or equal to zero as $|J|\geq l_c$ by hypothesis. Hence the inequality is true.   




Assume that we have equality. Then there are no cycles, $J=l_c$, and $s_h=t_h$ for all $h \geq 3$. As there are no cycles, then $\cup_{h\geq 2} T_h$ consists of only $(-1)$-curves in $X$ by Lemma \ref{nocycles}. This together with $s_h=t_h$ for $h\geq 3$ says that an arbitrary $E_i$ is either a $(-1)$-curve with $E_i \cdot C \geq 2$ (by Lemma \ref{nocycles}, we have that $E_i$ intersects distinct T-chains transversally at one point each), or it satisfies $E_i \cdot C=1$. This last ones are classified in \cite[Section 4]{FRU23} (summarized in Tables 1,2,3 in \cite{FRU23}). In particular, as we have the constraint $s_h=t_h$ for $h>2$ and no cycles, we have that the only possibilities for $E_i$ with $E_i \cdot C=1$ are of type T.2.1 and T.2.2 (notation as in \cite[Table 1]{FRU23}). This is because for all the other types we have some $E_i$ in $S_h \setminus \{ T_h \}$ with $h>2$, or a cycle.

Let us consider the configuration of curves in one connected component $G'$ of $G$. We want to prove that it consists of just one vertex, and so $X=S$. As it is a tree, if $G'$ has two or more vertices, then it has two or more leaves. For a leaf we have that the extreme vertex must be a T-chain. It is followed by a vertex $E_i$. If this vertex represents a $(-1)$-curve with $E_i \cdot C \geq 2$, then that T-chain must give a curve in $J$, just because the shape of T-chains in general. On the other hand, if this vertex represents a $E_i$ with $E_i \cdot C=1$, then it must be of type T.2.1 and T.2.2, and so it must give a curve in $J$ by the same reason as before. Therefore, at least we have two curves in $J$ if $G'$ has two or more vertices, but we have only one. Thus $G'$ consists of one vertex.  
\end{proof}

\begin{remark}
If $|J|=\text{min}\{|J|,l\}$ and equality holds, then the T-chains are of type $[x,2,\ldots,2]$. 
\end{remark}

\begin{proposition}
Let $G'$ be a connected component of $G$ that has only curves in $J^c$. Then $S$ must be a K3 surface or an Enriques surface.   \label{koda}
\end{proposition}

\begin{proof}
By \cite[Proposition 3.1]{RU22} we have that $S$ could be rational, K3, Enriques, or the Kodaira dimension of $S$ is $1$ or $2$. Let $L$ be the configuration of curves in $X$ corresponding to $G'$. Then $\pi^*(\pi(L))_{\text{red}}=L$ as it considers all the exceptional divisors over $\pi(L)$. In particular, the images of configurations of curves of connected components of $G$ are disconnected. We now go case by case.  

If $S$ is rational, then it is a Hirzebruch surface $\F_m$ for some $m\neq 1$ or $\P^2$. It is easy to verify that $\pi(L)$ can only be the $(-2)$-curve in $\F_2$. But this would mean that $L$ is an M-resolution of a Du Val singularity, and so there are no blow-ups, and no T-chain, a contradiction.

Same trick works for $S$ of general type, since we know that a connected configuration of curves whose intersection with $K_S$ is equal to zero must be a Du Val configuration.

If $S$ has Kodaira dimension $1$, then there is an elliptic fibration $S \to \P^1$. This is because T-chains cannot be only contained in fibers, as $K_W$ is big. Therefore, $\pi(L)$ must be part of a fiber. But there must be a curve in $J$ for some T-chain, as $K_W$ is big. Since the images of the configurations of connected components is disconnected, then $\pi(L)$ is a Du Val configuration, which is a contradiction as they do not have non-trivial M-resolutions.  
\end{proof}

A direct consequence of Theorem \ref{ineq} and Proposition \ref{koda} is the following.  

\begin{corollary}
We have $K_S \cdot \pi(C) \leq K_W^2-K_S^2 + \text{min}\{|J|,l\}$.
The equality holds if and only if $X=S$.
\label{ineqq}
\end{corollary}


\begin{remark}
Our assumption on the canonical class $K_W$ is big and nef. There is a unique canonical model $W_{can}$ of $W$ and a contraction morphism $\rho \colon W \to W_{can}$ such that $K_{W_{can}}$ is ample, $W_{can}$ has only T-singularities (including Du Val singularities), and $K_W^2=K_{W_{can}}^2$. It turns out that the pre-image of a Du Val singularity is its minimal resolution (ADE configuration), and over a non-Du Val T-singularity we have partial resolutions called \textit{M-resolutions} \cite{BC94}. Therefore, without loss of generality, we may assume from the beginning that all of these partial M-resolutions have been contracted to $W$, and so $\rho$ only contracts ADE configurations disjoint from the T-chains. Hence, we have Corollary \ref{ineq} with the smallest optimal $l$.     
\label{canonical}
\end{remark}

It turns out that the inequality in Corollary \ref{ineq} becomes a strong restriction when the Kodaira dimension of $S$ is $1$ and $K_W^2$ is small.

\begin{corollary}
Let $K_W^2 < 3p_g(W)-6$. Assume that $S \to \P^1$ is an elliptic fibration. Then $\pi(C)$ contains at most one double section, two sections, or one section, and no other multisections.
\label{low}
\end{corollary}

\begin{proof}
Let $p_g=p_g(S)=p_g(W)$. We must have $p_g\geq 2$ because $K_W$ is big and nef. If $p_g=2$, then $K_W^2=1$, and the proof will be contained in the proof of Theorem \ref{main5}. Assume $p_g\geq 3$. Let $F$ be the class of a fiber of $S \to \P^1$. we have that $K_S \equiv (p_g-1+ \sum_{i=1}^{p} \frac{\mu_i-1}{\mu_i})F$ where $\mu_i$ are the multiplicities of all multiple fibers of $S \to \P^1$. For the sake of notational convenience, we write $\Sigma=\sum_{i=1}^{p} \frac{\mu_i-1}{\mu_i}$. (It may happen that $\Sigma=0$, which means no multiple fibers.) Let $d_i$ be the number of curves $\Gamma$ in $\pi(C)$ such that $\Gamma \cdot F=i$. Then $\sum_{i\geq 1} i d_i = \pi(C) \cdot F$. Therefore, we evaluate the inequality in Corollary \ref{ineq} to obtain $(\sum_{i\geq 1} i d_i)(p_g-1+\Sigma) \leq K_W^2 + |J|$. Let $K_W^2=2p_g-4+k$ for some $k \in \Z$. As $p_g \geq 3$, we have $$|J|-2 + \frac{(p_g-1+\Sigma)(\sum_{i\geq 1} (i-1) d_i)}{p_g-2+\Sigma}  \leq \frac{k - 2 \Sigma}{p_g-2+\Sigma}.$$

If $|J| \geq 3$ or $|J|=2$ and $\sum_{i\geq 1} (i-1) d_i >0$, then $1 \leq \frac{k - 2 \Sigma}{p_g-2+\Sigma}$, and so $K_W^2 \geq 3p_g(S)-6$. But $K_W^2 < 3p_g(S)-6$. Hence, either $|J|=2$ and $\sum_{i\geq 1} (i-1) d_i=0$, which implies two sections, or $|J|=1$. Say that $|J|=1$ and $\sum_{i\geq 1} (i-1) d_i>0$. If $\sum_{i\geq 1} (i-1) d_i\geq 2$, then $K_W^2 \geq 3p_g-4+3\Sigma$. But $K_W^2 < 3p_g(S)-6$. Hence $\sum_{i\geq 1} (i-1) d_i=1$ and we have only a double section. 
\end{proof}

\begin{corollary}
If $K_W^2=2p_g(S)-4$, then $S$ has Kodaira dimension $1$ and there is an elliptic fibration $S \to \P^1$ such that $\pi(C)$ contains two sections or one section, and no other multisections. Moreover, if it contains two sections, then $S=X$ and $S \to W$ is the contraction of two Wahl chains of type $[p_g(S)+1,2,\ldots,2]$. If $K_W^2< 2p_g(S)-4$, then we must have that $\pi(C)$ contains exactly one section and no other multisections. 
\label{hori}
\end{corollary}

\begin{proof} 
By \cite[Proposition 3.1]{RU22} and because $p_g(W)\geq 3$, the Kodaira dimension of $S$ is equal to $1$ and there must be an elliptic fibration $S\to \P^1$. If $\pi(C)$ contains only a double section, then $2(p_g-1+\Sigma)\leq 2p_g-4+1$, a contradiction. If it contains two sections, then $\Sigma=0$ and we have equality. Therefore $X=S$. If $K_W^2< 2p_g(S)-4$, then $|J|$ cannot be bigger than $2$ by $$|J|-2 + \frac{(p_g-1+\Sigma)(\sum_{i\geq 1} (i-1) d_i)}{p_g-2+\Sigma}  \leq \frac{k - 2 \Sigma}{p_g-2+\Sigma}.$$ Therefore $|J|=1$ and so $\pi(C)$ contains exactly one section. 
\end{proof}

\section{Small surfaces} \label{s2}

By Corollary \ref{hori}, a classification of Horikawa surfaces with only non-Du Val T-singularities requires to understand the case when $\pi(C)$ contains exactly one section and no other multisections. It turns out that this particular geometric situation appears in various other contexts. In this section we will classify it completely.

\begin{definition}
We say that $W$ is a \textit{small surface} if $S$ is an elliptic surface with Kodaira dimension $1$, and $\pi(C)$ contains precisely one section and components of singular fibers.
\label{small}
\end{definition}

In particular $q(W)=q(S)=0$ and $p_g(W)=p_g(S) \geq 2$.

\begin{theorem}
Let $W$ be a small surface, and let $N$ be the number of fibers completely contained in $\pi(C)$. Then $l \geq \text{max} \{1,N-1\}$ and $K_W^2=p_g(S)-2+N$.
\label{formulaK^2}
\end{theorem}

\begin{proof}
The inequality for $l$ is trivial if $N\leq 2$. Let $N \geq 3$. We note that we have a T-chain from the section which may use at most two complete fibers. Each one of the $N-2$ resting fibers must be used to construct at least one extra T-chain. Therefore $l \geq 1+N-2=N-1$. We will now use the relation between log Chern numbers of the configuration $\pi(C)$, and $K_W^2$ and $l$ (see \cite[Section 3]{RU22}).

The configuration $\pi(C)$ is formed by $1$ section, $N$ complete singular fibers, and $N'$ parts of singular fibers. Let $N^*$ be the total number of fibers of type $I_n^*$ (for various $n$s), $II^*$, $III^*$ and $IV^*$. Since singular fibers may be cuspidal (type II) or self-tangent (type III), we blow-up to obtain a configuration with only simple singularities. Say that $\pi(C)$ contains $c$ type II fibers, and $a$ type III fibers. Let us consider $Z \to S$ which is the blow-up at the cusps of type II fibers twice, and at the tangency of type III fibers once. Let $D$ be the configuration of $\rho$ curves in $Z$ given by the total transform of $\pi(C)$. Note that $K_Z^2=-2c-a$, and $D$ has only double points and triple points. The number of triple points is $a+b+c$, where $b$ is the number of type IV fibers in $\pi(C)$.

The number of double points in $D$ is $N-N^*+\rho-1-3(a+b+c)$. Let $\Gamma$ be the section in $Z$. We have that $\Gamma^2=-(p_g(S)+1)$. We compute the log Chern numbers of $(Z,D)$: (see \cite[Proposition 3.3]{RU22}) $$\bar{c}_1^2= p_g(S)-3+2(N-N^*)-(a+b+c)$$  and $ \bar{c}_2=\chi_{\text{top}}(Z)-1+N-N^*-\rho-(a+b+c)$.

One can check that $\bar{c}_1^2$ and $\bar{c}_2$ do not change if we subtract the $(-1)$-curves over the type II and/or type III fibers. This may happen when we do not use those curves in the final $C$.

On the other hand, by \cite[Theorem 3.8]{RU22}, we have that $$\bar{c}_1^2=K_Z^2-l+\rho+2t_3^0 \ \ \ \text{and} \ \ \ \bar{c}_2=K_Z^2+\chi_{\text{top}}(Z)-l-K_W^2+ t_3^0,$$ where $t_3^0$ counts the number of $(-1)$-curves over triple points of $D$ which are not part of the final $C$. 

Then, we have $p_g(S)-3+2(N-N^*)-(a+b+c)=K_Z^2-l+\rho+2t_3^0$ and $\chi_{\text{top}}(Z)-1+N-N^*-\rho-(a+b+c)=K_Z^2+\chi_{\text{top}}(Z)-l-K_W^2+ t_3^0$, and so $$K_W^2-N-p_g(S)+2+N^*=t_3^0.$$ But now a quick check over the triple points of $D$ shows that all of their $(-1)$-curves must be used in $C$. (Otherwise, we would have M-resolutions over one $(-2)$-curve with Wahl singularities, which is not possible.) Thus, $t_3^0=0$. We have $N^*=0$ by Lemma \ref{no*}.
\end{proof}

\begin{remark}
Following the same proof of Theorem \ref{formulaK^2}, there is also a formula when we allow more sections or multiple sections in $\pi(C)$, but in that case the geometry is richer and so it is more complicated to describe. Let us assume that
$\pi(C)$ be formed by $e$ disjoint sections, and fibers and/or part of fibers. Let $N$, $N'$, $N^*$, and $t_3^0$ be as in the previous proof. Let $e'$ be the intersection number between the $e$ sections and the $N'$ partial fibers (so $N'\leq e'\leq eN'$). Then one can prove that $$K_W^2=(N+p_g(S)-2)e+e'-N'-N^*+t_3^0.$$ Therefore we can make $K_W^2/p_g(S)$ big via increasing the number of sections. However, the log BMY inequality (see \cite{L03}) tells us that there is an upper bound $$ K_W^2 \leq 9(1+p_g(S)) - \frac{3}{4} \sum_{i=1}^{\ell} \Big(d_i - \frac{1}{d_i n_i^2} \Big).$$
In Theorem \ref{geogrSmall}, we show how high $K_W^2$ can be for small surfaces.
\label{manySections}
\end{remark}

\begin{corollary}
If $K_W^2=2p_g(S)-4$, then either $X=S$ and $l=2$, or $W$ is a small surface and $l \geq \text{max} \{1,p_g(S)-3\}$. 
\end{corollary}

\begin{proof}
We have $K_W^2=2p_g-4=p_g-2+N$, and so $N=p_g-2$.
\end{proof}

From now on, the purpose will be to classify all small surfaces. This will use as a key ingredient P-resolutions of particular cyclic quotient singularities.

\begin{definition} 
Let $P \in \overline{W}$ be a c.q.s. A \textit{P-resolution} of $P \in \overline{W}$ is a partial resolution $f \colon W \to \overline{W}$ over $P$ such that $W$ has only T-singularities and $K_{W}$ is ample relative to $f$. 
\label{pres}
\end{definition}

The idea is the following. Let $W$ be a small surface. The configuration of rational curves $\pi(C) \subset S$ consists of one section $\Gamma$, some complete singular fibers, and some partial singular fibers. The composition of blow-ups $\pi \colon X \to S$ contains some P-resolutions of particular c.q.s. coming from the T-chain that contains the proper transform of $\Gamma$, and T-chains from fibers. Therefore the strategy has two parts. First we analyze all possible P-resolutions of these relevant c.q.s., and then we glue them to create all possible configurations of T-chains and exceptional divisors for any small surface. In \ref{app}, we classify all of what we need from these particular c.q.s. As an important application, we start with a lemma which reduces the possibilities of complete fibers in $\pi(C)$. We recall that $\pi(C)$ is connected by Proposition \ref{koda}. We use Kodaira notation for singular fibers of (relatively minimal) elliptic fibrations.    

\begin{lemma}
Let $W$ be a small surface. Then the configuration of curves $\pi(C)$ does not contain complete fibers of Kodaira types $I_n^*$, $II^*$, $III^*$, or $IV^*$.
\label{no*} 
\end{lemma}

\begin{proof}
Let $F$ be a complete fiber in $\pi(C)$ of type $I_n^*$, $II^*$, $III^*$, or $IV^*$. We note that the section $\Gamma$ must intersect one of the ending components of $F$. As $F$ has a particular (not chain) form, the pre-image $\pi^{-1}(F)$ must contain a P-resolution of the c.q.s. corresponding to either $[2,\ldots,2,3,2,\ldots,2]$ (where one side may not contain any $2$s), or $[2,\ldots,2,a,2,\ldots,2]$ (where both sides contain $2$s and $a>2$). It is key to note that the image of the curves in the non-Du-Val T-chains involved in these P-resolutions must fill up $F$. But a simple checked on all possible P-resolutions in Proposition \ref{2r2} shows that we always need some $A_n$ singularities, and so always not all the curves in $F$ are used. Therefore $\pi(C)$ cannot contain $F$.
\end{proof}

Next we fix the notation that will be used until the end of this section. Let $W$ be a small surface. Assume there are no curves $H$ connecting two T-singularities at their ends such that $H\cdot K_W=0$; otherwise, we can contract $H$ to form a new T-singularity. Let $\Gamma \subset \pi(C)$ be the section. Let $\bG$ be its proper transform in $X$. As $\bG$ must belong to a T-chain, we have three options for this T-chain:
\begin{itemize}
    \item[(0)] The T-chain is $\bG$.
    \item[(1)] The T-chain starts with $\bG$. Let $\bG'$ be the $\P^1$ next to $\bG$ in this T-chain.
    \item[(2)] The curve $\bG$ has two neighboring curves $\bG'$ and $\bG''$ in this T-chain (such that $\bG' \cdot \bG''=0$).
\end{itemize}

In case $(1)$ we have a fiber or part of a fiber $F'$ inside of $\pi(C)$. In case $(2)$ we have the same for two fibers $F'$ and $F''$. We will always express a c.q.s. by means of its H-J continued fraction.

\begin{lemma}
Assume the notation above. Consider the T-chain in $X$ that contains $\bG$. Let $r\geq p_g(S)+1$. Then we have one of the following options:
\begin{itemize}
\item[(0)] If the T-chain is $\bG$, then the T-chain is $[4]$.

\item[(1)] Assume that the T-chain starts with $\bG$. Then, either $(1.0)$ the T-chain has the form $[r,2,\ldots,2]$, or the pre-image of $\Gamma \cup F'$ consists of: 

\begin{itemize}
    \item[(1.1)] If $F'=I_1$, then a P-resolution of $[r,a,\underbrace{2,\ldots,2}_{a-4}]$, and a $(-1)$-curve connecting an end of the P-resolution with another curve in it.  
    
    \item[(1.2)] If $F'=I_{n}$ with $n\geq 2$, then a P-resolution of $[r,a,\underbrace{2,\ldots,2}_{n-2},3,$ $ \underbrace{2,\ldots,2}_{a-3}],$ and a $(-1)$-curve connecting an end of the P-resolution with another curve in it. 

    \item[(1.3)] If $F'=II$, then we have either 

    \begin{itemize}
        \item[(1.3.1)] a P-resolution of $[r,4]$ and a $(-1)$-curve tangent to the (proper transform of) $(-4)$-curve, or 

        \item[(1.3.2)] a P-resolution of $[r,5,2]$ and a $(-1)$-curve passing through the intersection of the (proper transform of) $(-5)$-curve and $(-2)$-curve, or 

        \item[(1.3.3)] a P-resolution of $[r,6,2,2]$, a $[4]$ and a $(-1)$-curve between them, or 

        \item[(1.3.4)] a P-resolution of $[r,7,\underbrace{2,\ldots,2}_{r'-2}]$, a P-resolution of $[3,r',2]$, and a $(-1)$-curve connecting the last curve of the first P-resolution with the proper transform of the curve in the second corresponding to $r'$.

    \end{itemize}

    \item[(1.4)] If $F'=III$, then we have either

    \begin{itemize}
        \item[(1.4.1)] a P-resolution of $[r,3,3]$ and a $(-1)$-curve passing through the intersection of the (proper transform of) two $(-3)$-curves, or 

        \item[(1.4.2)] a P-resolution of $[r,4,3,2]$, $[2,5]$ and a $(-1)$-curve, or 

        \item[(1.4.3)] a P-resolution of $[r,4]$, $[4]$, $(-2)$-curve, and a $(-1)$-curve, or  

        \item[(1.4.4)] a P-resolution of $[r,5,\underbrace{2,\ldots,2}_{r'-2}]$, a P-resolution of $[4,r',2]$, and a $(-1)$-curve connecting the last curve of the first P-resolution with the proper transform of the curve in the second corresponding to $r'$.  
    \end{itemize}

    \item[(1.5)] If $F'=IV$, then we have either

    \begin{itemize}
        \item[(1.5.1)] a P-resolution of $[r,3,2,3]$, a $[4]$ and a $(-1)$-curve between them, or 


        \item[(1.5.2)] a P-resolution of $[r,4,\underbrace{2,\ldots,2}_{r'-2}]$, a P-resolution of $[3,r',3]$, and a $(-1)$-curve connecting the last curve of the first P-resolution with the proper transform of the curve in the second corresponding to $r'$.
        
    \end{itemize}   
\end{itemize}

\item[(2)] Assume that $\bG$ has two neighboring curves $\bG'$ and $\bG''$ in this T-chain. Then either $F'$ is contained in $\pi(C)$ but $F''$ is not (or vice versa), or both $F'$ and $F''$ are contained in $\pi(C)$.

\begin{itemize}
    \item[(2.1)] In the first case, we have that the pre-image of $\Gamma \cup F' \cup (\pi(C) \cap F'')$ consists of what is described in $(1)$ adding a $[2,...,2]$ on the left of $[r,\ldots]$ for each P-resolution (that contains the section).

    \item[(2.2)] In the second case we have that the pre-image of $\Gamma \cup F' \cup F''$ consists of combinations of what is described in $(1)$ taking as center $r$ (the section) for each P-resolution. The parameters at the left and at the right of $[\ldots,r,\ldots]$ are independent.
\end{itemize}
\end{itemize}
\label{allforsection}
\end{lemma}

\begin{proof}
By Lemma \ref{no*} we do not have complete fibers of type $I_n^*$, $II^*$, $III^*$, or $IV^*$.

For $(1.1)$, we start blowing up the node of $F'$, and then the blow-ups at one of the two intersections points and in one direction. We do as much as we can. In this way, we obtain a chain of $\P^1$s with negative self-intersections $[r,a,\underbrace{2,\ldots,2}_{a-4}]$, and a $(-1)$-curve connecting the last $(-2)$-curve with the proper transform of the fiber. Then the pre-image of $\Gamma \cup I_1$ must be a P-resolution of the contraction of $[r,a,\underbrace{2,\ldots,2}_{a-4}]$ and a $(-1)$-curve. Any other alternative is easily checked to be impossible. Same happens for $(1.2)$.

For $(1.3)$, we need to have a blow-up at the singular point of $F'$ (type II). If nothing else happens, then we have the first possibility. Otherwise, we blow-up the tangent point of the new fiber, giving the second possibility. We now blow up the triple point in the new fiber. Let $P_1$ be the point of intersection between the $(-1)$-curve and the $(-6)$-curve in this fiber. The same for $P_2$ and the $(-3)$-curve, and $P_3$ for the $(-2)$-curve. We note that the proper transform of the $(-1)$-curve must belong to a T-chain in $X$. Therefore, over at least one of the $P_i$ we need to have no curve in the T-chain of this $(-1)$-curve. 

If that happens for $P_1$, then we have the option with $[r,7,\underbrace{2,\ldots,2}_{r'-2}]$ in the statement. If that happens for $P_2$, then we obtain a chain $[r,6,u,2]$ and $[4,2,\ldots,2]$. Then over $[4,2,\ldots,2]$ we should have a P-resolution of its contraction, but this happens only if $[4,2,\ldots,2]=[4]$. Therefore we obtain chains $[r,6,2,2]$ and $[4]$, and a P-resolution over $[r,6,2,2]$. Similarly for $P_3$, where we do not have a P-resolution with T-singularities over $[3,2,\ldots,2]$, and so we do not consider this case.

For $(1.4)$ we do a similar analysis as for $(1.4)$. But in this case we get an extra case, as the two blow-ups over the tangency of III gives the possibility of not considering the $(-1)$-curve and the $(-2)$-curve in the fiber (this does not affect $t_3^0$ as this is not considered as a triple point of $D$). For $(1.5)$ we use the same strategy, getting only two possibilities.

The alternatives for $(2)$ are obvious from the previous analysis.  
\end{proof}

The following lemma gives a restriction for (2.2) in Lemma \ref{allforsection}, which will be useful to optimize computations later.

\begin{lemma}
Let us consider the situation $F'=I_n$ and $F''=I_{n'}$ in (2.2) of Lemma \ref{allforsection}, and let $r\geq p_g+1$. Then we only need to consider P-resolutions of c.q.s. of type $[4,r,a,\underbrace{2,\ldots,2}_{a-4}]$, $[4,r,a,\underbrace{2,\ldots,2}_{n'-2},3,\underbrace{2,\ldots,2}_{a-3}]$, $[3,\underbrace{2,\ldots,2}_{n-2},3,r,a,\underbrace{2,\ldots,2}_{n'-2},3,\underbrace{2,\ldots,2}_{a-3}]$, or $[3,\underbrace{2,\ldots,2}_{n-2},3,r,a,\underbrace{2,\ldots,2}_{a-4}]$. \\
This is, there is no termination with $2$s in both sides. 
\label{In+r+In'}
\end{lemma}

\begin{proof}
Let us assume that there are ending $2$s in both sides. The section $\bG$ must belong to a T-chain in $X$ which ends with $2$s. Let us say that these $2$s are over the $I_{n'}$, i.e. over the right side of $r$. On the left side of $r$ we can have for the c.q.s. the chain (Case 1) $[\underbrace{2,\ldots,2}_{b-4},b,r,\ldots]$ if $n=1$, or (Case 2) $[\underbrace{2,\ldots,2}_{b-3},3,\underbrace{2,\ldots,2}_{n-2},b,r,\ldots]$ for $n\ge 2$. In both cases the key point will be to look at the centers (see after Proposition \ref{T-chain})  of the T-chains at the left of the T-chain of $\bG$.

(Case 1): We are assuming existence of $2$s on the left, so $b\geq 5$. Thus $2,\ldots,2,b,r$ cannot be part of the T-chain of $\bG$. So, there is another T-chain on the left of the T-chain of $\bG$. By \cite[Corollary 2.4]{FRU23}, after we blow-down all possible $(-1)$-curves  in the corresponding P-resolution of the c.q.s., we cannot blow-down any curve in the center. In addition, for the left-most T-chain, there must be a curve in the center with self-intersection $\leq -3$. Therefore the only possibility for center is $b$. But then we must blow-up between $b$ and $r$, and as we have $b-4$ $2$s on the left of $b$, we must blow-up more than one time, producing a $2$ on the left of the T-chain of $\bG$, which is a contradiction.     

(Case 2): This is similar to (Case 1). The left-most center could be $3$ or $b$, but the second case does not work as before. If it is $3$, then one can check that we must reach eventually $b$ as a center, and just as before we get a contraction and we need more than one blow-up between $b$ and $r$. (We need to pass through potential $2$ centers between $3$ and $b$, but always leaving a $2$ for the corresponding T-chain on the left.) 
\end{proof}

We now consider T-chains in $X$ which connect with $\bG$ via a $(-1)$-curve $\bE$. This T-chain must be inside of a fiber of $X \to S \to \P^1$, and so it is over a fiber $F$ of $S\to \P^1$. Therefore $F$ or part of $F$ is in $\pi(C)$.

\begin{lemma}
In the situation above, we have that $F$ is contained in $\pi(C)$. Let $b\geq 0$. Moreover, we have one of the following options for the pre-image of $F$ minus $\bE$: 

\begin{itemize}
    \item[(3.1)] If $F=I_1$, then a P-resolution of $[\underbrace{2,\ldots,2}_{b},a,\underbrace{2,\ldots,2}_{a-5}]$, and a $(-1)$-curve connecting an end of the P-resolution with another curve in it.  
    
    \item[(3.2)] If $F=I_{n\geq 2}$, then a P-resolution of $[\underbrace{2,\ldots,2}_{b},a,\underbrace{2,\ldots,2}_{n-2},3,\underbrace{2,\ldots,2}_{a-4}]$, and a $(-1)$-curve connecting an end of the P-resolution with another curve in it. 

    \item[(3.3)] If $F=II$, then we have either 

    \begin{itemize}
        \item[(3.3.1)] a P-resolution of $[\underbrace{2,\ldots,2}_{b},5]$ and a $(-1)$-curve tangent to the (proper transform of) $(-5)$-curve, or 

        \item[(3.3.2)] a P-resolution of $[\underbrace{2,\ldots,2}_{b},7,2,2]$, a $[4]$ and a $(-1)$-curve between them, or 

        \item[(3.3.3)] a P-resolution of $[\underbrace{2,\ldots,2}_{b},6,2]$ and a $(-1)$-curve passing through the intersection of the (proper transform of) $(-6)$-curve and $(-2)$-curve, or 

        \item[(3.3.4)] a P-resolution of $[\underbrace{2,\ldots,2}_{b},8,\underbrace{2,\ldots,2}_{r'-2}]$, a P-resolution of $[3,r',2]$, and a $(-1)$-curve connecting the last curve of the first P-resolution with the proper transform of the curve in the second corresponding to $r'$.
    \end{itemize}
    
    \item[(3.4)] If $F=III$, then we have either 

    \begin{itemize}
        \item[(3.4.1)] a P-resolution of $[\underbrace{2,\ldots,2}_{b},4,3]$ and a $(-1)$-curve passing through the intersection of the (proper transform of) $(-4)$-curve and $(-3)$-curve, or 
        
        \item[(3.4.2)] a P-resolution of $[\underbrace{2,\ldots,2}_{b},5,3,2]$, $[2,5]$ and a $(-1)$-curve, or 
        
        \item[(3.4.3)] a P-resolution of $[\underbrace{2,\ldots,2}_{b},5]$, $[4]$, $(-2)$-curve, and a $(-1)$-curve, or
        
        \item[(3.4.4)] a P-resolution of $[\underbrace{2,\ldots,2}_{b},6,\underbrace{2,\ldots,2}_{r'-2}]$, a P-resolution of $[4,r',2]$, and a $(-1)$-curve connecting the last curve of the first P-resolution with the proper transform of the curve in the second corresponding to $r'$. 
    \end{itemize}
    
    \item[(3.5)] If $F=IV$, then we have either

    \begin{itemize}
        \item[(3.5.1)] a P-resolution of $[\underbrace{2,\ldots,2}_{b},4,2,3]$, a $[4]$ and a $(-1)$-curve between them, or


        \item[(3.5.2)] a P-resolution of $[\underbrace{2,\ldots,2}_{b},5,\underbrace{2,\ldots,2}_{r'-2}]$, a P-resolution of $[3,r',3]$, and a $(-1)$-curve connecting the last curve of the first P-resolution with the proper transform of the curve in the second corresponding to $r'$.        
    \end{itemize}   
\end{itemize}
\label{allfiber}
\end{lemma}

\begin{proof}
The proof uses the same analysis as in the proof of Lemma \ref{allforsection}.    
\end{proof}

\subsection{List of building blocks for small surfaces} \label{blocks}

We produce a list of geometrical realizations for the options in Lemma \ref{allforsection} and Lemma \ref{allfiber}. They are described in $X$ starting with a configuration in $S$. Each example has a ``local $K^2$ over that configuration" which will be used to compute the global $K_W^2$ later. We also compute the discrepancy $d(\bG)$ of $\bG \subset X$ in its corresponding T-chain. After this list of examples, we will prove that there are no other realizations from Lemma \ref{allforsection} and Lemma \ref{allfiber}, and so this is the list.

\vspace{0.3cm}

\noindent
\textbf{(S0F)} This is from $(0)$ and $(1)$ in Lemma \ref{allforsection}. We use a section $\Gamma$ such that $\overline{\Gamma}^2=-r \leq -4$ together with $(r-4)$ $(-2)$-curves from some singular fiber. We obtain the Wahl chain $[r,2,\ldots,2]$. We have $K^2=0+r-3=r-3$ and $d(\bG)=-\frac{r-3}{r-2}$.

\vspace{0.4cm}

\noindent
\textbf{(S1F.1)} This is from $(1.1)$ and $(1.3)$ in Lemma \ref{allforsection}. We use a section $\Gamma$ such that $\overline{\Gamma}^2=-3$ and a fiber $F$ of one of the following types: $II$, $III$, or $I_n$ with $n\geq 1$. When $n=1$ or $F=II$, we obtain the Wahl chain $[3,5,2]$. If $n$ increases, we have the P-resolution
$[3,5,2]-(1)-[4,\underbrace{2,\ldots,2}_{n-3},3,2]$
over $[3,4,2,\ldots,2,3,2]$ with $(n-2)$ $2$s in the middle, as shown in Figure \ref{fS1F.1}. When $n=2$, the chain on the right is the Wahl chain $[5,2]$. If $F=III$, we obtain the Wahl chain $[3,5,2]$ connected to the T-chain $[2,3,4]$ by a $(-1)$-curve. The discrepancies of the curves attached to this $(-1)$-curve are $-\frac{2}{5}$ and $-\frac{2}{3}$ ($\frac{2}{5}+\frac{2}{3}=\frac{16}{15}>1$). In all cases $K^2=1$ (e.g. $K^2=-4+3+(n-(n-1)+1)=1$) and $d(\overline{\Gamma})=-\frac{3}{5}$.


\begin{figure}[htbp]
\centering
\vspace*{-1.5em}
\includegraphics[width=10cm]{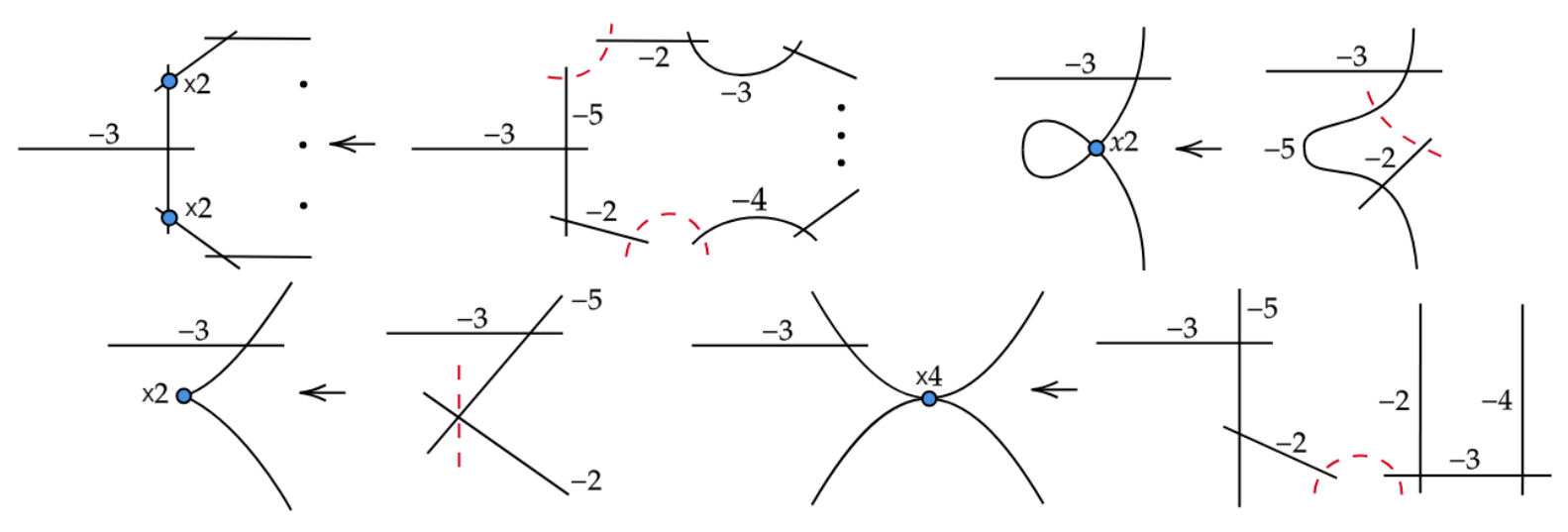}

\vspace*{-1.6em}
\caption{Configurations (S1F.1).}
\label{fS1F.1}
\end{figure}

\noindent
\textbf{(S1F.2)} This is in Lemma \ref{allforsection} $(1.2)$. We use a section $\Gamma$ such that $\overline{\Gamma}^2=-r\leq -4$ and an $I_n$ with $n\geq r-2$. When $n=r-2$, we obtain the Wahl chain $[r,r+1,\underbrace{2,\ldots,2}_{r-4},3,\underbrace{2\ldots,2}_{r-2}]$. If $n$ $\vspace*{-1em}$ increases, we have the P-resolution
$$[r,r+1,\underbrace{2,\ldots,2}_{r-4},3,\underbrace{2\ldots,2}_{r-2}]-(1)-[r+1,\underbrace{2,\ldots,2}_{n-r},3,\underbrace{2,\ldots,2}_{r-2}]$$
over $[r,r+1,\overbrace{2,\ldots,2}^{n-2},3,\overbrace{2,\ldots,2}^{r-2}]$, as shown in Figure \ref{fS1F.2}. When $n=r-1$, the chain on the right is the Wahl chain $[r+2,\underbrace{2,\ldots,2}_{r-2}]$. In all cases $K^2=-2(r-1)+(2r-3)+(n-(n-r+2)+1)=r-2$ and $d(\bG)=-\frac{r^2-2r}{r^2-r-1}$.

\begin{figure}[htbp]
\centering
\vspace*{-1.85em}
\includegraphics[width=8.8cm]{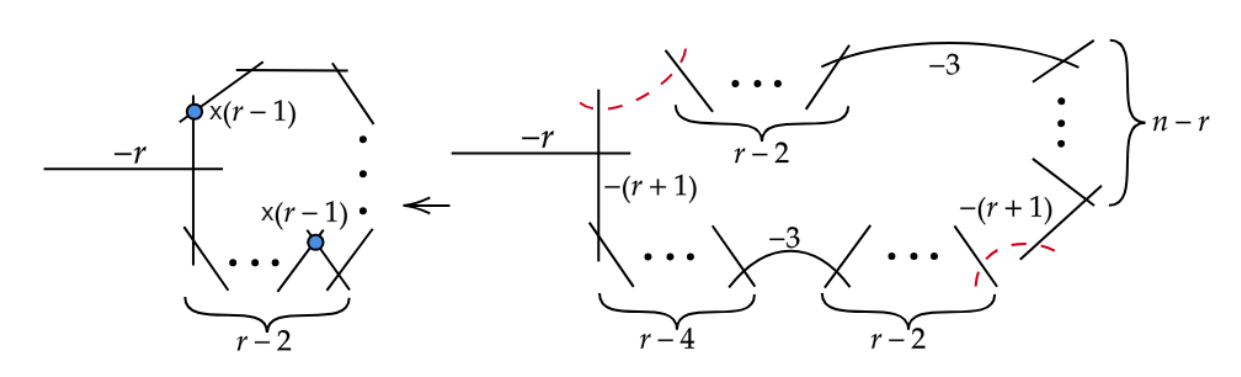}
\vspace{-1.4em}
\caption{Configuration (S1F.2).}
\label{fS1F.2}
\end{figure}

\noindent
\textbf{(S1F.3)} This is in Lemma \ref{allforsection} $(2.1)$. We use two $(-2)$-curves followed by a section $\Gamma$ such that $\overline{\Gamma}^2=-5$, and a fiber $F$ of one of the following types: $II$, $III$, $IV$, or $I_n$ with $n\geq 1$. If $F=III$, we obtain the Wahl chain $[2,2,5,4]$ connected to the chains $[4]$ and $[2]$ by a $(-1)$-curve. If $F=IV$, we obtain the Wahl chain $[2,2,5,4]$ connected to the T-chain $[3,2,3]$ by a $(-1)$-curve. The discrepancies of the curves attached to these $(-1)$-curves are ($-\frac{5}{7}$, $-\frac{1}{2}$, and $0$), and $(-\frac{5}{7}$ and $-\frac{1}{2}$), respectively. Note that ($\frac{5}{7}+\frac{1}{2}=\frac{17}{14}>1)$. When $n=1$ or $F=II$, we obtain the Wahl chain $[2,2,5,4]$. If $n$ increases, we have the P-resolution
$[2,2,5,4]-(1)-[3,\underbrace{2,\ldots,2}_{n-3},3]$
over $[2,2,5,3,2,\ldots,2,3]$ with $(n-2)$ $2$s in the middle, see Figure \ref{fS1F.3}. When $n=2$, the chain on the right is the Wahl chain $[4]$. In all cases $K^2=3$ (e.g. 
$K^2=-2+4+(n-1-(n-1)+1)=3$) and $d(\overline{\Gamma})=-\frac{6}{7}$.
\begin{figure}[htbp]
\centering
\vspace*{-1.5em}
\includegraphics[width=9.3cm]{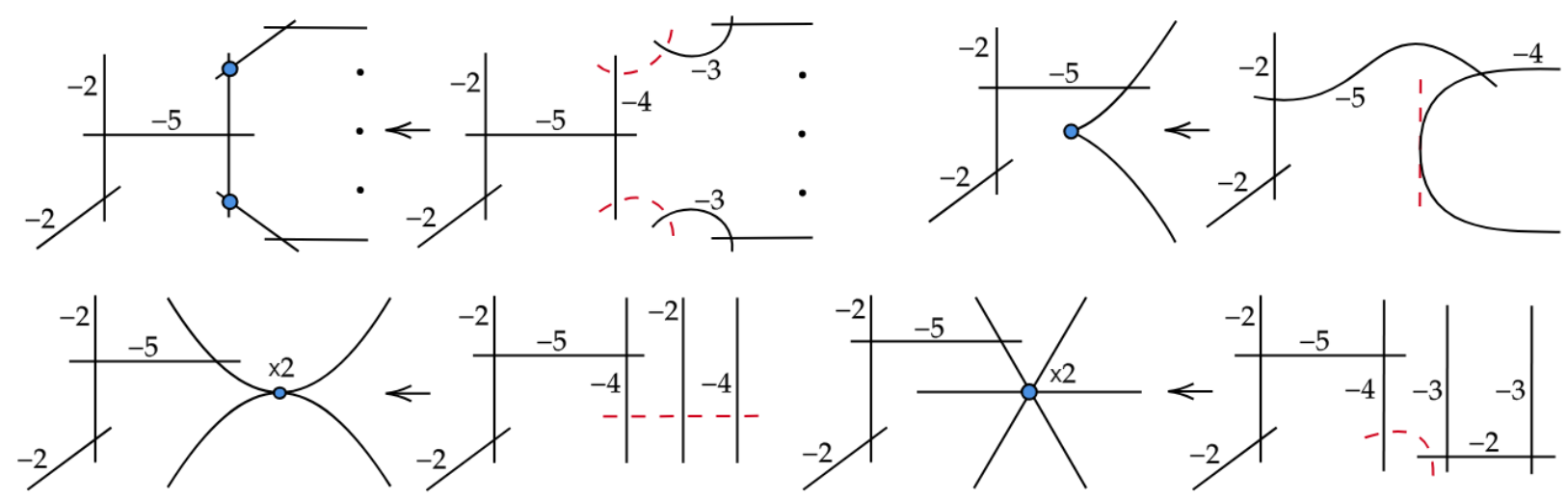} 
\vspace*{-1.2em}
\caption{Configurations (S1F.3).}
\label{fS1F.3}
\end{figure}

\vspace{0.1cm}

\noindent
\textbf{(S1F.4)} This is in Lemma \ref{allforsection} $(2.1)$. We use one $(-2)$-curve followed by a section $\Gamma$ such that $\overline{\Gamma}^2=-r$, and a fiber F of one of the following types: $II$, $III$, or $I_n$ with $n\geq r-2$. When $n=r-2$, we have the T-chain $[2,r,3,2,\ldots,2,3]$ with $(r-4)$ $2$ in the middle. If $n$ increases, then we have the P-resolution $[2,r,3,\underbrace{2,\ldots,2}_{r-4},3]-(1)- [3,\underbrace{2,\ldots,2}_{n-r},3]$
over $[2,r,3,2,\ldots,2,3]$ with $(n-2)$ $2$s in the middle, as in Figure \ref{fS1F.4}. When $n=r-1$, the chain on the right is the Wahl chain $[4]$. If $F=III$, then $\overline{\Gamma}^2=-4$ and we have the same P-resolution as when $F$ is a $I_2$. If $F=IV$ and $\overline{\Gamma}^{2}$ is either $-4$ or $-5$, then we have the same P-resolution as when $F$ is a $I_3$ but the T-chain $[2,4,3,3]$ is connected to another Wahl chain $[4]$ by a $(1)$-curve, and the T-chain $[2,5,3,2,3]$ is connected in the middle to another Wahl chain $[4]$ by a $(-1)$-curve, respectively. The discrepancies of the curves attached to these $(-1)$-curves are ($-\frac{3}{5}$ and $-\frac{1}{2}$) for the first case ($\frac{3}{5}+\frac{1}{2}=\frac{11}{10}>1$), and $(-\frac{5}{7}$ and $-\frac{1}{2})$ for the second one $(\frac{5}{7}+\frac{1}{2}=\frac{17}{14}>1)$.
In all cases $K^2=r-2$ (e.g. $K^2=-2+(r-2+1)+(n-r+2-(n-r+2)+1)=r-2$) and $d(\bG)=-\frac{2r-4}{2r-3}$.
\begin{figure}[htbp]
\centering
\vspace*{-1.4em}
\includegraphics[width=9.3cm]{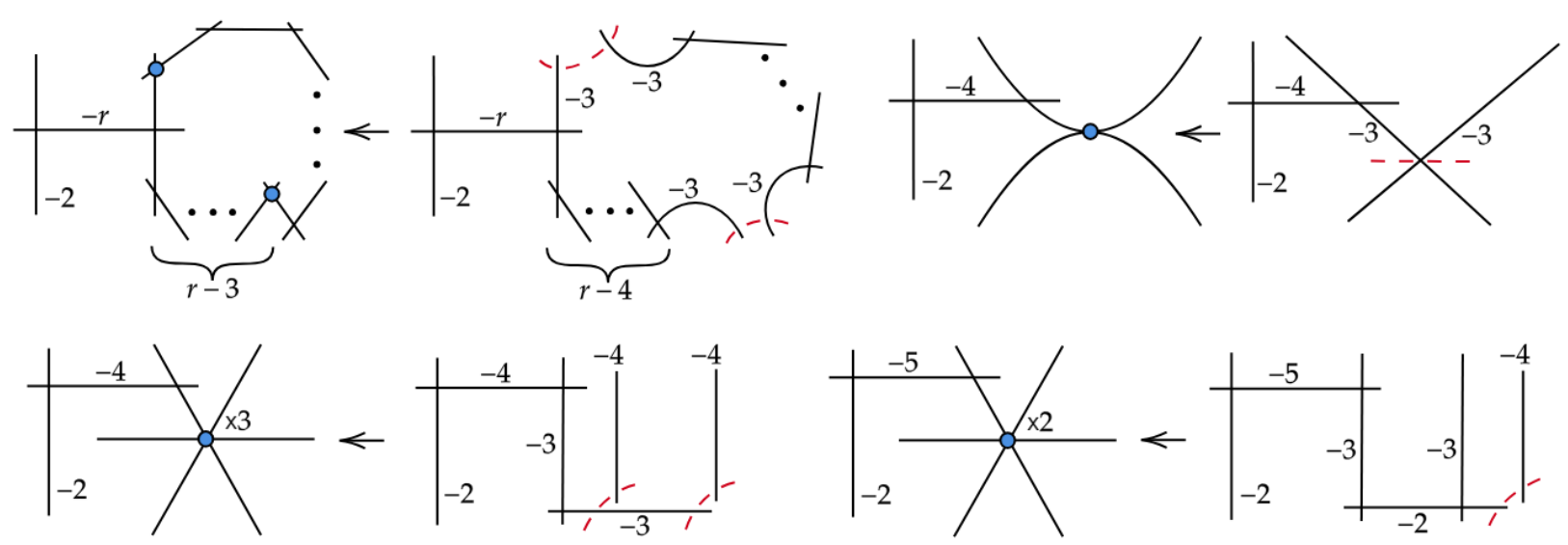}
\vspace*{-1.1em}
\caption{Configurations (S1F.4).}
\label{fS1F.4}
\end{figure}

\noindent
\textbf{(S2F.1)} This is in Lemma \ref{allforsection} $(2.2)$. We use a fiber $F$ of one the following types: $III$, $IV$, or $I_n$ with $n\geq 2$ followed by a section $\Gamma$ such that $\overline{\Gamma}^2=-4$, and a fiber $I_{n'}$ with  $n'\geq 3$. When ($n=2$ or $F=III$) and $n'=3$, we have the P-resolution $[3,3,5,3,2]-(1)-[3,6,2,3,2] \to [3,3,4,4,2,3,2].$ If $n$ and/or $n'$ increases, then we have the P-resolution $$[3,\underbrace{2,\ldots,2}_{n-4},3]-(1)-[3,3,5,3,2]-(1)-[3,6,2,3,2]-(1)-[4,\underbrace{2,\ldots,2}_{n'-5},3,2]$$ over $[3,\overbrace{2,\ldots,2}^{n-2},3,4,4,\overbrace{2,\ldots,2}^{n'-2},3,2]$, as in Figure \ref{fS2F.1}. If $F=IV$, then we have the same P-resolution as when $F=I_2$ but the T-chain $[3,3,5,3,2]$ is connected to a $[4]$ by a $(-1)$-curve. In all cases $K^2=3$. In the general case one computes $K^2=-10+(n-(n)+1)+(5)+(5)+(n'-(n'-1)+1)=3$. We have $d(\bG)=-\frac{12}{13}$.

\begin{figure}[htbp]
\centering
\vspace*{-1.5em}
\includegraphics[width=7.8cm]{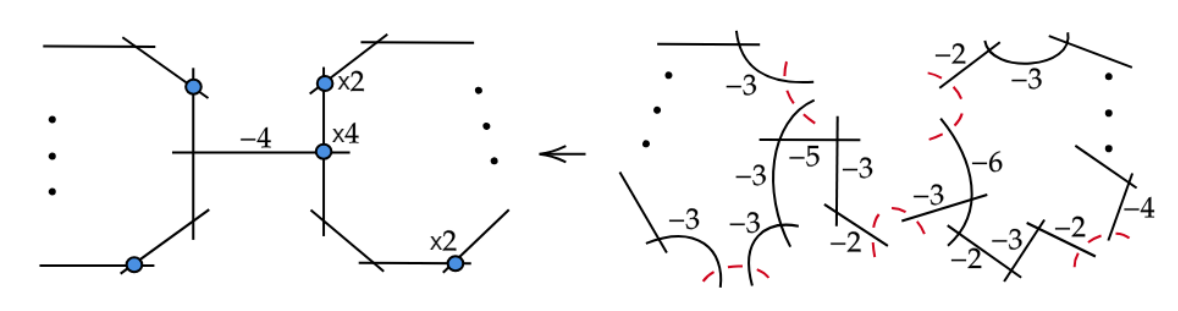}
\includegraphics[width=7.2cm]{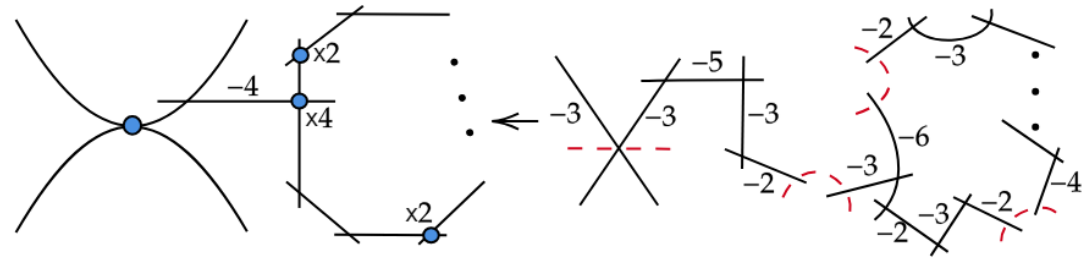} \ \ \hspace*{-1.3em}  \includegraphics[width=7.3cm]{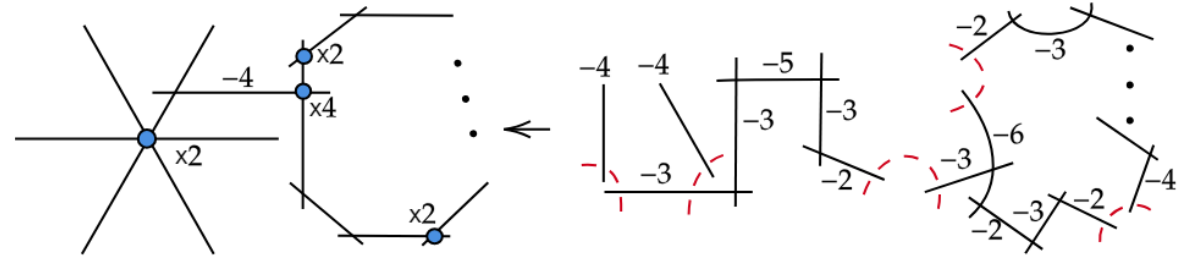}
\vspace*{-1.2em}
\caption{Configurations (S2F.1).}
\label{fS2F.1}
\end{figure}

\vspace{0.4cm}

\noindent
\textbf{(S2F.2)}
This is in Lemma \ref{allforsection} $(2.2)$. We use an $I_n$ with $n\geq 4$ followed by a section $\Gamma$ such that $\overline{\Gamma}^2=-4$, and a fiber $F$ of one of the following types: $II$, $III$, or $I_{n'}$ with $n'\geq 1$. When $n=4$ and ($n'=1$ or $F=II$), we have the P-resolution $[3,2,2,7,2]-(1)-[3,2,2,5,5,2] \to [3,2,2,3,4,5,2].$ 
If $n$ and/or $n'$ increases, then we have the P-resolution $$[3,\underbrace{2,\ldots,2}_{n-6},3]-(1)-[3,2,2,7,2]-(1)-[3,2,2,5,5,2]-(1)-[4,\underbrace{2,\ldots,2}_{n'-3},3,2]$$ over $[3,\overbrace{2,\ldots,2}^{n-2},3,4,4,\overbrace{2,\ldots,2}^{n'-2},3,2]$, as in Figure \ref{fS2F.2}. If $F=III$, then we have the same P-resolution as when $F=I_1$ but the Wahl chain $[3,2,2,5,5,2]$ is connected to a T-chain $[2,3,4]$ by a $(-1)$-curve. The discrepancies of the curves attached to this $(-1)$-curve are $-\frac{7}{16}$ and $-\frac{2}{3}$ respectively. In all cases $K^2=3$. In the general case one computes $K^2=-11+(n-4-(n-4)+1)+(5)+(6)+(n'-(n'-1)+1)=3$. We have $d(\bG)=-\frac{15}{16}$.

\begin{figure}[htbp]
\centering
\vspace*{-1em}
\includegraphics[width=9.5cm]
{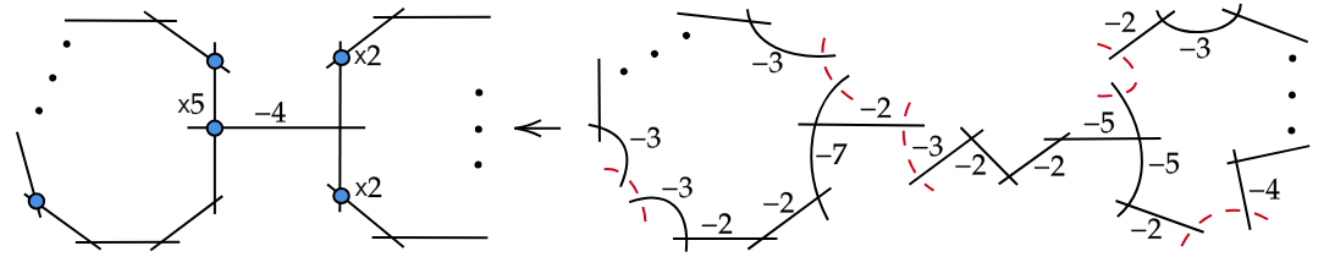}

\includegraphics[width=7.2cm]{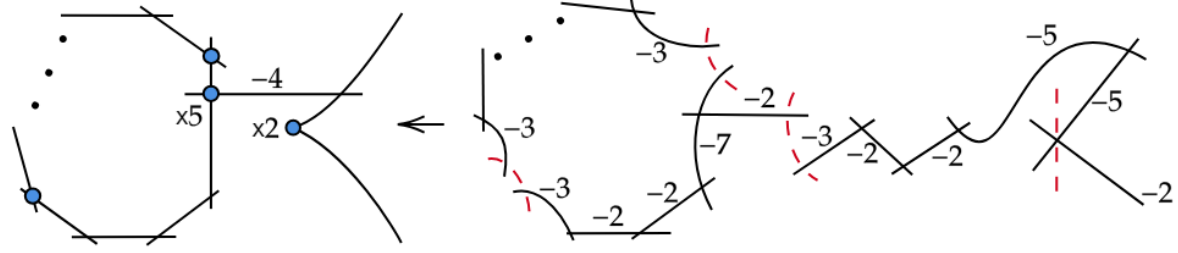} \ \ \hspace*{-0,8em}\includegraphics[width=7.4cm]{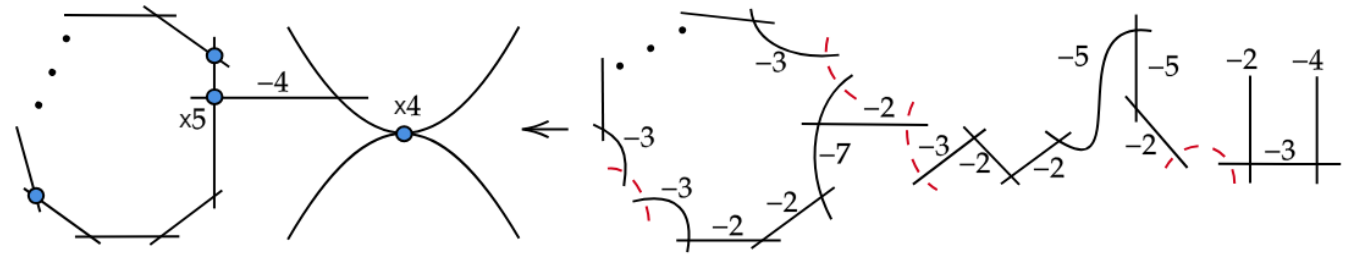}
\vspace*{-2.3em}
\caption{Configurations (S2F.2).}
\label{fS2F.2}
\end{figure}

\noindent
\textbf{(S2F.3)} This is in Lemma \ref{allforsection} $(2.2)$. We use an $I_n$ with $n\geq 4$ followed by a section $\Gamma$ such that $\overline{\Gamma}^2=-5$, and a fiber $F$ of one of the following types: $II$, $III$, or $I_{n'}$ with $n'\geq 1$. When $n=4$ and ($n'=1$ or $F=II)$, we have the Wahl chain $[3,2,2,3,5,5,2]$. If $n$ and/or $n'$ increases, then we have the P-resolution $$[3,\underbrace{2,\ldots,2}_{n-6},3]-(1)-[3,2,2,3,5,5,2]-(1)-[4,\underbrace{2,\ldots,2}_{n'-3},3,2]$$ over $[3,\overbrace{2,\ldots,2}^{n-2},3,5,4,\overbrace{2,\ldots,2}^{n'-2},3,2]$, as in Figure \ref{fS2F.3}. If $F=III$, then we have the same P-resolution as when $F=I_1$ but the Wahl chain $[3,2,2,3,5,5,2]$ is connected to the T-chain $[2,3,4]$ by a $(-1)$-curve. The discrepancies of the curves attached to this $(-1)$-curve are $-\frac{11}{25}$ and $-\frac{2}{3}$ respectively.
In all cases $K^2=4$. In the general case one computes $K^2=-6+(n-4-(n-4)+1)+(7)+(n'-(n'-1)+1)=4$. We have $d(\bG)=-\frac{24}{25}$.

\begin{figure}[htbp]
\centering
\vspace*{-1em}
\includegraphics[width=7.2cm]{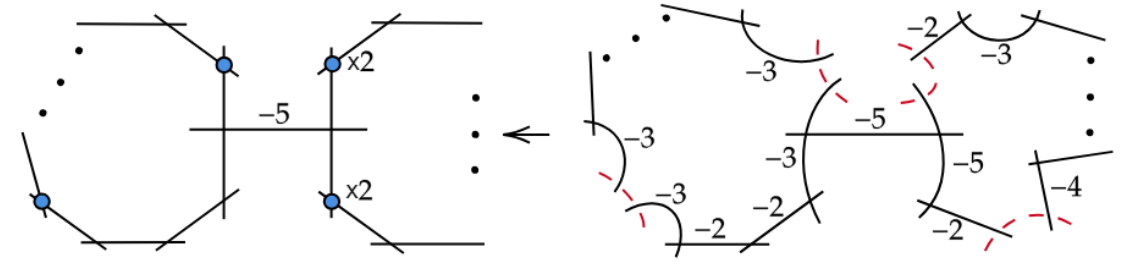}

\includegraphics[width=5.6cm]{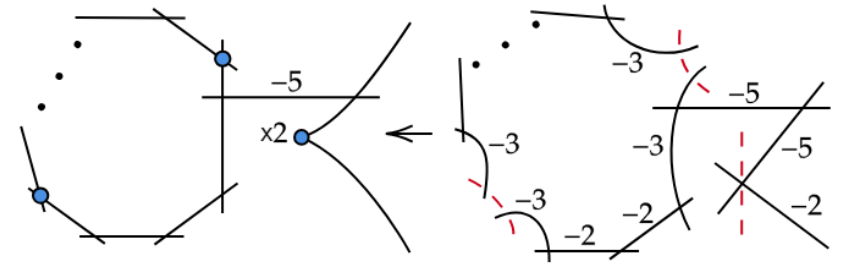} \ \ \hspace{0,2em} \includegraphics[width=7.1cm]{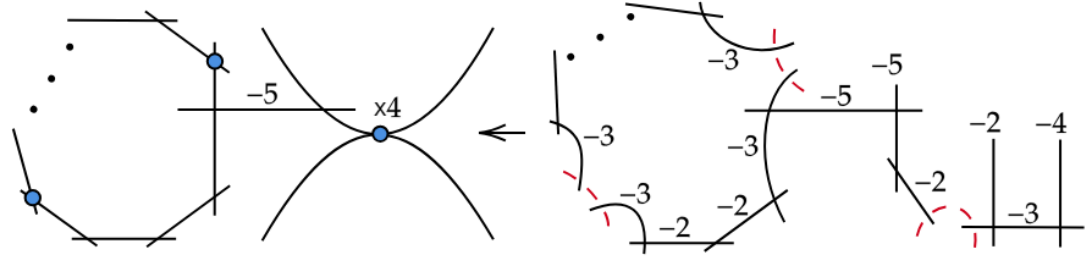}
\vspace*{-1.2em}
\caption{Configurations (S2F.3).}
\label{fS2F.3}
\end{figure}

\noindent
\textbf{(S2F.4)} This is in Lemma \ref{allforsection} $(2.2)$. We use an $I_n$ followed by a section $\Gamma$ such that $\overline{\Gamma}^2=-(r+3)$ and $n\geq r+1$; and a fiber $F$ of one of the following types: $II$, $III$, or $I_{n'}$ with $n'\geq 1$. When $n=r+1$ and ($n'=1$ or $F=II$), we have the P-resolution $$[3,\underbrace{2,\ldots,2}_{r-1},3,r+3,2]-(1)-[3,2,6,2] \to [3,\underbrace{2,\ldots,2}_{r-1},3,r,5,2].$$ If $n$ and/or $n'$ increases, then we have the P-resolution $$[3,\underbrace{2,\ldots,2}_{n-r-3},3]-(1)-[3,\underbrace{2,\ldots,2}_{r-1},3,r+3,2]-(1)-[3,2,6,2]-(1)-[4,\underbrace{2,\ldots,2}_{n'-3},3,2]$$ over $[3,\overbrace{2,\ldots,2}^{n-2},3,4,4,\overbrace{2,\ldots,2}^{n'-2},3,2]$, as in Figure \ref{fS2F.4}. If $F=III$, then we have the same P-resolution as when $n'=1$ but the Wahl chain $[3,2,6,2]$ is connected to the a T-chain $[2,3,4]$ by a $(-1)$-curve. The discrepancies of the curves attached to this $(-1)$-curve are $-\frac{3}{7}$ and $-\frac{2}{3}$.
In all cases $K^2=r-1$. In the general case one computes $K^2=-10+(1) + (r+3-2+1) + (4) + (2)=r-1$. We have $d(\bG)=-\frac{2r+2}{2r+3}$.

\begin{figure}[htbp]
\centering
\vspace*{-1.2em}

\includegraphics[width=9cm]{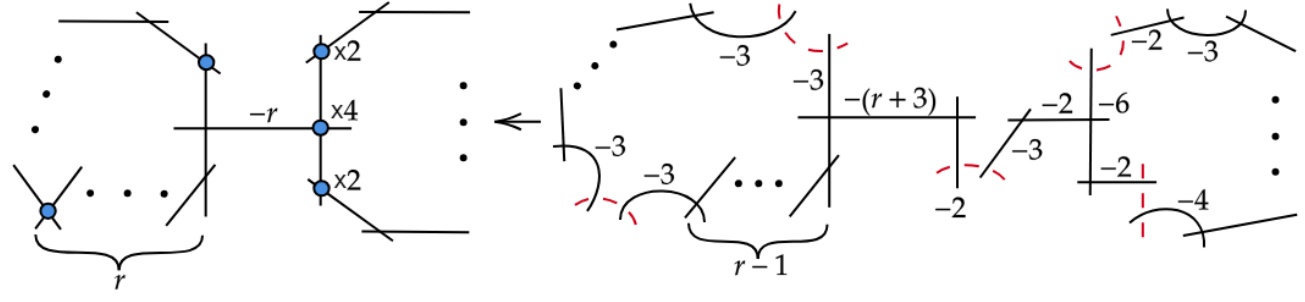}

\includegraphics[width=7cm]{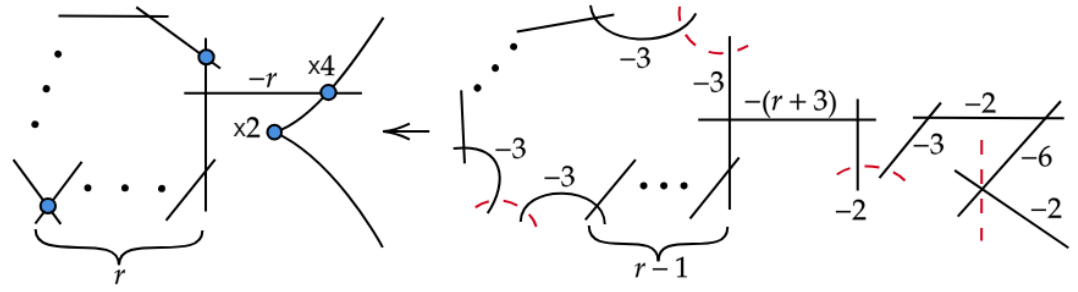} \ \ \hspace*{-1em} \includegraphics[width=7.6cm]{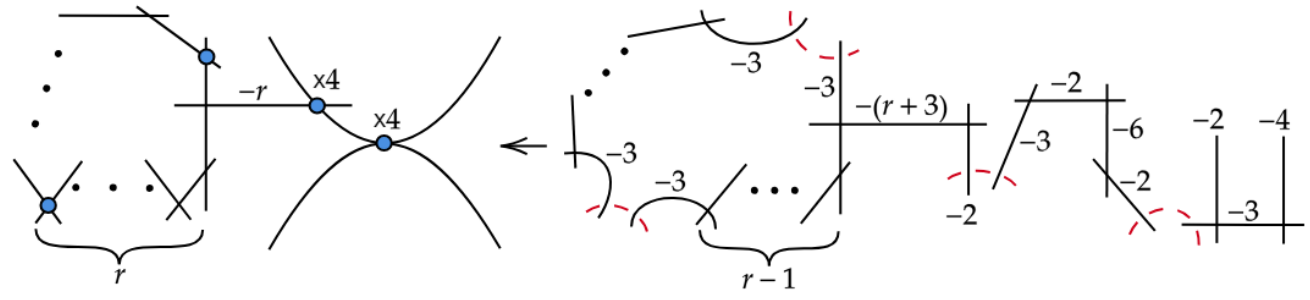}
\vspace*{-2.3em}
\caption{Configurations (S2F.4).}
\label{fS2F.4}
\end{figure}

\noindent
\textbf{(S2F.5)}
This is in Lemma \ref{allforsection} $(2.2)$. We use an $I_n$ or $III$ followed by a section $\Gamma$ such that $\overline{\Gamma}^2=-(r+2)$ and $n\geq r$; and an $I_{n'}$ with $n'\geq 2$. When $n=r$ and $n'=2$, we have the P-resolution $$[3,\underbrace{2,\ldots,2}_{r-2},3,r+2,2]-(1)-[3,5,3,2] \to [3,\underbrace{2,\ldots,2}_{r-2},3,r,4,3,2].$$ If $n$ and/or $n'$ increases, then we have the P-resolution $$[3,\underbrace{2,\ldots,2}_{n-r-2},3]-(1)-[3,\underbrace{2,\ldots,2}_{r-2},3,r+2,2]-(1)-[3,5,3,2]-(1)-[4,\underbrace{2,\ldots,2}_{n'-4},3,2]$$ over $[3,\overbrace{2,\ldots,2}^{n-2},3,r,4,\overbrace{2,\ldots,2}^{n'-2},3,2]$, as in Figure \ref{fS2F.5}. In all cases $K^2=r-1$. In the general case one computes $K^2=-9+(1) + (r+2-2+1) + (4) + (2)=r-1$. We have $d(\bG)=-\frac{2r}{2r+1}$.  
\begin{figure}[htbp]
\vspace*{-1.2em}
\centering
\includegraphics[width=7.3cm]{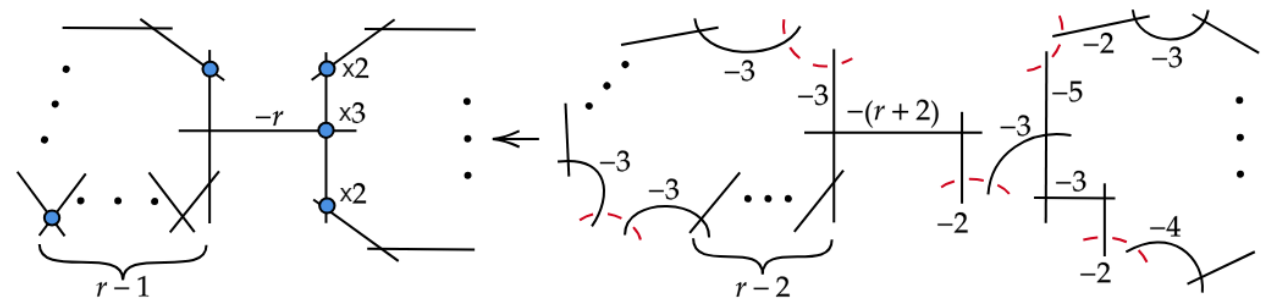} \hspace*{-0.4em} \includegraphics[width=7.3cm]{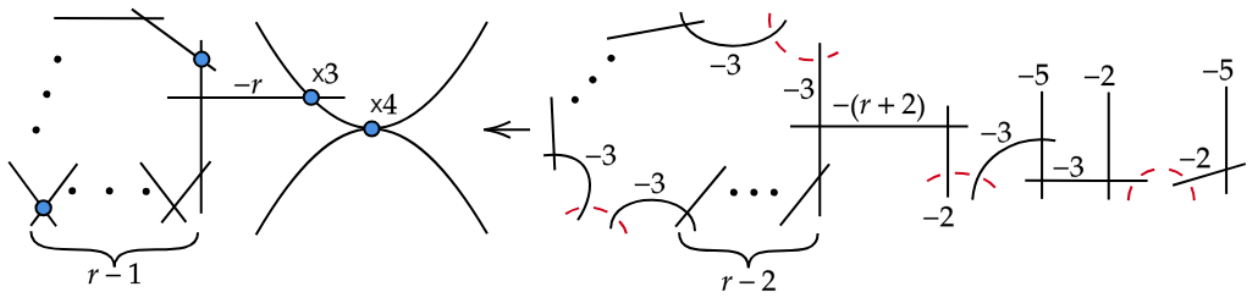}
\vspace*{-2.2em}
\caption{Configurations (S2F.5).}
\label{fS2F.5}
\end{figure}

\noindent
\textbf{(S2F.6)}
This is in Lemma \ref{allforsection} $(2.2)$. We use a fiber $F$ of type $III$, $IV$, or $I_n$ with $n\geq 2$ followed by a section $\Gamma$ such that $\overline{\Gamma}^2=-(r+2)$, and an $I_{n'}$ with $n'\geq r$. When $n'=r$ and ($n=2$ or $F=III$), we have the P-resolution $$[3,5,2]-(1)-[3,r+2,\underbrace{2,\ldots,2}_{r-3},3,2]-(1)-[3,r+3,\underbrace{2,\ldots,2}_{r-2},3,2]$$ over $[3,3,r,4,2,\ldots,2,3,2]$ with $(r-2)$ $2$s in the middle. If $n$ and/or $n'$ increases, then we have the P-resolution $$ [3,\underbrace{2,\ldots,2}_{n-4},3]-(1)-[3,5,2]-(1)-[3,r+2,\underbrace{2,\ldots,2}_{r-3},3,2]-(1)- \hspace{4cm}$$ $$\hspace{5cm} [3,r+3,\underbrace{2,\ldots,2}_{r-2},3,2]-(1)- [4,\underbrace{2,\ldots,2}_{n'-r-2},3,2]$$ over $[3,\underbrace{2,\ldots,2}_{n-2},3,r,4,\underbrace{2,\ldots,2}_{n'-2},3,2]$, as in Figure \ref{fS2F.6}. If $F=IV$, then the first chain in this P-resolution is a $[4]$, and the chain $[3,5,2]$ is connected to another $[4]$ in the fiber by the first curve. The discrepancies of these curves are $-\frac{3}{5}$ and $-\frac{1}{2}$ respectively. In all cases $K^2=r-1$. In the general case one computes $K^2=-(r+10)+(1) + (3) + (r+1) + (r+2)+ (2)=r-1$. We have $d(\bG)=-\frac{3r-2}{3r-1}$.  
\begin{figure}[htbp]
\centering
\vspace*{-1.2em}
\includegraphics[width=9cm]{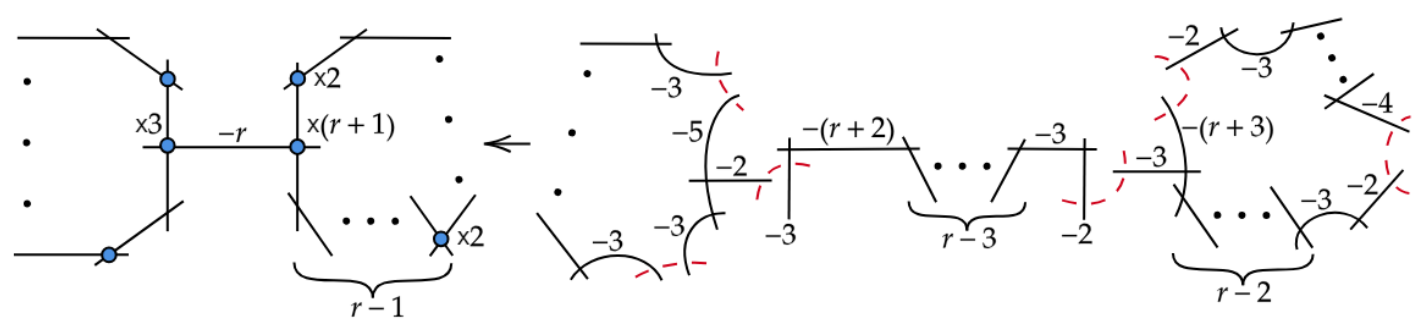}
\includegraphics[width=7.2cm]{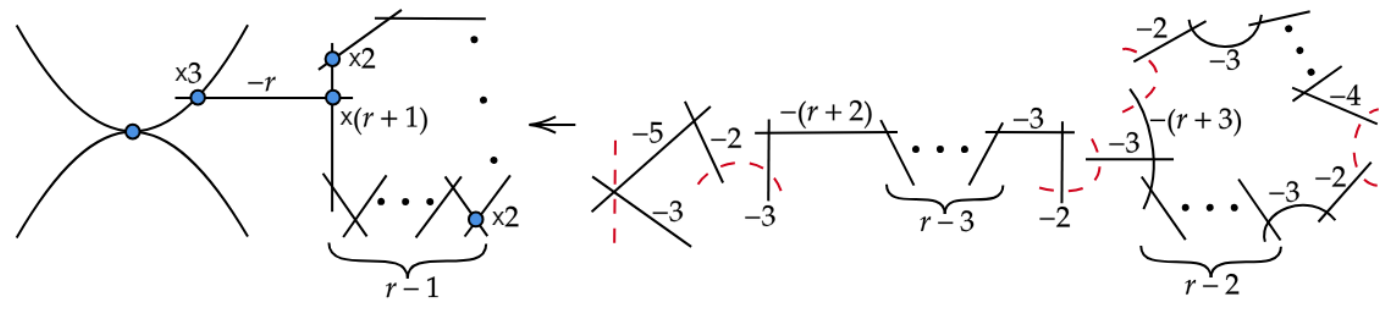}  \includegraphics[width=7.2cm]{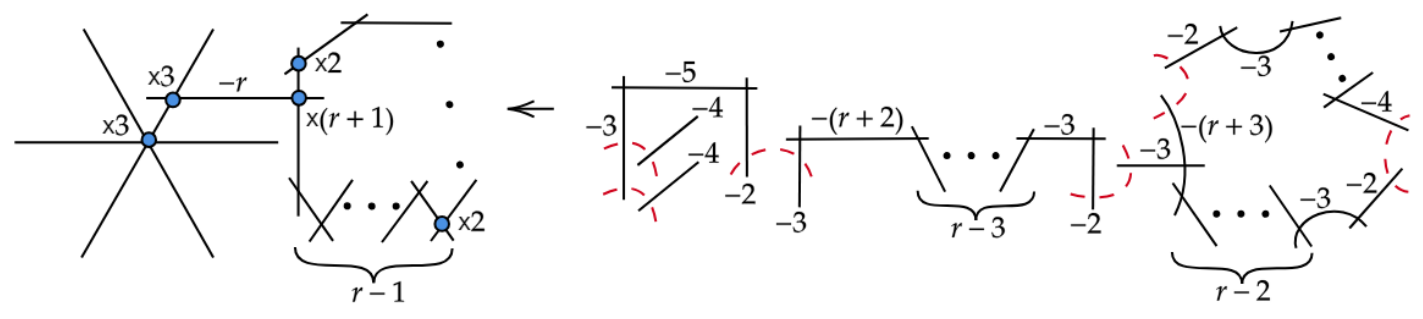}
\vspace*{-1.1em}
\caption{Configurations (S2F.6).}
\label{fS2F.6}
\end{figure}

\vspace{0.2cm}

\noindent
\textbf{(S2F.7)}
This is in Lemma \ref{allforsection} $(2.2)$. We use a fiber $F$ of type $II$, $III$, $IV$, or $I_{n}$ with $n\geq 1$ followed by a section $\Gamma$ such that $\overline{\Gamma}^2=-r$, and an $I_{n'}$ with $n'\geq r-1$. When $n'=r-1$ and $F=II$ or $I_1$, we have the Wahl chain $[4,r,5,2,\ldots,2,3,2,2]$ with $(r-3)$ $2$s in the middle. When $n'=r-1$ and $F=III$ we obtain the same Wahl chain plus one $[4]$ and one $[2]$. When $n'=r-1$ and $F=IV$, we obtain the same Wahl chain plus one $[3,2,3]$. If $n$ and/or $n'$ increases, then we have the P-resolution $$ [3,\underbrace{2,\ldots,2}_{n-3},3]-(1)-[4,r,5,\underbrace{2,\ldots,2}_{r-3},3,2,2]-(1)-[5,\underbrace{2,\ldots,2}_{n'-r-1},3,2,2]$$ over $[3,\overbrace{2,\ldots,2}^{n-2},3,r,5,\overbrace{2,\ldots,2}^{n'-2},3,2,2]$, as in Figure \ref{fS2F.7}. When $F=III$ or $F=IV$, the first chain of the P-resolution is connected to one $[4]$ or $[3,2,3]$ respectively. The discrepancies of the connected curves are $-\frac{6r-5}{8r-6}$ and $-\frac{1}{2}$. In all cases $K^2=r-1$. In the general case one computes $K^2=-8+(1) + (r+3)+ (3)=r-1$.  We have $d(\bG)=-\frac{8r-8}{8r-6}$.    

\begin{figure}[htbp]
\centering
\vspace*{-1.4em}
\includegraphics[width=7.1cm]{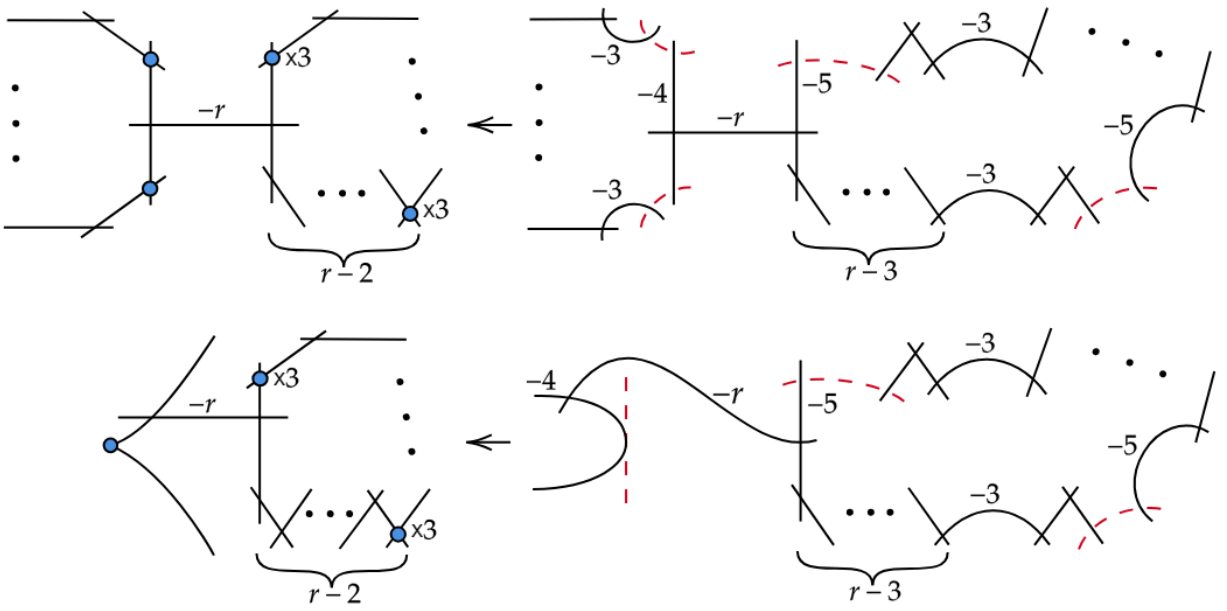}  \includegraphics[width=7.2cm]{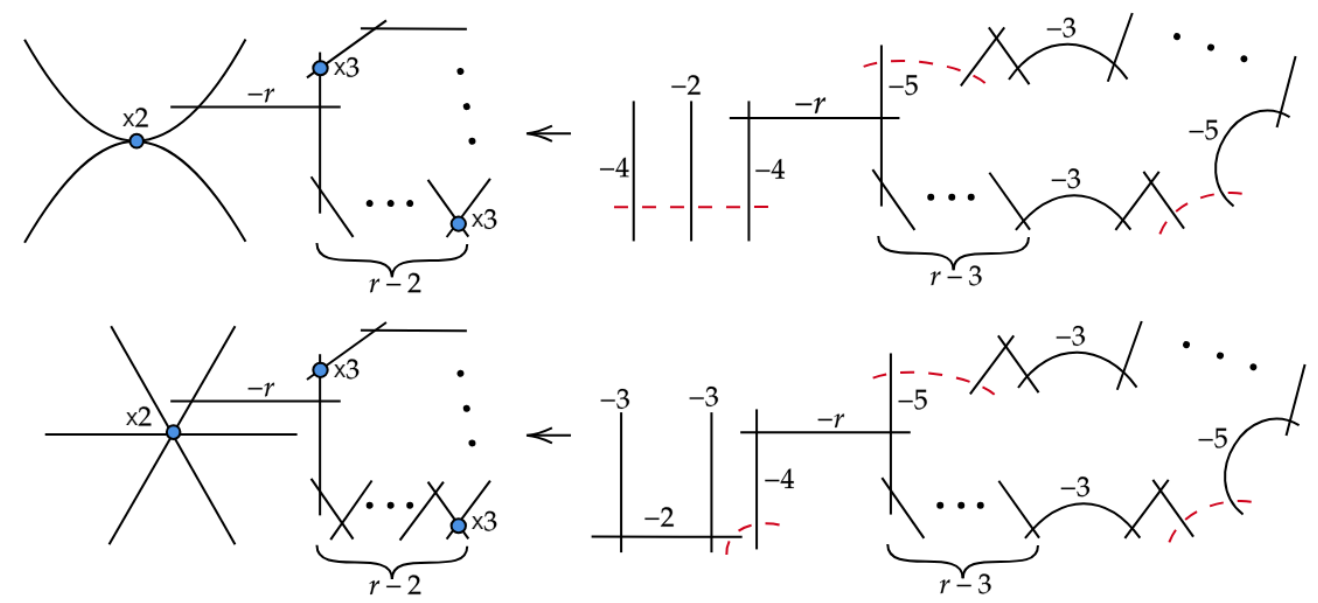} 
\vspace*{-1.1em}
\caption{Configurations (S2F.7).}
\label{fS2F.7}
\end{figure}

\noindent
\textbf{(S2F.8)}
This is in Lemma \ref{allforsection} $(2.2)$. We use a fiber $F$ of type $II$, $III$, $IV$, or $I_n$ with $n\geq 1$ followed by a section $\Gamma$ such that $\overline{\Gamma}^2=-(r+1)$, and an $I_{n'}$ with $n'\geq r$. When $n'=r$ and ($n=1$ or $F=II$), we have the P-resolution $$[4,r+1,\underbrace{2,\ldots,2}_{r-4},3,2,2]-(1)-[4,r+3,\underbrace{2,\ldots,2}_{r-2},3,2,2]$$ over $[4,r,5,2,\ldots,2,3,2,2]$ with $(r-2)$ $2$s in the middle. When $n'=r$ and $F=III$, we obtain the same Wahl chain plus one $[4]$ and one $[2]$. When $n'=r$ and $F=IV$, we obtain the same Wahl chain plus one $[3,2,3]$. If $n$ and/or $n'$ increases, then we have the P-resolution 

$$[3,\underbrace{2,\ldots,2}_{n-3},3]-(1)-[4,r+1,\underbrace{2,\ldots,2}_{r-4},3,2,2]-(1)- \hspace{6cm}$$ $$\hspace{5cm} [4,r+3,\underbrace{2,\ldots,2}_{r-2},3,2,2]-(1)-[5,\underbrace{2,\ldots,2}_{n'-r-2},3,2,2]$$ over $[3,\overbrace{2,\ldots,2}^{n-2},3,r,5,\overbrace{2,\ldots,2}^{n'-2},3,2,2]$, as in Figure \ref{fS2F.8}. When $F=III$ or $F=IV$, the first chain of the P-resolution is connected to one $[4]$ or $[3,2,3]$ respectively. The discrepancies of the connected curves are $-\frac{3r-4}{4r-5}$ and $-\frac{1}{2}$ ($\frac{3r-4}{4r-5}+\frac{1}{2}=\frac{10r-13}{8r-10}>1$, for $2r>3$). In all cases $K^2=r-1$. In the general case one computes $K^2=-(r+9)+(1) + (r+1)+(r+3)+ (3)=r-1$.  We have $d(\bG)=-\frac{4r-6}{4r-5}$.

\begin{figure}[htbp]
\centering
\vspace*{-1em}
\includegraphics[width=7.2cm]{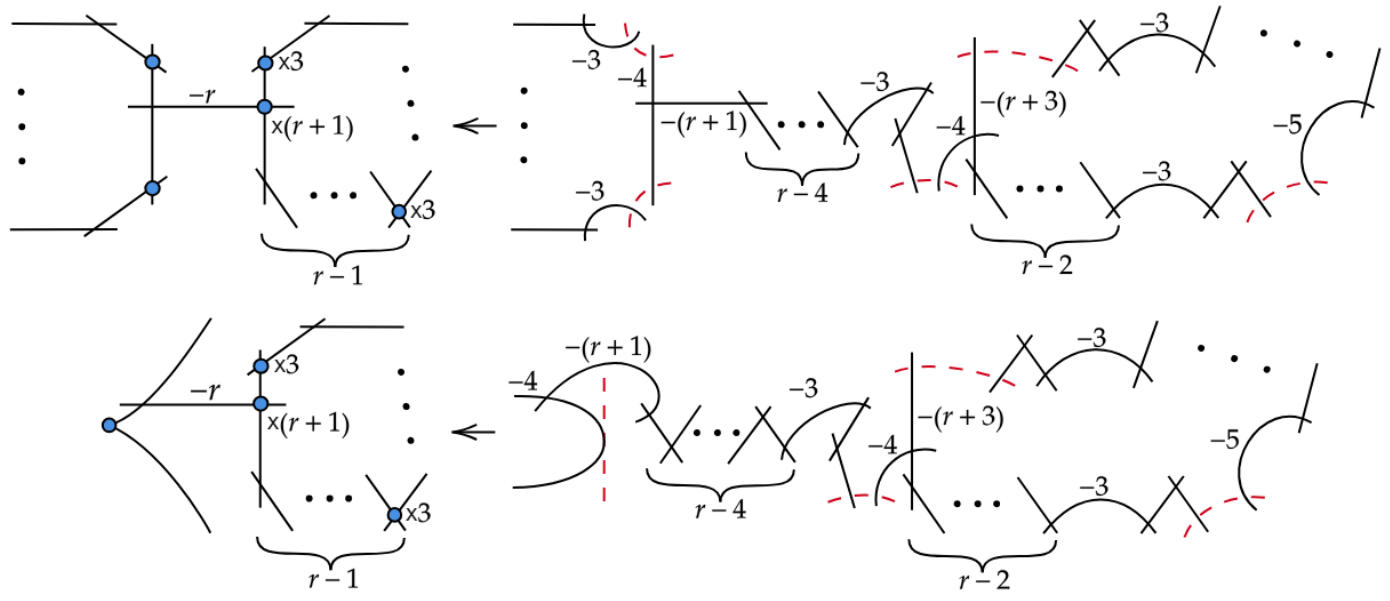} \includegraphics[width=7.3cm]{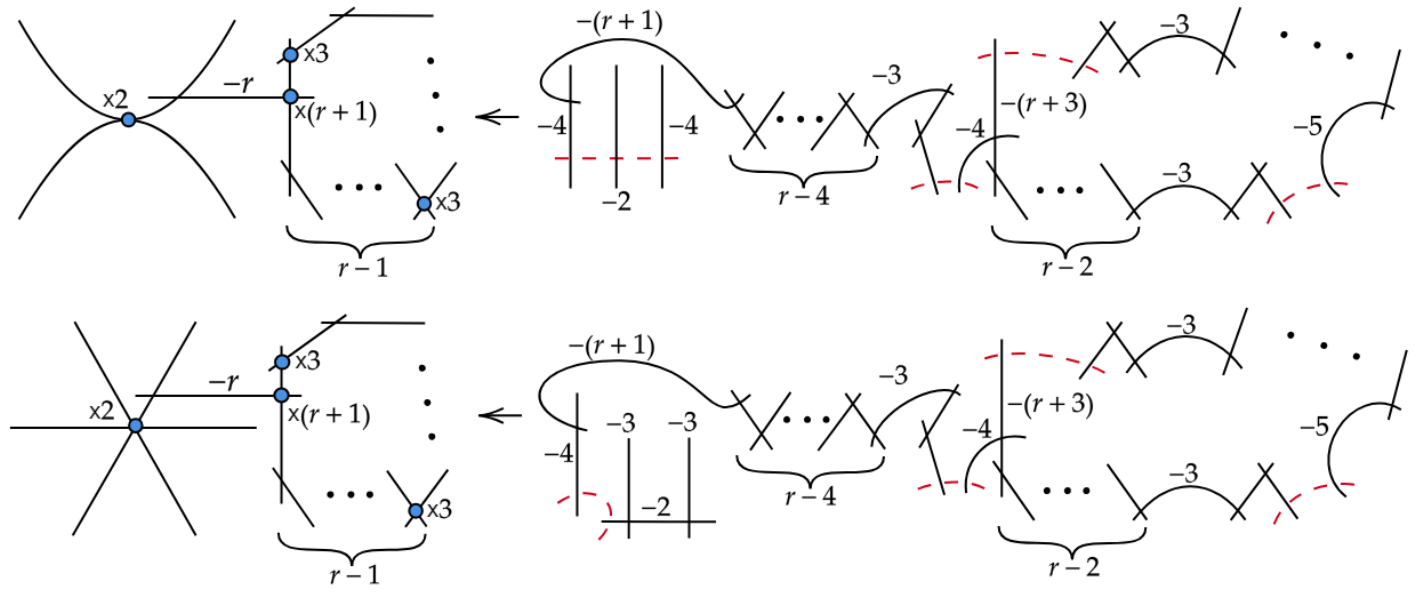}
\vspace*{-2.1em}
\caption{Configurations (S2F.8).}
\label{fS2F.8}
\end{figure}

\noindent
\textbf{(S2F.9)}
This is in Lemma \ref{allforsection} $(2.2)$. We use a fiber $F$ of type $III$, $IV$, or $I_n$ with $n\geq 2$ followed by a section $\Gamma$ such that $\overline{\Gamma}^2=-3$, and a fiber $F'$ of type $III$ or $I_{n'}$ with $n'\geq 2$. When ($n=2$ or $F=III$) and $F'=I_2$, we have the T-chain $[3,3,3,4,3,2]$. When $F'$ is changed to a $III$ we obtain the same T-chain plus one $[2,5]$, whereas when $F$ is changed to a $IV$ we have the same T-chain plus two $[4]$. If $n$ and/or $n'$ increases, then we have the P-resolution $[3,\underbrace{2,\ldots,2}_{n-4},3]-(1)-[3,3,3,4,3,2]-(1)-[4,\underbrace{2,\ldots,2}_{n'-4},3,2]$ over $[3,\overbrace{2,\ldots,2}^{n-2},3,3,4,\overbrace{2,\ldots,2}^{n'-2},3,2]$, as in Figure \ref{fS2F.9}. When $F'=III$, the T-chain $[3,3,3,4,3,2]$ is connected to one $[2,5]$ and the discrepancies of the connected curves are $-\frac{10}{13}$ and $-\frac{1}{3}$ respectively ($\frac{10}{13}+\frac{1}{3}=\frac{43}{39}>1$). When $F=IV$, the T-chain $[3,3,3,4,3,2]$ is connected to some $[4]$ and the discrepancies of the connected curves are $-\frac{8}{13}$ and $-\frac{1}{2}$ respectively ($\frac{8}{13}+\frac{1}{2}=\frac{29}{26}>1$). In all cases $K^2=3$. In the general case one computes $K^2=-6+(1)+(4)+3=2$.  We have $d(\Gamma)=-\frac{12}{13}$.

\begin{figure}[htbp]
\centering
\vspace*{-1.1em}
\includegraphics[width=7cm]{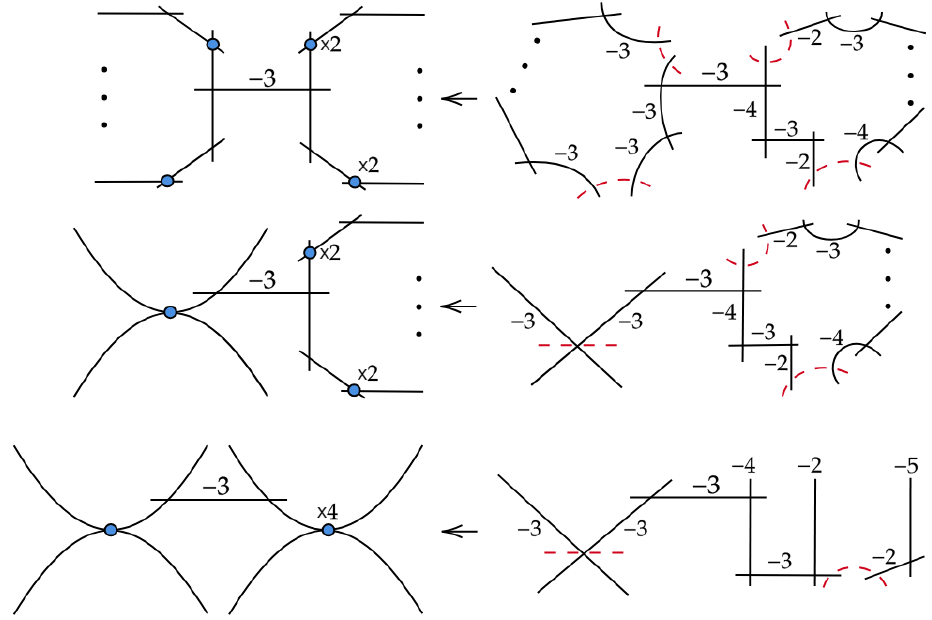} \ \includegraphics[width=7cm]{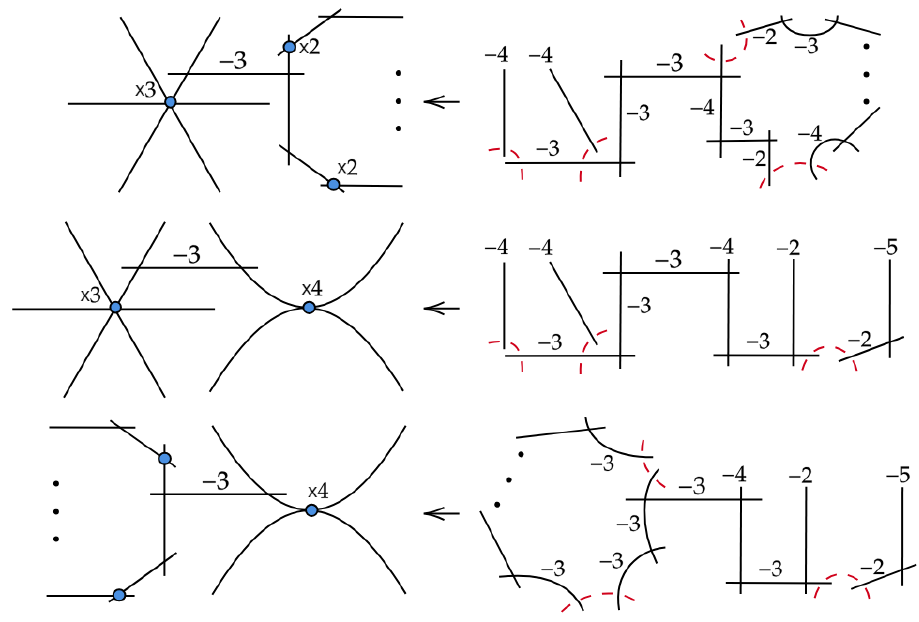}
\vspace*{-1.4em}
\caption{Configurations (S2F.9).}
\label{fS2F.9}
\end{figure}

\noindent
\textbf{(FIB)}
This is $(3.1)$ and $(3.2)$ in Lemma \ref{allfiber}. We use one $I_n$ with $n\geq 1$, or II, or III, or IV. It is connected with a section $\bG$ but this section is connected with some other P-resolution. When $n=1$ we have the Wahl chain $[2,5]$. When $n=2$ we have the Wahl chains $[2,5]$ and $[4]$. Cases II, III, and IV are similar, see Figure \ref{fFIB}. For $n\geq 3$, we obtain Wahl chain $[2,5]$ and T-chain $[3,\underbrace{2,\ldots,2}_{n-3},3]$ as shown in Figure \ref{fFIB}. In all cases $K^2=-1$. In the general case one computes $K^2=-(4)+(2)+(1)=-1$.
The discrepancy of the $(-2)$-curve which is joint to $\bG$ via a $(-1)$-curve is $-\frac{1}{3}$.   
\begin{figure}[htbp]
\centering
\vspace*{-1.35em}
\includegraphics[width=7cm]{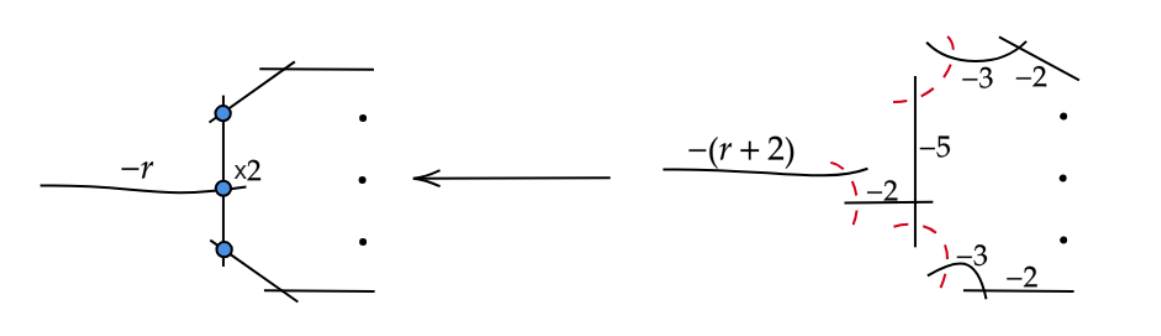} \ \includegraphics[width=7cm]{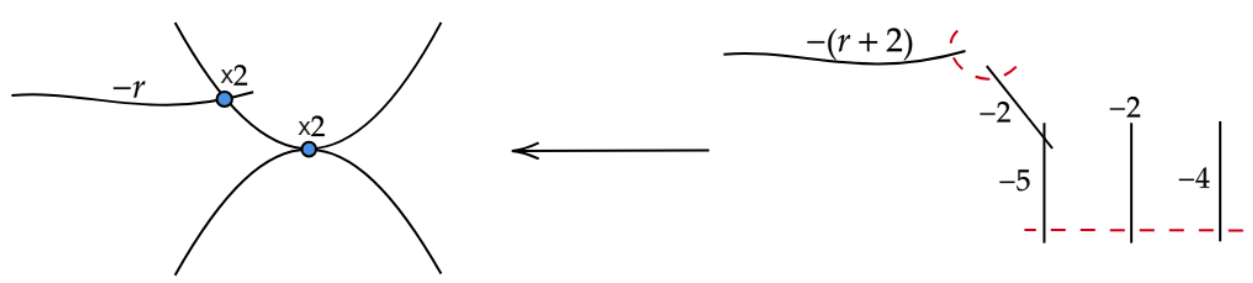} 

\includegraphics[width=6.5cm]{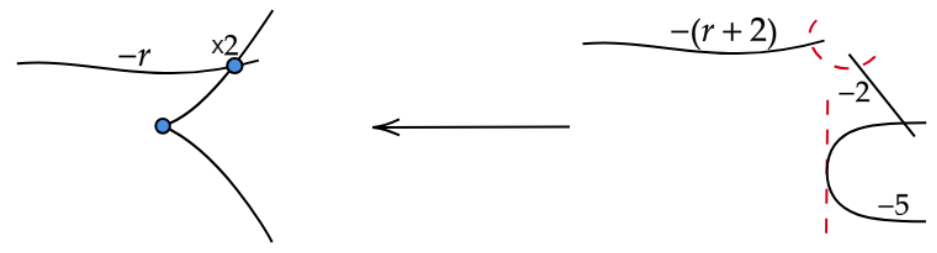} \ \hspace{0,2em}\includegraphics[width=7.2cm]{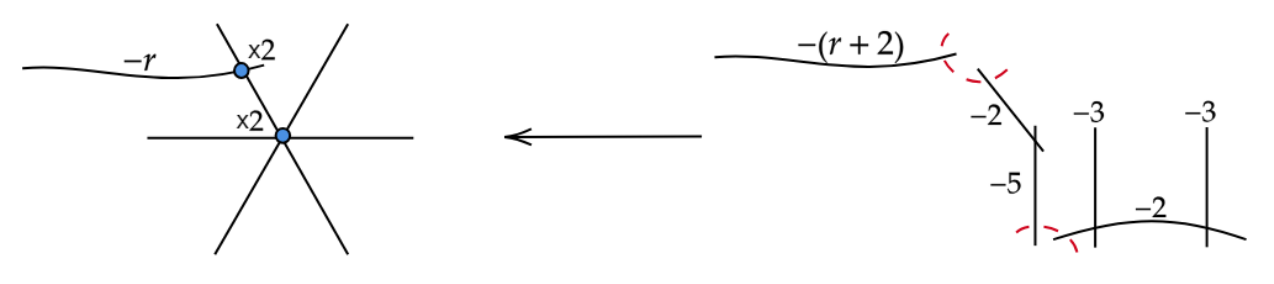}
\vspace*{-1.2em}
\caption{Configurations (FIB).}
\label{fFIB}
\end{figure}

\begin{remark}
We note that the building blocks S2F.7 y S2F.8 are related by a wormhole c.q.s. as defined and studied in \cite{UV22} (i.e. it admits two extremal P-resolutions). If there are global complex smoothings for the corresponding surfaces $W'$, then this would represent a wormhole in the KSBA moduli space of Horikawa surfaces. The c.q.s. is $[4,r,5,\underbrace{2,\ldots,2}_{r-2},3,2,2]$ and it appears for S2F.7 and S2F.8 when $n=1$ and $r=n'$. Both extremal P-resolutions have $\delta=2$, and for S2F.7 the indices have gcd$=2$, but for S2F.8 the gcd$=1$. We do not know what the topology is after we do the rational blowdown on the corresponding Wahl chains (see Subsection \ref{topo} and Remark \ref{smallsurfsimply}). Wormholes may change topology \cite[Section 6]{UV22}.
\label{wormhole}
\end{remark}

\subsection{Classification theorems} \label{ClassSmallSurf}

\begin{proposition}
All possible realizations for Lemma \ref{allforsection} and Lemma \ref{allfiber} are shown in the list \ref{blocks}.
\label{realizations}
\end{proposition}

\begin{proof} 
Of course, the strategy is to go through each of the possibilities in Lemma \ref{allforsection} and Lemma \ref{allfiber} checking if there are suitable P-resolutions for the claimed c.q.s. All of these P-resolutions are classified in \ref{app}.

For \textbf{(0)} and \textbf{(1.0)} in Lemma \ref{allforsection}, we must consider P-resolutions of $[r,2,\ldots,2]$. By Proposition \ref{2r2} the only one that works is contraction of $[r,2,\ldots,2]$ as a Wahl singularity, and this is S0F. For \textbf{(1.1)}, we need to consider P-resolutions of $[r,a,\underbrace{2,\ldots,2}_{a-4}]$ where $a\geq 4$. They are described in Proposition \ref{2ra2} when $b=0$ and $c=a-4$, and only (h) is possible as we cannot have a hanging $A_x$ over the ending $2$s. This gives an example in S1F.1. For \textbf{(1.2)} we need to know about P-resolutions of $[r,a,\underbrace{2,\ldots,2}_{n-2},3,\underbrace{2,\ldots,2}_{a-3}]$, which is Proposition \ref{2ra232} with $b=0$. Again there must not be a hanging $A_x$ over the ending $2$s. Then only (o) is possible, with $a=r+1$ and $n\geq r-2$; this is S1F.2 if $r>3$ or S1F.1 if $r=3$.

The cases \textbf{(1.3.1)}, \textbf{(1.3.2)} and \textbf{(1.3.3)} work as \textbf{(1.1)}, and so we obtain examples in S1F.1. For \textbf{(1.3.4)}, we need a P-resolution of $[r,7,\underbrace{2,\ldots,2}_{r'-2}]$, and a P-resolution of $[3,r',2]$. Again, the only P-resolution for $[r,7,\underbrace{2,\ldots,2}_{r'-2}]$ not having a hanging $A_x$ is (h) in Proposition \ref{2ra2}, but $a=5$, so not possible. Similarly for \textbf{(1.4)} we obtain only one example in S1F.1 from \textbf{(1.4.4)}, and for \textbf{(1.5)} we obtain none.     

Before checking possibilities for two fibers \textbf{(2.1)} and \textbf{(2.2)}, which is the hardest, we check possible P-resolutions only from fibers, i.e. Lemma \ref{allfiber}.

\textbf{(3.1)} and \textbf{(3.2)} can be checked via Propositions \ref{2r2} and \ref{2a232}. For the first, we easily get the answer. For the second we have five options, but to avoid hanging $A_x$ at the ends, we need $a=4$. Then we obtain (b) and (d), and they coincide in the examples of FIB. Similarly one can work out the cases \textbf{(3.3)}, \textbf{(3.4)} and \textbf{(3.5)}, which only give examples in FIB. We remark that the fiber II produces an example with a free $(-2)$-curve. For \textbf{(2.1)} with an $I_1$ fiber, we have to consider P-resolutions of $[2,\ldots,2,r,a,2,\ldots,2]$. Via Proposition \ref{2ra2} we realize that only (d) is possible with $c=a-4=0$, $r=a=3$, $b=1$, and (g) with $a=4$ and $b=2$. The first is S1F.4, the second in S1F.3. More general, for \textbf{(2.1)} with an $I_n$ fiber we need to check P-resolution in Proposition \ref{2ra232}. One can verify that the only possibilities are (i) producing S1F.4, and (m) producing S1F.3. For \textbf{(2.1)} plus \textbf{(1.3.1)}, \textbf{(1.3.2)} or \textbf{(1.3.3)}: Same analysis as $I_1$. For \textbf{(1.3.4)} we check Propositions \ref{2ra2} and \ref{typeII}. We obtain, all in all, members in S1F.3 and S1F.4. For \textbf{(2.1)} plus \textbf{(1.4.1)}, we have the same as for \textbf{(2.1)} plus $I_2$, so we obtain the same. For \textbf{(2.1)} plus \textbf{(1.4.2)}, we need to check P-resolutions of $[2,\ldots,2,r,4,3,2]$ (Proposition \ref{2ra232} with $a=4$ and $n=2$ and $b>0$) attached to $[5,2]$. There are none. For \textbf{(2.1)} plus \textbf{(1.4.3)}, we are as in \textbf{(2.1)} plus $I_1$, so we obtain the same. For \textbf{(2.1)} plus \textbf{(1.4.4)} we use Proposition \ref{2ra2} (where $b>0$, $a=5$, $c=1,2,3$) and Proposition \ref{typeIII} (where $a=3,4,5$). There are none. Similarly one can verify \textbf{(2.1)} plus \textbf{(1.5)}, obtaining S1F.3 and S1F.4.

The case \textbf{(2.2)} is the richest, and the subcase with two fibers $I_n, I_{n'}$ produces the majority of the possibilities. Because of Lemma \ref{In+r+In'} and the restriction on ``ending $2$s", we can directly use Propositions \ref{4rb232}, \ref{2ar323} and \ref{323rb232} to find all possibilities from S2F.1 to S2F.9. 

For $I_n$ plus II we have the following. \textbf{(1.3.1)} and \textbf{(1.3.2)}: Just as for $I_n + I_{n'}$ with $n'=1,2$. \textbf{(1.3.3)}: Here there is a small observation. We have the ending $2$ in $[\ldots,r,6,2,2]$ which may not be used. But if not, then the ending $2$ in $[\ldots,r,6,2]$ cannot be a center, and so any of their P-resolutions ends with a $2$. But this is connected to a $[4]$ and so discrepancies are not enough to make $K_W$ positive. So, it is enough to consider the options in Proposition \ref{2ar323}. \textbf{(1.3.4)}: We need to consider all P-resolutions in Proposition \ref{2ar323}, and the ones in Proposition \ref{typeII}.

For $I_n$ plus III we have the following. \textbf{(1.4.1)}: This is just a particular $I_n + I_2$. \textbf{(1.4.2)}: In this case we have a $[\ldots,r,4,3,2]$ and a $[5,2]$ with a $(-1)$-curve between the $[3]$ and the $2$ in $[5,2]$. If the last $2$ of $[\ldots,r,4,3,2]$ is considered in the P-resolution, then this is as $I_n+I_2$. If not, then we look for P-resolutions of $[\ldots,r,4,3,1,2,5]=[\ldots,r,3,3]$, and this can be analyzed as a $I_n+I_2$ as well. \textbf{(1.4.3)} and \textbf{(1.4.4)} are treated similarly as before. The case $I_n$ plus IV is similarly analyzed. 

We are only left with the small cases of combination of two of II, III, IV. One gets very few examples from these combinations.

For the case of Lemma \ref{allfiber}, we end up only with the building block FIB.
\end{proof}

\begin{theorem}[Classification of small surfaces] Let $W$ be a small surface with $K_W$ big and nef. Then $W$ is constructed from a S0F, or a S1F.i, or a S2F.i for some $i$, by adding, if necessary, a suitable amount of $s$ FIBs. Every such choice of building blocks is realizable. We have that $K_W^2=p_g-2+s+j$, where $j=0,1,2$ depends on the choice of S0F ($j=0$) or S$j$F.i. If $K_W$ is not ample, then its canonical model $W_{\text{can}}$ is obtained by contracting ADE configurations from components of fibers which are disjoint from the used building blocks.
\label{ClassSmallSurf}
\end{theorem}

\begin{proof}
Lemma \ref{allforsection} and Lemma \ref{allfiber} give all possibilities to potentially construct a small surface $W$ from one section and components of fibers. By Proposition \ref{realizations}, we have that the building blocks in the list \ref{blocks} are all the possible actual configurations from Lemma \ref{allforsection} and Lemma \ref{allfiber}. 

Therefore, we are left to check if we can glue building blocks S0F, or S1F.i, or a S2F.i, with FIBs (if necessary) to construct $W$. The self-intersection of the proper transform of the section $\Gamma$ is $\bG$ with $\bG^2=-r$. Note that $r=p_g+1+t$ for some $t$. Now, the only way to have $t>0$ is by adding FIBs, and so $t=2s$ where $s$ is the number of FIBs used. In this way $K_W^2=p_g-2+s+j$, where $j=0,1,2$ depends on the choice of S0F ($j=0$) or S$j$F.i. 

To finally show that $K_W$ is nef, we need to show that the joining $(-1)$-curve between S0F or SjF.i has the right discrepancies. (The rest of the $(-1)$-curves have been verified in the description of each building block in list \ref{blocks}.) As the relevant discrepancy from FIB is $-\frac{1}{3}$, it is key to know the discrepancy of $\bG$. But in each case it was verified to be strictly smaller than $-\frac{1}{3}$ up to the case S0F with $r=4$, but in that case there is no FIB. Therefore $K_{W'}$ is indeed big and nef in all cases.    
\end{proof}

\begin{theorem}[Geography of small surfaces]
Let $W$ be a small surface. Then
$$p_g-2\leq K^2_{W}\leq \Big(4+\frac{2}{3} \Big) p_g+\frac{11}{3}.$$
We have equality on the left if and only if $W$ is obtained by contracting a chain $[p_g(S)+1,2,\ldots,2]$ in $S$; and we have equality on the right if and only if $p_g \equiv 2$(mod $3$) and $W$ is obtained via S2F.7 gluing a suitable number of FIBs. Every $K^2$ is realizable by some $W$.
\label{geogrSmall}
\end{theorem}

\begin{proof}
By Theorem \ref{formulaK^2}, we have that $K_{W}^2=p_g(S)-2+N$ where $N$ is the number of fibers completely contained in $\pi(C)$. As $N\geq 0$, the lower bound is reached when $N=0$. In this case $p_g(S)\geq 3$, and so Corollary \ref{ineq} implies that $X=S$ and $W$ is obtained by contracting the Wahl chain $[p_g(S)+1,2,\ldots,2]$.

For the upper bound, we will bound $N$. First, by Theorem \ref{ClassSmallSurf}, we know that $W$ is constructed from one of the configurations of types S0F or S1F.i or S2F.j, which contain the section $\bG$, and $s$ of type FIB. Therefore $K_{W}^2=p_g(S)-2+s+k$ where $k=0,1,2$ depending on the type S0F or S1F.i or S2F.j used. 

On the other hand, we have $ \sum \chi_{\text{top}}(F_{\text{sing}})=12(p_g(S)+1),$ 
where $\chi_{top}(F_{\text{sing}})$ is the topological characteristic of the singular fibers $F_{\text{sing}}$. The Euler characteristic of fibers is the smallest when the fiber is $I_1$ or II, and it is equal to $1$. Therefore, to maximize $s$ we consider type FIB with only $I_1$s, and for the other complete fibers used we only consider them for $I_n$ type (we can have $0$, $1$ or $2$ of them). We check only the types which have a free $r=p_g+1+2s$, since for the others $s$ is very small. Now, by looking at the restriction of $r$ in relation to $n$ for each of the types, one can check that: 

\vspace{0.1cm}  
S0F: $s\leq \frac{1}{3}(11p_g+14)$ and so $K_W^2=\frac{1}{3}(14p_g+8)$.

S1F.2: $s\leq \frac{1}{3}(11p_g+13)$ and so $K_W^2=\frac{1}{3}(14p_g+10)$.

S1F.4: $s\leq \frac{1}{3}(11p_g+11)$ and so $K_W^2=\frac{1}{3}(14p_g+8)$.

S2F.4, S2F.5 and S2F.6: $s\leq \frac{1}{3}(11p_g+9)$ and so $K_W^2=\frac{1}{3}(14p_g+9)$.

S2F.7: $s\leq \frac{1}{3}(11p_g+11)$ and so $K_{W}^2=\frac{1}{3}(14p_g+11)$.

S2F.8: $s\leq \frac{1}{3}(11p_g+10)$ and so $K_W^2=\frac{1}{3}(14p_g+10)$.
\vspace{0.1cm} 

\noindent
Therefore, we obtain a maximum value of $K_W^2$ for S2F.7. 

Let us now fix $p_g\geq 2$. We want to show realization of $W$ for any $K^2$ between the bounds. Let us assume $p_g\equiv 1$ (mod $3$), for the other cases a similar strategy works. By \cite[Theorem 2.3]{Sh05} there exists an elliptic surface with sections $S' \to \P^1$ which has precisely an $I_{10 p_g+9}$ and $(2p_g+3)$ $I_1$ as singular fibers. By \cite[Lemma 2.4]{MP89}, we have existence of elliptic surfaces with sections $S \to \P^1$ which have precisely an $I_{10 p_g+9-a}$ and $(2p_g+3+a)$ $I_1$ as singular fibers, for any $0\leq a \leq 10 p_g+8$. We consider $a=\frac{1}{3}(5p_g-4)$, and then we construct a $W$ from a S1F.2 configuration and $s=\frac{1}{3}(11p_g-13)$ FIBs which has $K_W^2=\frac{1}{3}(14p_g+10)$. We now decrease $a$ $1$ by $1$ to fill out all possible values in the interval via again a S1F.2 and a suitable number of FIBs.     
\end{proof}

\subsection{Topology of rational blowdowns of small surfaces} \label{topo}

Let $W$ be a small surface, and let $\phi \colon X \to W$ be its minimal resolution. We recall that the exceptional divisor of $\phi$ is $C=\sum_{i=1}^{l} C_i$, where the $C_i$ are T-chains. We also have $\pi \colon X \to S$ composition of blow-downs into the minimal elliptic surface $S$. The image $\pi(C)$ contains precisely one section $\Gamma$ and $N$ complete fibers, together with some extra curves from fibers. 

Let us consider the composition of blow-ups $X' \to X$ which transforms each T-chain with associated T-singularity $\frac{1}{d_i n_i^2}(1,d_i n_i a_i -1)$ into $d_i$ Wahl chains of the Wahl singularity $\frac{1}{n_i^2}(1,n_i a_i -1)$ (i.e. the minimal resolution of an M-resolution of $W$). Let us recall the rational blowdown of the corresponding Wahl chains $C'$ (over $C$) in $X'$, following the point of view in \cite[Section 2]{RU22} (see also \cite[Section 5]{MORSU24}) and all references there. Due to Fintushel, Stern and Park, we can construct a closed smooth 4-manifold $Y$ by rationally blowdown the configuration $C'$. It can be seen as the surgery over the links of the Wahl singularities in $W'$, where $W'$ is the contraction of $C'$, which glues the corresponding Milnor fibers of their $\Q$-Gorenstein smoothings. If we have a $\Q$-Gorenstein smoothing for $W'$, then it coincides with the rational blowdown. Otherwise, it is only a closed smooth 4-manifold $Y$. 

\begin{definition}
The rational blowdown of a small surface is the rational blowdown of $C'$ in $X'$ explained above.
\end{definition}

The rational blowdown $Y$ has a symplectic structure due to a result of Symington, which is induced by a symplectic structure on $X'$ and $C'$. If $K_W$ is big and nef, then we choose it to be induced by an ample divisor on $X'$ which is very close to the canonical class of $W'$ (see proof of \cite[Theorem 2.3]{RU22}). More precisely, we have that $H_2(Y,\Z) \subset H_2(W',\Z)$ with finite index, and the symplectic form for $Y$ is in the class of $mK_{W'}$ for some $m>0$. In this way, as shown in \cite[Theorem 2.3]{RU22}, $Y$ is a smoothly minimal symplectic 4-manifold.

We compute some topological invariants of the rational blowdown $Y$ of a small surface $W$ with $K_W$ big and nef. First we have $\chi_{\text{top}}(Y)=\chi_{\text{top}}(W)$ and $K_Y^2=K_W^2=p_g-2+N$. Let $(b^+,b^-)$ be the signature of the unimodular quadratic form on $H^2(Y,\Z)$. Then we have the second Betti number $b_2(Y)= b^+ + b^-=\chi_{\text{top}}(W)-2$, and for its signature $\sigma(Y)$ we have $\sigma(Y)=b^+-b^-=\frac{1}{3}(K^2_W-2\chi_{\text{top}}(W))$. Therefore $$ b^+=2\chi(\O_W)-1=2p_g+1 \ \ \ \ \ b^-=10\chi(\O_W)-K^2_W-1=9p_g +11-N,$$ as the Noether's formula $12 \chi(\O_W)=K_W^2+\chi_{\text{top}}(W)$ holds for $W$. Thus $\sigma(Y)=-7p_g-10+N$, which is negative by Theorem \ref{geogrSmall}, and the intersection form must be indefinite. By Freedman's theorem, if the rational blowdown of a small surface $W$ with $K_W$ big and nef is simply-connected and the intersection is odd, then it is homeomorphic to $(2p_g+1)\C \P^2 \# (9p_g+11-N) \overline{\C \P^2}$.

\begin{remark}
It can be shown that the rational blowdown of a small surface is simply-connected up to possibly the ones built from S2F.7. This of course uses the classification of small surfaces. The computation can be done via the Seifert-Van-Kampen theorem as in e.g. \cite[Theorem 3]{LP07}. That trick works when we can get rid of all loops around the exceptional divisors over T-singularities. For that we use that indices of singularities are coprime (in case this holds), singular fibers which are not part of the construction, and the FIB building blocks for the corresponding loops. In the case of S2F.7 all of that fails, so it is not clear if the rational blowdown is simply-connected or not.
\label{smallsurfsimply}  
\end{remark}


\begin{remark}
By Theorem \ref{geogrSmall}, small surfaces produce via rational blowdown symplectic closed 4-manifolds with $K^2$ in the complex forbidden region between the Noether line and the half-Noether line. In fact these small surfaces are precisely all the possible minimal rational blowdowns from a non-singular complex projective surface and algebraic 2-spheres, giving a classification of them. First we know that the complex algebraic surface $X$ must be the minimal resolution of a small surface $W$ if $K_W$ is big and nef by Corollary \ref{hori}. Let us assume that $W$ with only T-singularities has $K_W$ not big and nef with $K_W^2<2p_g-4$. Assume $W$ has only Wahl singularities via an M-resolution. Then we can run MMP on $W$ with respect to $K_W$ as if we have a complex $\Q$-Gorenstein smoothing of $W$. If we have a flip, then we replace $W$ by a birational $W^+$ which produces the same rational blowdown as $W$. If we have a divisorial contration, then there is a contraction on $W$ and on its rational blowdown given by a $(-1)$-curve. But this last one is not possible, so we have only flips. Eventually this process stops with a big and nef canonical class, and  we have a classification of this last situation. 

We also point out that $K_X^2 \geq p_g(X)-2$ is a general Noether inequality valid for any normal KSBA stable $X$ with integral $K_X^2$ (if not integral, take $K_X^2 > p_g(X)-3$); Cf. \cite[Section 8.2]{Liu17}.       
\end{remark}

\section{Classification of KSBA Horikawa surfaces} \label{s3}

In this section we classify all surfaces $W$ with only T-singularities, $K_W$ big and nef, and $K_W^2=2p_g(W)-4$. 

\begin{remark}
By Remark \ref{canonical}, we indeed give a classification of all KSBA Horikawa surfaces with T-singularities, including Du Val singularities. We do have a canonical map $W \to W_{\text{can}}$ where $W_{\text{can}}$ is the KSBA surface. This map contracts $K_W$-zero curves between two Wahl singularities (M-resolutions) and contracts disjoint ADE configurations of curves which do not pass through singularities. In this last case, the contraction takes only components of some singular fibers in the elliptic fibration $S \to \P^1$.   
\end{remark}

\begin{remark}
Sometimes in the classification we have a parametric T-chain which is a Wahl chain for the initial parameter. To save notation, we consider its M-resolution instead, and so we write $n$ Wahl singularities. For example, we may have that $n$ $[6,2,2]$ is $[6,2,2]$ for $n=1$, and the T-chain $[5,2,\ldots,2,3,2,2,2]$ for $n>1$, where the number of $2$s in between is $n-2$. This is not always the case, but it will be clear from the type of the building block used. On the other hand, the parameters shown have a maximum bound which may not be realizable. We are only using as bounds the topological Euler characteristic $\chi_{\text{top}}(S)$ of $S$ which is equal to $12(1+p_g)$, and that the presence of a semi-stable fiber implies the existence of at least $3$ singular fibers. 
\label{notation}
\end{remark}

\begin{theorem} 
Let $W$ be a surface with only non-Du Val T-singularities, $K_W$ big and nef, and $K_W^2=2p_g(W)-4$. As usual, let $p_g=p_g(W)=p_g(S)$, and let $X \to W$ be its minimal resolution. Then the possible T-chains (see Remark \ref{notation}) in this contraction are the following.

\vspace{0.1cm} 
\noindent
If $p_g=3$:
        \begin{itemize}
            \item[(i)] Two $[4]$, where each $(-4)$-curve is a section.
            \item[(ii)] One $[2,2,6]$, one $[2,5]$ and $n$ $[4]$, where $0\leq n \leq 43$ (S0F).
            \item[(iii)] One $[4,5,3,2,2]$ and $n$ $[6,2,2]$, where $0\leq n\leq 44$ (S1F.2).
            \item[(iv)] One T-chain $[2,4,3,3]$ ($d=2$) and $n$ $[4]$, where $0 \leq n \leq 43$ (S1F.4).            
        \end{itemize}

\noindent
If $p_g=4$:
        \begin{itemize}
            \item[(i)] Two $[5,2]$, where each $(-5)$-curve is a section.
            \item[(ii)] One $[9,2,2,2,2,2]$, two $[2,5]$, and $n$ $[4]$, where $0\leq n \leq 52$ (S0F).
            \item[(iii)] One $[7,8,2,2,2,3,2,2,2,2,2]$, one $[2,5]$, $n_1$ $[4]$, and $n_2$ $[9,2,2,2,2,2]$, where $0\leq n_1+n_2\leq 53$ (S1F.2).
            \item[(iv)] One T-chain $[2,7,3,2,2,2,3]$ ($d=2$), one $[2,5]$, and $n$ $[4]$, where $0\leq n\leq 52$ (S1F.4).    
            \item[(v)] One $[3,2,2,3,5,5,2]$, $n_1$ $[4]$, and $n_2$ $[5,2]$, where $0\leq n_1+n_2\leq 54$ (S2F.3).
            \item[(vi)] One T-chain $[3,2,2,2,2,3,8,2]$ ($d=2$), one $[3,2,6,2]$, $n_1$ $[4]$, and $n_2$ $[2,5]$, where $0\leq n_1+n_2\leq 52$ (S2F.4).
            \item[(vii)] One T-chain $[3,2,2,2,3,7,2]$ ($d=2)$, one $[3,5,3,2]$, $n_1$ W-chains $[4]$, and $n_2$ W-chains $[2,5]$, where $0\leq n_1+n_2\leq 52$ (S2F.5).
            \item[(viii)] One $[3,7,2,2,3,2]$, one $[3,8,2,2,2,3,2]$, one $[3,5,2]$, $n_1$ W-chains $[4]$, and $n_2$ W-chains $[2,5]$, where $0\leq n_1+n_2\leq 52$ (S2F.6).
            \item[(ix)] One $[4,5,5,2,2,3,2,2]$, $n_1$ $[4]$, and $n_2$ $[6,2,2]$, where $0\leq n_1+n_2\leq 54$ (S2F.7).            
            \item[(x)] One $[4,6,2,3,2,2]$, one $[4,8,2,2,2,3,2,2]$, $n_1$ $[4]$, and $n_2$ $[6,2,2]$, where $0\leq n_1+n_2\leq 53$ (S2F.8).
        \end{itemize}

\noindent
If $p_g\geq 5$:
        \begin{itemize}
            \item[(i)] Two $[p_g+1,2,\ldots,2]$, where each $(-p_g-1)$-curve is a section.
            \item[(ii)] One $[3p_g-3,2,\ldots,2]$ with  $(3p_g-7)$ $2$s, $(p_g-2)$ $[2,5]$, and $n$ $[4]$, where $0\leq n \leq 8p_g+20$ (S0F).
            \item[(iii)] One W-chain $[3p_g-5,3p_g-4,2,\ldots,2,3,2,\ldots,2]$ with $(3p_g-9)$ $2$s in the middle and $(3p_g-7)$ $2$s on the right, $(p_g-3)$ $[2,5]$, $n_1$ $[4]$, and $n_2$ $[3p_g-3,2,\ldots,2]$, where $0\leq n_1+n_2\leq 8p_g+22$ (S1F.2).
            \item[(iv)] One T-chain $[2,3p_g-5,3,2,\ldots,2,3]$ ($d=2$) with $(3p_g-9)$ $2$s, $(p_g-3)$ $[2,5]$, and $n$ $[4]$ where $0\leq n \leq 8p_g+20$ (S1F.4).
            \item[(v)] One T-chain $[3,2,\ldots,2,3,3p_g(S)-4,2]$ ($d=2$) with $(3p_g-8)$ $2$s in the middle, $(p_g-4+n_1)$ $[2,5]$, one $[3,2,6,2]$, and $n_2$ $[4]$, where $0\leq n_1+n_2\leq 8p_g+21$ (S2F.4).
            \item[(vi)] One T-chain $[3,2,\ldots,2,3,3p_g-5,2]$ ($d=2$) with $(3p_g-9)$ $2$s in the middle, $(p_g-4+n_1)$ $[2,5]$, one $[3,5,3,2]$, and $n_2$ $[4]$, where $0\leq n_1+n_2\leq 8p_g+21$ (S2F.5).
            \item[(vii)] One $[3,3p_g(S)-5,2,\ldots,2,3,2]$ with $(3p_g-10)$ $2$s in the middle, one $[3,3p_g-4,2,\ldots,2,3,2]$ with $(3p_g-9)$ $2$s in the middle, $(p_g-4+n_1)$ $[2,5]$, one $[3,5,2]$, and $n_2$ $[4]$ where $0\leq n_1+n_2\leq 8p_g+21$ (S2F.6).
            \item[(viii)] One $[4,3p_g-7,5,2,\ldots,2,3,2,2]$ with $(3p_g-10)$ $2$s in the middle, $(p_g-4)$ $[2,5]$, $n_1$ $[4]$, and $n_2$ $[6,2,2]$, where $0\leq n_1+n_2\leq 8p_g+23$ (S2F.7).
            \item[(ix)] One W-chain $[4,3p_g-6,2,\ldots,2,3,2,2]$ with $(3p_g-11)$ $2$s in the middle, one $[4,3p_g-4,2,\ldots,2,3,2,2]$ with $(3p_g-9)$ $2$s in the middle, $(p_g-4)$ $[2,5]$, $n_1$ $[4]$, and $n_2$ $[6,2,2]$, where $0\leq n_1+n_2\leq 8p_g+22$ (S2F.8).
        \end{itemize}
 \label{classHorikawa}  
\end{theorem}

\begin{proof}
By Corollary \ref{hori}, we have that $W$ is either a Lee-Park example, which is (i), or a small surface. Let us consider the case $W$ small. We have $K_W^2=p_g-2+N$ by Theorem \ref{formulaK^2}, and so $N=p_g-2$. (We also have that the number of T-chains $l$ satisfies $l \geq \text{max} \{1,p_g-3\}$.) Hence $p_g=3$ implies $N=1$, $p_g=4$ implies $N=2$, and $p_g\geq 5$ implies $N\geq 3$. Below we use the classification in Theorem \ref{ClassSmallSurf} considering each building block.

First of all we eliminate some potential building blocks. We recall that on $S$ we have $\Gamma^2=-(p_g+1)$. As the section in S1F.1 and S2F.9 has self-intersection $-3$ and we have $p_g\geq 3$, then it is impossible to use them. For S2F.1 and S2F.2 we have a section of self-intersection $-4$, and so not possible for any $p_g\geq 3$. For S1F.3 we have a section of self-intersection $-5$, and same analysis shows that it is impossible. Below we consider all other building blocks.

For $p_g=3$ we have $N=1$, and so we can have a building block S0F or S1F.j only. In the case of S1F.j we cannot have FIBs as $N=1$. Then the only possible are S1F.2 and S1F.4 for $r=4$. For S0F we can have one arbitrary large FIB from an $I_n$ (i.e. $n>>0$), which adds a $[5,2]$ and $(n-1)$ $[4]$s. Using the $\chi_{\text{top}}(S)$ (see Remark \ref{notation}), we obtain that $48\geq 3+n+1$, so $43 \geq n-1$. Similarly we obtain (iii) and (iv) for S1F.2 and S1F.4.

When $p_g=4$ we have $N=2$. Thus we can construct with S0F, S1F.i, and S2F.j, but in this last one we cannot use any FIBs. For S0F, we start with a $(-5)$-section and use two FIBs. Then we obtain (ii). For S1F.2 and S1F.4 we can use just one FIB, and we obtain (iii) and (iv). For the rest we have no FIB, one can check that (v)-(x) are the possibilities. The case S2F.3 is special here, since self-intersection of the section is $-5$.     
When $p_g\geq 5$ then we can go one-by-one obtaining what is claimed. 
\end{proof}

\begin{example}
For $K^2=8$, there are distinct Lee-Park examples with complex smoothings into the nonspin and spin components. In particular, this shows that rational blowdowns of Lee-Park examples may not be even homeomorphic. The first Lee-Park example is the content of \cite{LP11} when $p_g=6$. In particular, one begins with $\F_7$ and performs 4 blow-ups over a single fiber, so that we obtain the Wahl chain $[7,\bar{2},2,\bar{2}]$ where the bars indicate $(-1)$-curves intersecting the corresponding $(-2)$-curves. Lee and Park proved that the contraction $W_0$ of $[7,2,2,2]$ has $\Q$-Gorenstein smoothings, and it has a 2-divisible nonsingular irreducible curve $D_0$ which lifts in the smoothing, and so we have a double cover of the smoothing which is a $\Q$-Gorenstein smoothing of the double cover $W \to W_0$, where $W$ has two Wahl chains of type $[7,2,2,2]$. In Corollary \ref{LeeParknoSpin} we prove that the general fiber lives in the nonspin component of the moduli space. 

The second Lee-Park example and smoothing can be constructed as follows. Consider again $\F_7$ and one fiber, but now blow-up 3 times to obtain the Wahl chain $[7,2,\bar{2},2]$ where the bar indicates the $(-1)$-curve intersecting the corresponding $(-2)$-curve. The contraction $W_0'$ of $[7,2,2,2]$ is a Del Pezzo surface and has no local-to-global obstructions to deform. A $\Q$-Gorenstein smoothing is a degeneration of $\P^2$. In \cite[Theorem 1.10]{DVS24} it is proved that there are degenerations of $\P^2$ into $W_0'$ so that nonsingular degree $10$ curves specializes into a nonsingular curve $\Gamma_0$ in $W_0'$. We also have a line bundle $L$ on $W_0'$ so that $2L$ is isomorphic to $\O_{W_0'}(\Gamma_0)$ and this lifts to $\P^2$. Therefore, we have a double cover of the $\Q$-Gorenstein smoothing, producing a $W'$ with two Wahl singularities of type $[7,2,2,2]$ as degenerations of spin Horikawa surfaces. By our classification Theorem \ref{classHorikawa} we have that $W'$ is a Lee-Park example. The next theorem shows that this is the only one with T-singularities for spin surfaces.     
\label{spinK^2=8}
\end{example}

\begin{theorem}
There are no degenerations of spin Horikawa surfaces into KSBA surfaces with only T-singularities, except for $K^2=8$.
\label{nospin}
\end{theorem} 

\begin{proof}
Let $W_t$ be a one-parameter complex family of spin Horikawa surfaces degenerating to $W_0$ with $K_{W_0}$ ample and only T-singularities. As this is a KSBA degeneration, this is a $\Q$-Gorenstein smoothing of $W_0$. After possibly a base change, we can resolve simultaneously all Du Val singularities on $W_0$, and so we have a one-parameter complex family of $W_t$ degenerating to a $W$ with only non-Du Val T-singularities, $K_W$ big and nef, and $K_W^2=2p_g(W)-4$. By Theorem \ref{classHorikawa} it must be one of the surfaces in the list (i)-(ix) for $p_g\geq 5$, since spin surfaces only exist for $K^2$ divisible by $8$ (see \cite[VII Section 9]{BHPV04}).

For the purpose of the proof, we consider an M-resolution \cite{BC94} of the one-parameter complex deformation $W_t$, and so, on the central fiber $W$ we have only Wahl singularities. If $n_1,\ldots,n_{\ell}$ are the indices of these singularities and $M_1,\ldots,M_{\ell}$ the corresponding Milnor fibers, then we have the exact sequence $$\mathop{\oplus}\limits_{i=1}^{\ell} H_2(M_i,\Z)=0 \to H_2(W_t,\Z) \to H_2(W,\Z) \to  
\mathop{\oplus}\limits_{i=1}^{\ell} H_1(M_i,\Z) \to 0=H_1(W_t,\Z)$$ because $H_2(M_i)=0$ for $\Q$-Gorenstein smoothings, and Horikawa surfaces are simply-connected. Therefore we have an easy topological lifting of classes from $H_2(W,\Z)$ into $H_2(W_t,\Z)$. Moreover $K_W$ lifts to $K_{W_t}$.

For the cases (ii)-(ix), we always have at least one building block FIB. In this FIB we have a $(-1)$-curve $\Gamma$ intersecting the $(-5)$-curve in $[2,5]$ twice, or just once but also intersecting once a $[4]$ or the $(-3)$-curve in a $[3,2,\ldots,2,3]$. Then for its image in $W$ we have $\Gamma \cdot K_W =\frac{1}{3}$ in the first case, and $\frac{1}{6}$ in the second. Now, $3\Gamma$ or $6 \Gamma$ lifts to a class in $H_2(W_t,\Z)$. Therefore $K_{W_t}$ is not even, and so $W_t$ cannot be spin.

The case (i) corresponds to Lee-Park examples. In Example \ref{spinK^2=8} we show that $K^2=8$ has indeed degenerations into Lee-Park examples. Subsection \ref{s43} works out the $K^2>8$ using a characterization of the quotient family and some 3-fold birational geometry. We conclude the nonexistence of such degenerations in Corollary \ref{LeeParknoSpin}.


   
\end{proof}

\begin{remark}
Of course there is a similar statement for any small surface that uses a FIB. Together with Section \ref{topo}, this characterizes the topological type of simply-connected rational blowdowns of small surfaces, except (possibly) for the building block S2F.7. In particular, they are all homeomorphic to $$(2p_g+1)\C \P^2 \# (10p_g+9-K^2) \overline{\C \P^2}$$ depending only on $p_g$ and $K^2$. The diffeomorphism type remains unclear. By \cite[Theorem 2.3]{RU22}, the rational blowdown is a minimal symplectic 4-manifold, and so an exotic copy of $(2p_g+1)\C \P^2 \# (10p_g+9-K^2) \overline{\C \P^2}$.  
\label{exotic}     
\end{remark}

\begin{remark}
As we noticed in Remark \ref{wormhole}, the cases (viii) and (ix) are wormholes \cite{UV22}. Although they do not necessarily have complex deformations, it would be interesting to know if the rational blowdowns are homeomorphic. If so, what are the diffeomorphism types?
\end{remark}
  
\begin{remark}
In the case of Horikawa surfaces with $K^2=2$ (i.e. $p_g=3$), we know all possible degenerations into KSBA Horikawa surfaces with only T-singularities by \cite{ESU24}. In this case the moduli space is connected of dimension $36$, and the Horikawa surfaces are double covers of plane octics. The Gorenstein KSBA compactification was studied in \cite{A20}. According to Theorem \ref{classHorikawa}, we have $4$ families of surfaces $W$ with only T-singularities. We can check that for all of them the rational blowdown is simply-connected and has odd intersection form, and so in this case we have only one topological type $7\C \P^2 \# 37 \overline{\C \P^2}$. We have complex smoothings only for certain surfaces in (i) and (iv), described explicitly in \cite{ESU24}. Are the rational blowdowns of (ii) and (iii) diffeomorphic to $K^2=2$ Horikawa surface as well?  
\end{remark}
 

We finish proving the last two results from the introduction.

\begin{proof}[Proof of Theorem \ref{main5}]

By \cite[Proposition 3.1]{RU22} and because $p_g=2$, we have that $S$ has Kodaira dimension $1$, and we have an elliptic fibration $S \to \P^1$. Let us consider the notation in the proof of Corollary \ref{low}. By replacing $K^2=1$ and $p_g=2$, we have  $$ \Big(\sum_{i\geq 1} (i-1)d_i \Big) + \Big(\sum_{i\geq 1} id_i \Big) \Sigma \leq 1.$$ If we have multiple sections, then $\Sigma=0$. If we have only sections, then $\Sigma=0$. So we have $\sum_{i\geq 1} (i-1)d_i \leq 1$. If there is a multiple section, then it must be one double section and we have no blow-ups, i.e. $X=S$. As $K_W^2=1$, we cannot have any other curve in the T-chain, so we obtain (i).

We now have only sections. We note that the previous inequality is not enough to bound sections. We use now Corollary \ref{euler}: $K_S\cdot \pi(C) -l_c \leq K_W^2 - K_S^2$. Evaluating we have $s-l_c\leq 1$, where $s$ is the number of sections in $\pi(C)$. They are $(-3)$-curves. We know that each of the $l_c$ components has at least one section (Proposition \ref{koda}). 

Assume $s\geq 2$ and that there is a component with just one section. As there is another section, over this component on $X$ we would have the P-resolution of $[2,\ldots,2,3,2,\dots,2]$, but Proposition \ref{2r2} shows that this is impossible. Then there is one component that contains $s=2$ sections. By Corollary \ref{euler} the graph $G$ is a tree and $S_h=T_h$ for $h \geq 3$. Therefore, these two sections either intersect transversally at one point or they are disjoint. If they do intersect, then we must have a chain of $(-2)$-curves between them in the image of $\pi(C)$ (i.e. no other configuration between them in $\pi(C)$ is allowed). Again, as there is no cycle, we must have a P-resolution of $[2,\ldots,2,3,2,\ldots,2,3,2,\ldots,2]$ in $X$. Its dual fraction is $[a+1,b+1,c+1]$ where $a,b,c$ are the consecutive numbers of $2$s. But there are only two zero continued fractions of length $3$: $[2,1,2]$ and $[1,2,1]$. The second one contracts all the $2$s from its minimal resolution, so it does not work. The first one contracts $[3,2,\ldots,2,3]$. Therefore there are no blow-ups, and the T-chain is $[3,2,\ldots,2,3]$. This is (ii).    

We are now in the case $s=1$, i.e. a small surface. A simple inspection on their classification Theorem \ref{ClassSmallSurf} gives (iii) and (iv). We have that $2p_g-3=p_g-2+N$, and so $N=p_g-1=1$. Therefore only S0F or S1F.j are possible, and with $j=1,4$.
\end{proof}

\begin{proof}[Proof of Corollary \ref{main6}] This is just a simple inspection using Theorem \ref{ClassSmallSurf}. We note that $N=3$ and the initial self-intersection of the section is $-5$. 
\end{proof}

\section{Complex smoothings} \label{s4}

In this section, we will prove that Horikawa T-surfaces may have $\Q$-Gorenstein smoothings only for $p_g<10$, except for the Lee-Park examples. This is very peculiar as various moduli spaces of simply-connected surfaces of general type have tons of degenerations into T-surfaces (see e.g. \cite{RU21,RU22}). On the other hand we prove that Lee-Park examples are degenerations in only one of the two components.

The strategy splits into three parts: study all possible quotients of c.q.s. by an involution, the use that deformations of canonical models with an involution extend the involution to the KSBA limit, and the study of the quotient family which turns out to be again $\Q$-Gorenstein but with a limited amount of non-Du Val T-singularities. Key is our classification of small surfaces, and in particular the property that they have many singularities. 

Finally, to rule out the possibility of other degenerations of Lee-Park examples (they have only two singularities), we use our analysis of involutions on Wahl singularities, and the machinery of birational geometry of degenerations. As a summary, the quotient family admits a flip that allows us to prove that Lee-Park examples live in just one of the two components (of course when there are two).  

Some notation. We denote a nonsingular curve germ by $\D$. A deformation of a surface $W$ over $\D$ will be denoted by $(W \subset \W) \to (0 \in \D)$ or $W_t \rightsquigarrow W$, where $W_t$ are the fibers over $t\neq 0$. A deformation is a smoothing if $W_t$ is nonsingular. Let $W$ be a normal projective surface. The KSBA moduli space parametrizes $W$ with $K_W$ ample and only log canonical singularities, and allows just $\Q$-Gorenstein deformations of $W$ \cite{KSB88}. Following \cite[Definition 3.1]{H04}, we say that $(W \subset \W) \to (0 \in \D)$ is a \textit{$\Q$-Gorenstein deformation} of $W$ if $W$ has log canonical singularities, and locally at each of them this deformation is the quotient of a $\Z/n$-equivariant deformation of its canonical cover, where $n$ is the index of the singularity. Koll\'ar--Shepherd-Barron \cite{KSB88} adopted another definition, but for deformations over $\D$ it is just the condition $K_{\W}$ $\Q$-Cartier. In our case both definitions coincide (see \cite[Lemma 3.4]{H04}).  

\subsection{Involutions on singularities} \label{s41}

\begin{theorem}
Let $0<q<m$ be coprime integers. Assume that $\frac{1}{m}(1,q)$ admits an involution $\tau$ such that $\frac{1}{m}(1,q)/\langle \tau \rangle$ is either a c.q.s. $\frac{1}{m_0}(1,q_0)$ or a nonsingular point. Then the options for the quotient are:

\begin{itemize}
    \item[(a)] A nonsingular point. This can only happen when $q=m-1$. 
    \item[(b)] $m_0=2m$ and $q_0=q,m+q$.
    \item[(c)] $m_0=m$ and $q_0=2q$.
    \item[(d)] $m_0=\frac{m}{2}$ and $q_0=q$.
    \item[(e)] $m_0=\frac{m}{\text{gcd}(q+1,m)}$ and $q_0=\frac{q+1}{\text{gcd}(q+1,m)}$, only when $q^2 \equiv 1($mod $m)$, $q \neq m-1$. 
\end{itemize}
\label{invcqs}
\end{theorem}

\begin{proof}
The set-up of the proof follows closely \cite[Section 1]{C87}. Let $\C^2 \to \C^2/G$ be the quotient map defining $\frac{1}{m}(1,q)$, where $G=\Z/m$. We have an involution $\tau$ acting on $\frac{1}{m}(1,q)$. As in \cite[Section 1]{C87}, we consider the group $\Gamma$ of liftings of $\langle \tau \rangle$. They are automorphisms of $(\C^2,0)$, and fit in a exact sequence (that may or may not split) $$ 1 \to G \to \Gamma \to \langle \tau \rangle \to 1.$$ Thus $|\Gamma|=2m$. By \cite[Cartan's Lemma]{C87}, we can assume that $\Gamma$ acts linearly on $(\C^2,0)$. 

Suppose that $\Gamma$ has no pseudo-reflections. Then $\Gamma$ is one of the known finite subgroups of GL$(2,\C)$ with no pseudo-reflections (see e.g. \cite[Section 2]{PPSU18}), and so it must be $\Z/2m$. Thus we must have that $\C^2/\Gamma$ is $\frac{1}{2m}(1,q)$ or $\frac{1}{2m}(1,q+m)$. This is case (b). 

Suppose now that there is some pseudo-reflection $\gamma \in \Gamma$. As $G \cap \langle \gamma \rangle =1$, we have that $\gamma^2=1$, and so a reflection, $\gamma$ is a lifting of $\tau$, and that $\Gamma$ is a semi-direct product between $G$ and $\gamma$. Let us assume that $G$ acts on $\C^2$ via $\theta= \left[\begin{array}{cc} \zeta_m & 0 \\ 0 & \zeta_m^q \end{array}\right]$, and let $\gamma= \left[\begin{array}{cc} a & b \\ c & d \end{array}\right]$. As $\gamma^2=1$, we have $a^2+bc=1$, $d^2+bc=1$, $c(a+d)=0$, and $b(a+d)=0$. If $a+d\neq 0$, then we get $\gamma=\pm 1$ which is absurd. Therefore we assume $a=-d$.

We must have that $\gamma \theta \gamma =\theta^j$ for some $j$. This translates into equations: $\zeta a^2+ \zeta^q bc=\zeta^j$, $\zeta^q a^2+ \zeta bc=\zeta^{qj}$, $(\zeta-\zeta^q)ab=0$, and $(\zeta-\zeta^q)ac=0$.

Say that $\zeta=\zeta^q$. Then $\zeta=\zeta^j$ and so $\Gamma$ is abelian. We can change coordinates to have $\gamma= \left[\begin{array}{cc} a & 1-a \\ 1+a & -a \end{array}\right]$, and the ring $\C[x,y]^{\langle \gamma \rangle}$ is generated by $(a+1)x+(1-a)y$ and $(x-y)^2$. Therefore, the quotient $\C^2/\langle \gamma \rangle =\C^2$, and the additional quotient by $G$ is $\frac{1}{m}(1,2)$. This is (c) with $q=1$, or (a) if $m=2$.

Say that $\zeta \neq \zeta^q$ and $a=0$. Then we have $\zeta^q=\zeta^j$ and $\zeta=\zeta^{qj}$, and so $q^2\equiv 1(\text{mod} \ m)$. If $q\equiv -1(\text{mod} \ m)$, then $\Gamma= \langle  \left[\begin{array}{cc} \zeta_m & 0 \\ 0 & \zeta_m^{-1} \end{array}\right],  \left[\begin{array}{cc} 0 & 1/b \\ b & 0 \end{array}\right]\rangle$, and so it is generated by reflections. This is (a). Assume $u=\text{gcd}(q+1,m) \neq m$, and let $q+1=ut_1$. Let $v=m/u$, and so $\text{gcd}(v,t_1)=1$. Since $q^2\equiv 1(\text{mod} \ m)$, we have $q-1=vt_2$ for some $v,t_2$. As $\Gamma= \langle  \theta= \left[\begin{array}{cc} \zeta_m & 0 \\ 0 & \zeta_m^{q} \end{array}\right],  \gamma= \left[\begin{array}{cc} 0 & 1/b \\ b & 0 \end{array}\right]\rangle$, one can check that the reflections generate the normal subgroup  $\Gamma'= \langle  \theta^v,  \gamma \rangle$. Then $\C^2/ \Gamma'\simeq \C^2$. Generators of the invariant algebra are $b^u x^u+y^u$ and $xy$. Then the quotient of $\C^2/ \Gamma'$ by the cyclic group $\Gamma/\Gamma'$ is $\frac{1}{v}(1,t_1)$. By computing the action of $\theta$ in $b^u x^u+y^u$ and $xy$. 




Say that $\zeta \neq \zeta^q$ and $a \neq 0$. Then $b=c=0$, and so $\gamma=\left[\begin{array}{cc} 1 & 0 \\ 0 & -1 \end{array}\right]$ or $\left[\begin{array}{cc} -1 & 0 \\ 0 & 1 \end{array}\right]$ and $j=1$ (so abelian). If there is no other reflection, then $m$ is odd and the quotient by $\Gamma$ is just as was done for (c) $q=1$. We have the rest of case (c). If $m$ is even, then we have the extra reflection $-\gamma$ as $-1 \in G$. Then we quotient by $\Z/2 \times \Z/2$ to obtain $\C^2$. It is explicitly given by $\C^2 \to \C^2$, $(x,y) \mapsto (x^2,y^2)$. We then have case (d).    
\end{proof}

\begin{corollary}
The quotient of a Wahl singularity by an involution is never a Wahl singularity (or nonsingular point).
\label{invWahl}
\end{corollary}

\begin{proof}
Let $\frac{1}{n^2}(1,na-1)$ be a Wahl singularity whose quotient is a Wahl singularity. By Theorem \ref{invcqs} we have four possibilities. It cannot be nonsingular because we would need $na-1=n^2-1$. For (b) and (d) we have that $m_0=2n^2$ and $m_0=n^2/2$ cannot be squares. For (c) we would need $2(na-1)=na'-1+n^2t$ for some $t$, and so $n=1$, a contradiction. Finally case (e) is not possible for Wahl singularities.
\end{proof}








\begin{corollary}
If a non-Du Val T-singularity $(P \in W)$ admits an involution $\tau$ so that the quotient is $\frac{1}{m_0}(1,1)$, then $(P \in W)=\frac{1}{4}(1,1)$ and $m_0=8$, $(P \in W)=\frac{1}{9}(1,2)$ and $m_0=9$, $(P \in W)=\frac{1}{9}(1,5)$ and $m_0=9$, or $(P \in W)=\frac{1}{4d}(1,2d-1)$ and $m_0=2$.

    
    
\label{quotTo1/m(1,1)}
\end{corollary}

\begin{proof}
This is a simple inspection via Theorem \ref{invcqs}.
\end{proof}

\subsection{Involution on a KSBA family} \label{s42}

The following is a well-known fact, which fits nicely in our context. For a proof see \cite[Section 2.4]{FPRR22}.

\begin{proposition}
Let $W$ be a KSBA surface with a 1-parameter $\Q$-Gorenstein deformation $(W \subset \W) \to (0 \in \D)$. Assume that we have a compatible involution over $\D \setminus 0$, i.e. a section of Aut$\big((\W \setminus W)|_{\D \setminus \{0\}} \big)$ of order two. Then, up to possibly a base change, it extends to an involution of $\W$ over $\D$.
\label{extendInv}
\end{proposition}





We note that a base change of a 1-parameter $\Q$-Gorenstein smoothing of a surface with only T-singularities is again a 1-parameter $\Q$-Gorenstein smoothing of a surface with only T-singularities (see e.g. \cite{HTU17}). Let $W$ be a surface with only quotient singularities $(P_i \in W)$. Assume there is a smoothing $(W \subset \W) \to (0 \in \D)$. Let $M_i$ be the Milnor fiber of $(P_i \in W)$ corresponding to the smoothing. We know that $\mu_i:=b_2(M_i)$ is the only non-necessarily trivial Betti number of $M_i$. It is called the Milnor number of the smoothing. A quotient singularity with $\mu=0$ is a Wahl singularity. By \cite[Section 3]{KSB88}, the deformations of quotient singularities are in one-to-one correspondence with their P-resolutions, which are partial resolutions with only T-singularities and ample relative canonical class. The smoothing is a blow-down of a $\Q$-Gorenstein smoothing of the P-resolution. If the P-resolution corresponding to $\mu_i$ contains $r_i$ curves and the Milnor numbers of the T-singularities are $\nu_{i,j}$, then $\mu_i=r_i+\sum_j \nu_{i,j}$. 

\begin{corollary}
Let $W$ be a surface with only T-singularities, $K_W$ ample, and $K_W^2=2p_g(X)-4$ with $p_g \geq 7$. Assume that we have a $\Q$-Gorenstein smoothing $(W \subset \W) \to (0 \in \D)$. Then there is an involution on the family, whose quotient is a smoothing $(W_0 \subset \W_0) \to (0 \in \D)$, where $W_0$ has only quotient singularities, and the general fiber is a Hirzebruch surface. Moreover, the smoothing $(W_0 \subset \W_0) \to (0 \in \D)$ is $\Q$-Gorenstein at every singularity and they are Wahl singularities, except by potentially one singularity $P \in W_0$, whose smoothing must have Milnor number equal to $1$. The corresponding P-resolution over $(P \in W_0)$ has no singularities if and only if $(P \in W_0)$ is a c.q.s. of the form $\frac{1}{m}(1,1)$ for some $m$.
\label{quotDegHori}    
\end{corollary}

\begin{proof}
By Horikawa's classification \cite{H76}, since $p_g\geq 7$ and $K_{W_t}$ is ample, the canonical map of the Horikawa surface $W_t$ is a finite double cover onto a Hirzebruch surface $\F_d$. As $W$ is a KSBA surface, this involution extends to the whole $\Q$-Gorenstein smoothing by Proposition \ref{extendInv}, and we have a quotient $(W_0 \subset \W_0) \to (0 \in \D)$, where the central fiber $W_0$ has only quotient singularities. This smoothing has Milnor numbers $\mu_i$ and corresponding $P$-resolutions at each quotient singularity. Consider the partial resolution $W'_0 \to W_0$ which is the P-resolution at each singularity. After possibly a base change, we have by \cite[Theorem 3.5]{KSB88} a $\Q$-Gorenstein smoothing $W'_t \rightsquigarrow W'_0$ where $W'_t$ is the general fiber of $(W_0 \subset \W_0) \to (0 \in \D)$, and so a Hirzebruch surface. By \cite[Proposition 2.6]{HP10}, we have $\rho(W'_0)+\sum_j \nu_j=2$, where $\nu_j$ are the Milnor numbers of the smoothed singularities in $W'_t \rightsquigarrow W'_0$. As $\rho(W_0) \geq 1$, we have that either $W'_0=W_0$ and so the initial smoothing was already $\Q$-Gorenstein, or $\rho(W_0)=1$ and only one singularity has a P-resolution with one curve and only Wahl singularities. The smoothing of that singularity has Milnor number equal to $1$. Finally, the only quotient singularity that has no singularities over a P-resolution and Milnor number equal to $1$ is $\frac{1}{m}(1,1)$ for some $m$.
\end{proof}

The assumption on $p_g$ was only to avoid $\P_{\C}^2$ and cones over rational curves as images of Horikawa surfaces $X$ under $K_X$. The next theorem and proof is due to Juan Pablo Z\'u\~niga.

\begin{theorem}
Let $\F_n \rightsquigarrow W_0$ be a $\Q$-Gorenstein smoothing of a surface $W_0$ with only T-singularities. Then $W_0$ has at most $4$ singularities. Moreover there can be at most one non-Wahl T-singularity: either $A_1$ or $\frac{1}{2n^2}(1,2na-1)$ for some coprime $0<a<n$. 
\label{JPZ}
\end{theorem}

\begin{proof}
Since $W_0$ is a normal projective degeneration of rational surfaces and it has only rational singularities, then $W_0$ is a rational surface. Moreover, since $\F_n \rightsquigarrow W_0$ is a $\Q$-Gorenstein smoothing, then $K_{W_0}^2+\rho(W_0)+\sum_{P} \mu_P=10$, where $\mu(P)$ are the Milnor numbers of the T-singularities. We have $K_{W_0}^2=8$, and so $\rho(W_0)+\sum_{P} \mu_P=2$. Therefore $\rho(W_0)=1$ and there is at most one non-Wahl T-singularity, or $\rho(W_0)=2$ and there are only Wahl singularities. In the first case we can consider an M-resolution of the unique T-singularity together with its $\Q$-Gorenstein smoothing \cite{BC94}. From now on we assume that $W_0$ has $\rho(W_0)=2$, and so only Wahl singularities. At then end we will conclude that we have at most $4$ singularities. 

We now use strongly \cite{M01} to control the geometry of $W_0$. Let $X \to W_0$ be its minimal resolution. We have $h^0(-K_{X})=h^0(-K_{W_0}) \geq h^0(-K_{F_n})=9$ (see \cite[proof of Theorem 4]{M01}). Let $X \to \F_e$ be the composition of blow-downs into a Hirzebruch surface with the maximum possible $e$ (called weight of $X$ in \cite{M01}). Then we have an induced fibration $X \to \F_e \to \P^1$ whose general fiber is $\P^1$. Since we satisfy $h^0(-K_X)+\text{min}\{e,3\} \geq 8$, there cannot be an irreducible curve $\Gamma$ in $X$ not contained in a fiber of $X \to \P^1$ such that $\Gamma^2 \leq -2$ except for $\sigma_{\infty}$, which is the proper transform of the negative curve in $\F_e$. Therefore, the divisor $E$ of $X \to W_0$ consists of curves in fibers and possibly $\sigma_{\infty}$. Let $F_1,\ldots,F_h$ be the singular fibers containing curves in $E$. Let $r_i+1$ be the number of components in $F_i$, and let $c_i$ be tha number of components in $F_i$ contracted by $X \to W_0$.  

First we note that $E$ cannot be formed only by components of fibers. This is because in that case we have $\rho(W_0)=2=2 + \sum_{i=1}^h (r_i -c_i)$, and so $r_i=c_i$ for all $i$. Then at each fiber $F_i$ we do not contract exactly $1$ curve, which must be a $(-1)$-curve. As $F_i$ must become a fiber in $\F_e$, it is easy to check that this is equivalent to have a Wahl chain equal to a dual Wahl chain, and that cannot be. Then the exceptional divisor $E$ is formed by $\sigma_{\infty}$ and components in the fibers $F_1,\ldots,F_h$. We will prove that $h \leq 2$. Just as before, we have $\rho(W_0)=2=2 + \sum_{i=1}^h (r_i -c_i)-1$, and so $r_i=c_i$ for all $i$ except some $i_0$ where $r_{i_0}-c_{i_0}=1$. Assume $h\geq 3$. As $\sigma_{\infty}$ must be part of a Wahl chain, there must be a singular fiber, say $F_3$, intersecting $\sigma_{\infty}$ through a non exceptional curve of $X \to W_0$. Therefore $i_0=3$. Using that for a Wahl chain $[e_1,\ldots,e_r]$ we have $\sum_{i=1}^r e_i=3r+1$, we can check that $F_3$ cannot contain one(two) Wahl chain(s) and two $(-1)$-curves. Therefore $h \leq 2$. If $h=1$, then there are two curves in $F_1$ not contracted, and so at most we have $3$ singularities. If $h=2$, one sees that at most we can construct $4$ chains.

To finish, we need to deal with the potential $A_1$ that we resolved. If $h\leq 2$, then at most four chains including the $A_1$ resolved. If $h\geq 3$, then $h=3$ and $i_0=3$, since there is at most one $A_1$. And by the previous argument, this $A_1$ resolved must be in $F_3$. But the same argument on self-intersections shows that $F_3=[1,2,1]$, and then from the other two fibers we get at most $3$ singularities, so $4$ in total. 
\end{proof}

\subsection{Lee-Park examples revisited} \label{s43}

Let $p_g\geq 7$. Consider a $W$ of Lee-Park type (see Example \ref{leepark}). Let $S \to \P^1$ be an elliptic fibration with sections of self-intersection $-(p_g+1)$. Consider two disjoint Wahl chains $[p_g+1,2,\ldots,2]$. We construct $W$ by contracting them. Let $W \to W'$ be the canonical model, and so $W'$ has the previous two Wahl singularities, and possibly some Du Val singularities. (It cannot acquire new non-Du Val T-singularities because there are no $(-1)$-curves in $S$.) Assume that we have a $\Q$-Gorenstein smoothing $(W' \subset \W') \to (0 \in \D)$, and so the general fiber $W'$ is a Horikawa surface with ample canonical class. By Corollary \ref{quotDegHori}, there is an involution on $(W' \subset \W') \to (0 \in \D)$ so that its quotient $(W_0' \subset \W_0') \to (0 \in \D)$ has $W_0'$ with one Wahl singularity (from the orbit of two) and some extra $A_1$, or it has two singularities, and one of them is Wahl. But, by Corollary \ref{invWahl}, this second case is not possible. Therefore, the quotient is $\Q$-Gorenstein. The general fibers are Hirzebruch surfaces $\F_n$ for some $n$.


\begin{proposition}
The surface $W_0'$ has exactly one singularity and it is $[p_g+1,2,\ldots,2]$. The minimal resolution of $W_0'$ is the composition of $p_g-2$ blow-up over a fiber, representing the zero continued fraction $[1,2,\ldots,2,1]$.
\label{preLeePark}
\end{proposition}

\begin{proof}


As we said, in the quotient we must have $W_0'$ with a single Wahl singularity $[p_g+1,2,\ldots,2]$, and maybe one $A_1$ singularity. By the same Manetti's argument in the proof of Theorem \ref{JPZ}, we have that the minimal resolution of $W_0'$ has a birational map to $\F_e$ with $e$ maximal, and the exceptional divisor $E$ is contained in one or two fibers, and contains the proper transform of the negative section $\sigma_{\infty}$. Because of the $2$s in $[p_g+1,2,\ldots,2]$, one can check that it is impossible for two fibers. Thus we have only one degenerated fiber for $X \to \P^1$. The $-(p_g+1)$-curve must be $\sigma_{\infty}$, and the only positions for the two $(-1)$-curves in this fiber is $[1,2,\ldots,2,1]$. Notice that there is no room for an extra $(-2)$-curve, and so the case of one $A_1$ singularity is not possible. 
\end{proof}

\begin{theorem}
Let $W$ be a Lee-Park example and suppose that $K_W^2=8k$ for some $k>1$. Assume that $W$ admits a $\Q$-Gorenstein smoothing. Then the corresponding nonsingular fibers live in the component of the moduli space whose general member has a canonical map onto $\F_0$. 
\label{LeePark}
\end{theorem}

\begin{proof}
If $W$ admits a $\Q$-Gorenstein smoothing, then its canonical model admits a $\Q$-Gorenstein smoothing as well, since general Horikawa surfaces have ample canonical class. Then we apply Proposition \ref{preLeePark} to construct a $\Q$-Gorenstein smoothing $(W_0' \subset \W_0') \to (0 \in \D)$, quotient of the canonical family, whose $W_0'$ has one Wahl singularity $[p_g+1,2,\ldots,2]$ and the general fiber is $\F_n$ for some $n$. From the minimal resolution of $W_0'$, we obtain a $(-1)$-curve that intersects the $(-2)$-curve marked with a bar in $[p_g+1,\bar{2},\ldots,2]$. Let $\Gamma^{-}$ be its image in $W_0'$. At this moment, we can perform a flip of k1A type on $(W_0' \subset \W_0') \to (0 \in \D)$, with flipping curve $\Gamma^{-}$. Our reference is \cite{HTU17}. (This is because $p_g=4k-2$ is even, see Remark \ref{DivContrLeePark} for Lee-Park examples with odd $p_g$.) This happens over a smoothing $(Y \subset \Y) \to (0 \in \D)$, where the birational morphism $\Gamma^{-} \subset W_0' \subset \W_0' \to Q \in Y \subset \Y$ is an extremal nbhd. The point $Q \in Y$ is the c.q.s. $\frac{1}{4}(1,1)$. (This is because $[2,\ldots,2,1,1+p_g]$ contracts to $[4]$.) Therefore, the flip $(W_0^+ \subset \W_0^+) \to (0 \in \D)$ has $W_0^+=\F_4$ and the flipped curve is the $(-4)$-curve. Thus the general fiber, which is the same as the general fiber of $\W_0' \to \D$, is $\F_n$ with $n\leq 4$ (and even). In this way, we must be in the component of the moduli space whose general member has a canonical map onto $\F_0$.     
\end{proof}

\begin{corollary}
For $p_g \geq 7$, a $\Q$-Gorenstein smoothing of a Lee-Park example is never spin.
\label{LeeParknoSpin}
\end{corollary}

\begin{proof}
The general Horikawa surface in the spin component has as canonical map a double cover of $\F_{\frac{p_g+2}{2}}$ \cite{H76}, and we just proved that it belongs to the other component. 
\end{proof}

\begin{remark}
Theorem \ref{LeePark} says that the only possible way to $\Q$-Gorenstein smooth a Lee-Park example is via a construction as in \cite{LP11}. Indeed, we have proved that $W_0'$ has a single singularity, and the branch divisor of the double cover $W' \to \W_0'$ does not pass through this singularity. 
\end{remark}

\begin{remark}
Of course Lee-Park examples and $\Q$-Gorenstein smoothings are also possible when $K^2 \neq 8k$ \cite{LP11}. If $p_g$ is even, we obtain that both $(-1)$-curves from $X$ become flipping curves in   
$(W_0' \subset \W_0') \to (0 \in \D)$. We recall that both curves are the ones touching the $(-2)$-curves marked by a bar in $[p_g+1,\bar{2},2,\ldots,2,\bar{2}]$. In the case of odd $p_g$, we obtain that the first curve is a k1A of divisorial type, and so, automatically, the general fiber is an $\F_1$, and we have a contraction to the classical degeneration of $\P^2$ into $\overline{\F}_4$, the contraction of the $(-4)$-curve in $\F_4$. The other $(-1)$-curve in $X$ becomes of flipping type in $W$. 
\label{DivContrLeePark}
\end{remark}

\subsection{Zone for smoothable Horikawa T-surfaces} \label{s44}

We will show soon a result about smoothability of Horikawa T-surfaces for $p_g\geq 10$. So, to be uniform in the arguments, we will assume from now on $p_g\geq 7$. In this case the canonical map has as image a Hirzebruch surface $\F_d$ for some $d$. If $K^2$ is not divisible by $8$, then the general Horikawa is a double cover of $\F_0$ or $\F_1$ branched along a nonsingular and irreducible curve. Otherwise, there are two components, and the general Horikawa is a double cover of either $\F_0$ branch along a nonsingular and irreducible curve, or $\F_{\frac{p_g+2}{2}}$ branch along two nonsingular and irreducible disjoint curves. The canonical class in both cases is ample as it is the pull-back of an ample class by a finite cover (see \cite{H76}, \cite[p.296]{BHPV04} for particular descriptions). Therefore general Horikawa surfaces have ample canonical class. 

\begin{theorem}
Let $p_g \geq 10$. Then the only KSBA degenerations of Horikawa surfaces with only T-singularities (not all Du Val) are Lee-Park examples, and only for the nonspin component.
\label{NOsmoothable}
\end{theorem}

\begin{proof}

Let $W'$ be a Horikawa surface with only T-singularities and ample $K_{W'}$ in the closure of the corresponding Gieseker moduli space of Horikawa surfaces for a given $K^2 \geq 16$. As the general nonsingular Horikawa surface has ample canonical class, we can assume we have a $\Q$-Gorenstein smoothing $(W' \subset \W') \to (0 \in \D)$. By Corollary \ref{quotDegHori}, we have the existence of an involution on $(W' \subset \W') \to (0 \in \D)$ and the corresponding quotient $(W_0' \subset \W_0') \to (0 \in \D)$ where $W_0'$ has only Wahl singularities except maybe by one quotient singularity whose smoothing has Milnor number equal to $1$. We replace this smoothing by the P-resolution over that singularity, denoted by $W'_t \rightsquigarrow W_0$. It will be key in our argument below the number of singularities of $W_0$, which must be at most $4$ by Theorem \ref{JPZ}. By Theorem \ref{classHorikawa}, we have a list that classifies the M-resolution $W$ of $W'$. From that list we deduce that $W'$ may have Du Val singularities coming from fibers of the corresponding elliptic fibration, for exactly the cases (iv), (v) and (vi) we have T-singularities of type $\frac{1}{dn^2}(1,dna-1)$ with $d>1$ (for each case only one such singularity and with $d=2$), and for the rest we have Wahl singularities. We already analyzed the Lee-Park examples in Corollary \ref{LeeParknoSpin}. We now go case by case checking that no small surface in the list in Theorem \ref{classHorikawa} has a quotient into $W_0'$.



For (ii), we have on $W'$ at least one Wahl singularity $[3p_g-3,2,\ldots,2]$ and $p_g-2$ Wahl singularities $[2,5]$. These $p_g-2 \geq 8$ singularities could be in orbits of two, or the involution may fix some. As the quotient of $[3p_g-3,2,\ldots,2]$ is not Wahl or nonsingular point (Corollary \ref{invWahl}), and has not the form $1/m(1,1)$ (Corollary \ref{quotTo1/m(1,1)}), then it has one singularity at least in its P-resolution. Hence $p_g-2$ Wahl singularities $[2,5]$ must be in orbits. So we have at least $1+\frac{p_g-2}{2} \geq 5$ singularities in $W_0$, a contradiction with Theorem \ref{JPZ}. The case (iii) can be handled as for case (ii), producing a contradiction. In case (iv) we have $W'$ with at least a $d=2$ T-singularity, and $p_g-3$ Wahl singularities of type $[2,5]$. The same argument as above works except for the case $p_g=10$, where we can fix one $[2,5]$ and the other $3$ are in an orbit. By Corollary \ref{quotTo1/m(1,1)}, we have that the quotient of $[2,5]$ singularity must be $\frac{1}{9}(1,1)$. Therefore at least we have $4$ singularities of order $9$ in a rational surface $W'_0$ with $\rho(W'_0)=1$, and this contradicts the orbifold topological Euler characteristic inequality $\sum_{m_p} \frac{m_p-1}{m_p} \leq 3$; see e.g. \cite[Section 1]{K08}. The cases (v), (vi), (vii) and (ix) have $p_g-4$ $[2,5]$ singularities, one Wahl singularity which is not in Corollary \ref{quotTo1/m(1,1)}, and either one $d=2$ T-singularity or two Wahl singularities at least in $W'$. Then there is enough room to argue as in cases (ii) and (iii). Th case (viii) is the closest to be realizable. The surface $W'$ has at least one Wahl singularity (no $[2,5]$) and $(p_g-4)$ $[2,5]$. If $p_g \geq 11$, then we can argue as in the previous cases for the nonexistence of such smoothing. In the case $p_g=10$, we have one Wahl singularity $(P \in W')$ (no $[2,5]$) and $6$ $[2,5]$. The $(P \in W')$ must be fixed by the involution, and so the others are in orbits. Say that the quotient of $(P \in W')$, which is not nonsingular by Corollary \ref{invWahl}, has order $m$. Then by the orbifold topological Euler characteristic inequality, we have $\frac{m-1}{m} \leq \frac{1}{3}$, and so $m \leq \frac{3}{2}$, a contradiction.

\end{proof}

\appsection{Relevant c.q.s. and their P-resolutions} \label{app}

Let $0<\Omega < \Delta$ be coprime integers. In \cite{KSB88}, Koll\'ar and Shepherd-Barron found a bijection between irreducible components of the deformation space of $\frac{1}{\Delta}(1,\Omega)$ and P-resolutions of $\frac{1}{\Delta}(1,\Omega)$ (Definition \ref{pres}). The geometric idea is that any deformation of $\frac{1}{\Delta}(1,\Omega)$ is the blow-down of a $\Q$-Gorenstein deformation of a P-resolution of $\frac{1}{\Delta}(1,\Omega)$. Soon after, Christophersen \cite{C91} and Stevens \cite{S91} gave a combinatorial way to find all P-resolutions. If $$\frac{\Delta}{\Delta-\Omega}=[b_1,\ldots,b_s],$$ then they prove that P-resolutions are in bijection with the set $$ K(\Delta/\Delta-\Omega) = \{ [k_1,\ldots,k_s]=0 \
\text{such that} \ 1 \leq k_i \leq b_i \}.$$ 

We recall that a \textit{zero continued fraction} is a Hirzebruch-Jung continued fraction whose value is equal to $0$. Thus $ K(\Delta/\Delta-\Omega)$ is the set of zero continued fractions bounded by $[b_1,\ldots,b_s]$. We will say that the elements of $K(\Delta/\Delta-\Omega)$ are the zero continued fractions associated to the H-J continued fraction of $\frac{\Delta}{\Omega}$. The set of all zero continued fractions can be listed as follows $[1,1]$, $[2,1,2]$, $[1,2,1]$, $[1,3,1,2]$, $[3,1,2,2]$, $[2,2,1,3]$, $[2,1,3,1]$, $[1,2,2,1]$, $\ldots$, where from one to the next we apply the ``arithmetic blowing-up" identity $$u- \frac{1}{v} = u+1 - \frac{1}{1-\frac{1}{v+1}}.$$ This list of zero continued fractions is in bijection with triangulations of polygons \cite{C91,S91,HTU17}. A \textit{triangulation of a convex polygon} $V_0V_1 \dots V_s$ is given by drawing some non intersecting diagonals on it which divide the polygon into triangles. For a fixed triangulation, one defines $v_i$ as the number of triangles that have $V_i$ as one of its vertices. Note that $v_0+v_1+\ldots+v_s = 3(s-1)$. The number of zero continued fractions of length $s$ is the \textit{Catalan number} $\frac{1}{s}\binom{2(s-1)}{s-1}$.

\begin{notation}
Suppose that the P-resolution $W^+ \to \overline{W}$ corresponds to a zero continued fraction $[k_1,\ldots,k_s]$. Let $d_j:=b_j-k_j \geq 0$. Let~ $d_{j_0},\ldots,d_{j_{\ell}}$ be the set of nonzero $d_i$ with $j_0<j_2<\ldots<j_{\ell}$. Then the surface $W^+$ contains curves $\Gamma_1, \allowbreak  \ldots, \allowbreak  \Gamma_{\ell}$ and T-singularities at $P_0,\ldots,P_{\ell}$ of type $\frac{1}{d_{j_i}n_i^2}(1,d_{j_i} n_i a_i -1)$. Here we include $n_i=a_i=1$ which corresponds to an $A_{d_{j_i}-1}$ singularity for $d_{j_i} \geq 2$, and $d_{j_i}=n_i=a_i=1$ which corresponds to a nonsingular point. We have that $P_i, P_{i+1}$ belong to $\Gamma_{i+1}$, and $\Gamma_i, \Gamma_{i+1}$ form a toric boundary at $P_i$ for all $i$. In the case of $i=0$ or $i=\ell$, we have only one part of the toric boundary. We always order the exceptional curves of the minimal resolution of $P_i$ from $\Gamma_i$ to $\Gamma_{i+1}$. The proper transforms of $\Gamma_i$ have self-intersection $-c_i$. In this way, in the minimal resolution, the exceptional curves together with the proper transforms of the $\Gamma_i$ form a chain. Then this P-resolution will be denoted by $$ \left[d_{i_0}{n_0 \choose a_0}\right]-(c_1)-\left[d_{i_1}{n_1 \choose a_1}\right]-(c_2)-\ldots -(c_{\ell})- \left[d_{i_{\ell}}{n_{\ell} \choose a_{\ell}}\right],$$ representing the Hirzebruch-Jung continued fraction in the minimal resolution. The $n_i,a_i$ are computed as indicated below.
\end{notation}

In \cite[Corollary 10.1]{PPSU18}, it is described a geometric algorithm to obtain the P-resolution of $[k_1,\ldots,k_s] \in K(\Delta/\Delta-\Omega)$. First, for a given P-resolution, we note that $$[b_s,\ldots,b_1]-(1)-\left[d_{i_0}{n_0 \choose a_0}\right]-(c_1)-\left[d_{i_1}{n_1 \choose a_1}\right]-(c_2)-\ldots -(c_{\ell})- \left[d_{i_{\ell}}{n_{\ell} \choose a_{\ell}}\right]=0,$$ since $\left[d_{i_0}{n_0 \choose a_0}\right]-(c_1)-\left[d_{i_1}{n_1 \choose a_1}\right]-(c_2)-\ldots -(c_{\ell})- \left[d_{i_{\ell}}{n_{\ell} \choose a_{\ell}}\right]$ contacts to $\frac{1}{\Delta}(1,\Omega)$, and $[b_1,\ldots,b_s]$ is its dual fraction. 

\begin{algorithm} [for P-resolutions]
\label{algo}

\noindent
\begin{itemize} 
    \item[(0)] If $i_0=1$, then $n_0=a_0=1$. Otherwise $\frac{n_{0}}{n_{0} - a_{0}}=[b_{1},\ldots, b_{i_0-1}]$.
    
    \item[(1)] 
     We have $[b_s,\ldots,b_1]-(1)-[d_{i_0}{n_0 \choose a_0}]-(c_1)-\ldots$. We can blow-down the $(-1)$-curve and new $(-1)$-curves consecutively until we obtain the new chain 
    $$[b_{s},\ldots, b_{i_0+1},b_{i_0}-d_{i_0},b_{i_0-1},\ldots,b_1]-(c_1)-\ldots.$$ 
    
    \item[(2)] If $b_{i_1}-d_{i_1}=1$, then we contract this $(-1)$-curve and all new $(-1)$-curves in the subchain $[b_{s},\ldots, b_{i_0+1},b_{i_0}-d_{i_0},b_{i_0-1},\ldots,b_1]$ until there are none. 
    
    \item[(3)] Then the original $(-c_1)$-curve becomes a $(-1)$-curve, and we have $$ \frac{n_1}{n_1 - a_1}=[b_{1},\ldots, b_{i_0-1},b_{i_0}-d_{i_0},b_{i_0+1},\ldots,b_{i_1-1}]$$ if this is not equal to $1$. Otherwise $n_1=a_1=1$.
    
    \item[(4)] We now repeat starting in (1) with the $d_{i_1}$.
    
    \item[(5)] We end with $[\ldots,b_{i_{\ell}}-d_{i_{\ell}},\ldots,b_{i_0}-d_{i_0},\ldots]=0$, which is the zero continued fraction corresponding to the P-resolution.
\end{itemize}
\end{algorithm}

To simplify notation: 

\begin{itemize}
    \item For $\ldots -(c_j)-\left[1 {n_j \choose a_j}\right]-(c_{j+1})-\ldots$ we write $\ldots -(c_j)-\left[{n_j \choose a_j}\right]-(c_{j+1})-\ldots$.

    \item For $\ldots -(c_j)-\left[d_{i_j}{1 \choose 1}\right]-(c_{j+1})-\ldots$ we write $\ldots -(c_j)-A_{d_{i_j}}-(c_{j+1})-\ldots$. An $A_0$ is considered as smooth point, and so the next line applies.
    
    \item For $\ldots -(c_j)-\left[1 {1 \choose 1}\right]-(c_{j+1})-\ldots$ we write $\ldots -(c_j)-(c_{j+1})-\ldots$.

    \item Endings of the type $A_x-\left[d {n \choose a}\right]-\ldots$ or $\ldots-\left[d {n \choose a}\right]-A_x$ will mean $A_{x-1}-(2)-\left[d {n \choose a}\right]-\ldots$ or $\ldots-\left[d {n \choose a}\right]-(2)-A_{x-1}$ if $x\geq 1$ respectively; otherwise ($x=0$) it will mean $\left[d {n \choose a}\right]-\ldots$ or $\ldots-\left[d {n \choose a}\right]$.
\end{itemize}

\begin{remark}
We note that if $i_0=1$ or $i_{\ell}=s$, then the corresponding P-resolution begins with a $A_{d_1}$ singularity, or ends with a $A_{d_s}$ singularities, including the nonsingular case when $d=1$. 
\label{typeA}
\end{remark}


Below we are going to compute P-resolutions of particular c.q.s. For that purpose, we need some simple properties of zero continued fractions which we briefly describe now.

\begin{lemma}
Consider a convex polygon $\PP$ with $s+1$ sides. The indices from the vertices of $\PP$ will be taken mod $s+1$. Consider $[v_1,\ldots,v_s]=0$ for a triangulation of $\PP$.
    \begin{itemize}
        \item[(a)] At least two $v_i$ must be equal to $1$. Furthermore, for $s\geq 3$, the values of $1$ cannot be in consecutive positions.
        \item[(b)] Let $s \geq 2$. If $[k_1,\ldots,k_{i-1},1,k_{i+1},\ldots,k_s]=0$ for $i\neq 1,s$, then $[k_1,\ldots,k_{i-1}-1,k_{i+1}-1,\ldots,k_{s}]=0$. If $[1,k_2,\ldots,k_s]=0$ or $[k_1,\ldots,k_{s-1},1]=0$, then $[k_2-1,\ldots,k_s]=0$ and $[k_1,\ldots,k_{s-1}-1]=0$ respectively. Furthermore, If $s\geq 4$ and $[k_1, \ldots, k_{i-1}, \allowbreak 1,k_{i+1},\ldots,k_{s}]=0$ for $i\neq 1,s$, then $k_{i-1}$ and $k_{i+1}$ cannot be both equal to $2$.
    \end{itemize}
    \label{easy}
\end{lemma}

\begin{proof}
    These properties are all straightforward.  
\end{proof}

\begin{remark}
In each of the next propositions, we begin with a particular c.q.s. $\frac{\Delta}{\Omega}=[e_1,\ldots,e_l]$, and we consider all possible $[k_1,\ldots,k_s] \in K(\Delta/\Delta-\Omega)$, which have the form $k_i=b_i-x_i$ where $\frac{\Delta}{\Delta-\Omega}=[b_1,\ldots,b_s]$ and $x_i\geq 0$. There must be $k_i=1$ for some position $i$ by Lemma \ref{easy}(a), and there is a limited amount of such positions, which depends on the chains of $2$s in $[b_1,\ldots,b_s]$ by Lemma \ref{easy}(b). Then we analyze realizability for each of these particular positions.
\label{strategy}
\end{remark}

\begin{proposition}
Let $r\geq 3$ and $a,b\geq 0$. Consider the H-J continued fraction $[\underbrace{2,\ldots,2}_{a},r,\underbrace{2,\ldots,2}_{b}].$ Then its associated zero continued fractions and corresponding P-resolutions are:
{\tiny 
\begin{itemize}
    \item[(a)] $[1,\underbrace{2,\ldots,2}_{r-3},1]$ for $r\geq 3$, and $a,b\geq 0$; $A_{a}-(r)-A_{b}$.
    \item[(b)] $[r-2,1,\underbrace{2,\ldots,2}_{r-3}]$ for $a\geq r-4$, $r\geq 4$ and $b\geq 0$; $A_{a-r+4}-[\binom{r-2}{r-3}]-A_{b}$.
    \item[(c)] $[\underbrace{2,\ldots,2}_{r-3},1,r-2]$ for $a\geq 0$, $r\geq 4$ and $b\geq r-4$; $A_{a}-[\binom{r-2}{1}]-A_{b-r+4}$.
\end{itemize}}
    \label{2r2}
\end{proposition}

\begin{proof}
The dual chain is $[a+2,2,\ldots,2,b+2]$ with $r-3$ $2$s. If $r=3$, then the only possibility for a zero continued fraction is $[1,1]$. Let $r\geq 4$. As in Remark \ref{strategy}, there are only $4$ positions $i$ for $k_i=1$: $1,2,r-2,r-1$. If $(a+2)-x_1=1$, using Lemma \ref{easy} repeatedly we obtain (a). If $2-x_2=1$, using that lemma repeatedly we have $0=[(a-r+5)-x_1,(b+1)-x_{r-1}]$. Then, $x_1=a-r+4$ and $x_{r-1}=b$, so we obtain (b). The third case, $2-x_{r-2}=1$, is analogous to the previous one. Finally, for $(b+2)-x_{r-1}=1$, we have the first case.
\end{proof}

\begin{proposition} 
Let $r\geq 3$, $a\geq 3$ and $b,c\geq 0$. Consider the H-J continued fraction $[\underbrace{2,\ldots,2}_b,r,a,\underbrace{2,\ldots,2}_c]$. Then its associated zero continued fractions and P-resolutions are:
{\tiny 
\begin{itemize}
    \item[(a)] $[1,\underbrace{2,\ldots,2}_{r+a-5},1]$ for $r,a\geq 3$, and $b,c\geq 0$; $A_{b}-(r)-(a)-A_{c}$.
    
    \item[(b)] $[1,\underbrace{2,\ldots,2}_{r-3},3,\underbrace{2,\ldots,2}_{a-4},1,a-2]$ for $r\geq 3, a\geq 4$, $b\geq 0$ and $c\geq a-4$; $A_{b}-(r)-[\binom{a-2}{1}]-A_{c-a+4}$.
    
    \item[(c)] $[r-2,1,\underbrace{2,\ldots,2}_{r-4},3,\underbrace{2,\ldots,2}_{a-3},1]$ for $r\geq 4$, $a\geq 3$, $b\geq r-4$ and $c\geq 0$; $A_{b-r+4}-[\binom{r-2}{r-3}]-(a)-A_{c}$.
    
    \item[(d)] $[r+a-4,1,\underbrace{2,\ldots,2}_{r+a-5}]$ for $r\geq 3$ and $a\geq 3$, $b\geq r+a-6$ and $c\geq 0$; \\ $A_{b-r-a+6}-[\binom{r+a-4}{r+a-5}]-(1)-[\binom{a-1}{a-2}]-A_{c}$.
    
    \item[(e)] $[\underbrace{2,\ldots,2}_{r+a-5},1,r+a-4]$ for $r\geq 3, a\geq 3$, $b\geq 0$ and $c\geq r+a-6$; \\ $A_{b}-[\binom{r-1}{r-2}]-(1)-[\binom{r+a-4}{r+a-5}]-A_{c-r-a+6}$.
    
    \item[(f)] $[r-1,1,\underbrace{2,\ldots,2}_{r-4},3,\underbrace{2,\ldots,2}_{a-4},1,a-1]$ for $r,a\geq 4$, $b\geq r-3$ and $c\geq a-3$; \\ $A_{b-r+3}-[\binom{r-1}{r-2}]-(1)-[\binom{a-1}{1}]-A_{c-a+3}$.
    
    \item[(g)] $[a,2,1,3,\underbrace{2,\ldots,2}_{a-2}]$ for $r=5$, $a\geq 3$, $b\geq a-2$ and $c\geq 0$; $A_{b-a+2}-[\binom{2a-1}{2a-3}]-A_{c}$.
    
    \item[(h)] $[\underbrace{2,\ldots,2}_{r-2},3,1,2,r]$ for $a=5$, $r\geq 3$, $b\geq 0$ and $c\geq r-2$; $A_{b}-[\binom{2r-1}{2}]-A_{c-r+2}$.
\end{itemize}
}
\label{2ra2} 
\end{proposition}

\begin{proof}
The cases $r=3$ and $a\geq 3$, or $r\geq 3$ and $a=3$ can be checked with ease. If $r\geq 4$ and $a\geq 4$, then the dual chain is $[b+2,2,\ldots,2,3,2,\ldots,2,c+2]$ with $(r-3)$ $2$s on the left and $(a-3)$ $2$s on the right. As in Remark \ref{strategy}, we need to study $6$ positions $i$ for $k_i=1$. We check one by one. If $(b+2)-x_1=1$, then (using Lemma \ref{easy}) we have $0=[2-x_{r-1},2-x_{r},\ldots,2-x_{r+a-4},(c+2)-x_{r+a-3}]$. The possibilities for an entry equal to $1$ are at the ends, or $2-x_{r+a-4}=1$. If we have a $1$ at the ends of the last zero continued fraction, we obtain the first case. On the other hand, if $2-x_{r+a-4}=1$, using Lemma \ref{easy} repeatedly we have $0=[1-x_{r-1},(c-a+5)-x_{r+a-3}]$, so we obtain the case (b). The second possibility is $2-x_2=1$, so using Lemma \ref{easy} repeatedly we have $0=[(b-r+5)-x_1,2-x_{r-1},2-x_{r}\ldots,2-x_{r+a+-4},(c+2)-x_{r+a-3}]$. The possibilities for an entry equal to $1$ are at the ends of the chain or the ends of the sequence of $2$s. If we have a $1$ at the ends of the chain, we obtain the case (c). If $2-x_{r-1}=1$ or $2-x_{r+a-4}=1$, we obtain the cases (d) or (f), respectively. The third possibility is $2-x_{r-2}=1$. If $r\geq 6$, using Lemma \ref{easy}(b) once, we have a contradiction. If $r=4$, then this case is a particular case of the second possibility. If $r=5$, we obtain the case (g). Due to the symmetry of the dual chain, the other cases are analogous.
\end{proof}

\begin{proposition}
Let $r\geq 3$, $a\geq 3$, $b\geq 0$, and $n\geq 2$ be integers. Consider the H-J continued fraction $[\underbrace{2,\ldots,2}_b,r,a,\underbrace{2,\ldots,2}_{n-2},3,\underbrace{2,\ldots,2}_{a-3}]$. Then its associated zero continued fractions and corresponding P-resolutions are: 
{\tiny 
\begin{itemize}
    \item[(a)] $[1,\underbrace{2,\ldots,2}_{r+a-4},1]$ for $b\geq 0$, $r\geq 3$, $a\geq 3$ and $n\geq 2$; $A_{b}-(r)-(a)-A_{n-2}-(3)-A_{a-3}$.
    
    \item[(b)] $[1,\underbrace{2,\ldots,2}_{r-3},3,\underbrace{2,\ldots,2}_{a-4},1,a-1,1]$ for $r\geq 3$, $a\geq 4$, $b\geq 0$ and $n\geq a-2$; $A_{b}-(r)-[\binom{a-2}{1}]-A_{a-4}-(3)-A_{a-3}$.
    
    \item[(c)] $[1,\underbrace{2,\ldots,2}_{r-3},3,\underbrace{2,\ldots,2}_{a-3},1,a-1]$ for $b\geq 0$, $r\geq 3$, $n\geq 2$ and $a\geq 3$; $A_{b}-(r)-[n\binom{a-1}{1}]$.
    
    \item[(d)] $[r-2,1,\underbrace{2,\ldots,2}_{r-4},3,\underbrace{2,\ldots,2}_{a-2},1]$ for $b\geq r-4$, $r\geq 4$, $a\geq 3$ and $n\geq 2$; \\  $A_{b-r+4}-[\binom{r-2}{r-3}]-(a)-A_{n-2}-(3)-A_{a-3}$.
    
    \item[(e)] $[r+a-3,1,\underbrace{2,\ldots,2}_{r
    +a-4}]$ for $b\geq r+a-5$, $r\geq 4$, $a\geq 3$ and $n\geq 2$; \\ $A_{b-r-a+5}-[\binom{r+a-3}{r+a-4}]-(1)-[\binom{a}{a-1}]-(1)-[(n-1)\binom{2}{1}]-A_{a-3}$.
    
    \item[(f)] $[r+a-4,1,\underbrace{2,\ldots,2}_{r+a-6},3,1]$ for $b\geq r+a-6$, $r\geq 4$, $n\geq 2$ and $a\geq 3$; \\ $A_{b-r-a+6}-[\binom{r+a-4}{r+a-5}]-(1)-[\binom{a-1}{a-2}]-A_{n-2}-(3)-A_{a-3}$.

    \item[(g)] $[r,1,\underbrace{2,\ldots,2}_{r-4},3,\underbrace{2,\ldots,2}_{a-4},1,a-1,2]$ for $b\geq r-2$, $r\geq 4$, $n\geq a-2$ and $a\geq 4$; \\ $A_{b-r+2}-[\binom{r}{r-1}]-(1)-[\binom{2a-3}{a-1}]-(1)-[(n-a+2)\binom{2}{1}]-A_{a-3}$.
    
    \item[(h)] $[r-1,1,\underbrace{2,\ldots,2}_{r-4},3,\underbrace{2,\ldots,2}_{a-4},1,a,1]$ for $b\geq r-3$, $r\geq 4$, $n\geq a-1$ and $a\geq 4$; \\ $A_{b-r+3}-[\binom{r-1}{r-2}]-(1)-[\binom{a-1}{1}]-A_{n-a+1}-(3)-A_{a-3}$. 
    
    \item[(i)] $[a+1,2,1,3,\underbrace{2,\ldots,2}_{a-1}]$ for $r=5$, $b\geq a-1$, $n\geq 2$ and $a\geq 3$; \\ $A_{b-a+1}-[\binom{2a+1}{2a-1}]-(1)-[(n-1)\binom{2}{1}]-A_{a-3}$.
    
    \item[(j)] $[a,2,1,3,\underbrace{2,\ldots,2}_{a-3},3,1]$ for $r=5$, $b\geq a-2$, $n\geq 2$ and $a\geq 3$; \\ $A_{b-a+2}-[\binom{2a-1}{2a-3}]-A_{n-2}-(3)-A_{a-3}$.
    
    \item[(k)] $[3,\underbrace{2,\ldots,2}_{r-3},3,1,2,r,2]$ for $a=5$, $b\geq 1$, $r\geq 3$ and $n\geq r-1$; \\ $A_{b-2}-[\binom{4(r-1)}{2r-1}]-(1)-[(n-r+1)\binom{2}{1}]-(2)-A_1$.

    \item[(l)] $[\underbrace{2,\ldots,2}_{r-2},3,1,2,r+1,1]$ for $a=5$, $b\geq 0$, $r\geq 3$ and $n\geq r$; \\ $A_{b}-[\binom{2r-1}{2}]-A_{n-r}-(3)-(2)-A_1$.

    \item[(m)] $[3,\underbrace{2,\ldots,2}_{r+a-6},1,r+a-4,2]$ for $b\geq 1$, $r\geq 4$, $n\geq r+a-5$ and $a\geq 3$; \\ $A_{b-1}-[\binom{2r-3}{r-1}]-(1)-[\binom{2r+2a-9}{r+a-4}]-(1)-[(n-r-a+4)\binom{2}{1}]-A_{a-3}$. 
   
    \item[(n)] $[\underbrace{2,\ldots,2}_{r+a-5},1,a+r-3,1]$ for $b\geq 0$, $r\geq 3$, $n\geq r+a-4$ and $a\geq 3$; \\ $A_{b}-[\binom{r-1}{1}]-(1)-[\binom{r+a-4}{1}]-A_{n-r-a+4}-(3)-A_{a-3}$.
    
    \item[(o)] $[\underbrace{2,\ldots,2}_{r-2},3,\underbrace{2,\ldots,2}_{a-4},1,a-2,r]$ for $b\geq 0$, $r\geq 3$, $n\geq a-3$ and $a\geq r+1$; \\ $A_{b}-[\binom{ra-2r-1}{a-2}]-(1)-[(n-a+3)\binom{r}{1}]-A_{a-r-1}$.
\end{itemize}
}
\label{2ra232}
\end{proposition}
\begin{proof}
The cases ($r=3$ and $a\geq 3$) and ($r\geq 3$ and $a=3$) can be checked with ease. Assume that $r\geq 4$ and $a\geq 4$. Its dual chain is $[b+2,2,\ldots,2,3,2,\ldots,2,n+1,a-1]$ with $(r-3)$ $2$s on the left and $(a-3)$ $2$s on the right. As in Remark \ref{strategy}, we have $8$ possible positions $i$ for $k_i=1$. If $(n+1)-x_1=1$, using Lemma \ref{easy} repeatedly we have $0=[2-x_{r-1},2-x_{r},\ldots,2-x_{r+a-4},(n+1)-x_{r+a-3},(a-1)-x_{r+a-2}]$. We must have an entry equal to $1$. In this case, we have $3$ possibilities for it; thus, we obtain the first $3$ cases. If $2-x_2=1$, using Lemma \ref{easy} we descend to a situation where we have five possibilities for an entry equal to 1. These are the next 5 cases on the list. The other elements are obtained in the same way.
\end{proof}

\begin{proposition}
Let $a\geq 4$, $b\geq 0$, and $n\geq 2$ be integers. Consider the H-J continued fraction $[\underbrace{2,\ldots,2}_{b},a,\underbrace{2,\ldots,2}_{n-2},3,\underbrace{2,\ldots,2}_{a-4}].$ Then its associated zero continued fractions and corresponding P-resolutions are: 

{\tiny
\begin{enumerate}
    \item[(a)] $[1,\underbrace{2,\ldots,2}_{a-2},1]$ for $b\geq 0$, $a\geq 4$, and $n\geq 2$;   $A_{b}-(a)-A_{n-2}-(3)-A_{a-4}$.
    
    \item[(b)] $[a-1,1,\underbrace{2,\ldots,2}_{a-2}]$ for $b\geq a-3$, $a\geq 4$ and $n\geq 2$;  $A_{b-a+3}-[\binom{a-1}{a-2}]-(1)-[(n-1)\binom{2}{1}]-A_{a-4}$.
    
    \item[(c)] $[a-2,1,\underbrace{2,\ldots,2}_{a-4},3,1]$ for $b\geq a-4$, $a\geq 4$ and $n\geq 2$;  $A_{b-a+4}-[\binom{a-2}{a-3}]-A_{n-2}-(3)-A_{a-4}$.
    
    \item[(d)] $[3,\underbrace{2,\ldots,2}_{a-4},1,a-2,2]$ for $b\geq 1$, $a\geq 4$ and $n\geq a-3$; $A_{b-1}-[\binom{2a-5}{a-2}]-(1)-[(n-a+3)\binom{2}{1}]-A_{a-4}$.
    
    \item[(e)] $[\underbrace{2,\ldots,2}_{a-3},1,a-1,1]$ for $b\geq 0$, $a\geq 4$ and $n\geq a-2$;  $A_{b}-[\binom{a-2}{1}]-A_{n-a+2}-(3)-A_{a-4}$.   
\end{enumerate}
}
\label{2a232}
\end{proposition}

\begin{proof}
Its dual chain is $[b+2,2,\ldots,2,n+1,a-2]$ with $(a-3)$ $2$s in the middle. As in Remark \ref{strategy}, we have $4$ possible positions $i$ for $k_i=1$ (position at $n+1$ is not possible in this case). If $(b+2)-x_1=1$, using Lemma \ref{easy} repeatedly we obtain the first case. If $2-x_2=1$, using Lemma \ref{easy} repeatedly we have $0=[(b-a+5)-x_1,n-x_{a-1},(a-2)-x_{a}]$. The possibilities for an entry equal to $1$ are in the middle of the chain or at the ends of the chain. These cases are (b) and (c), respectively. If $2-x_{a-2}=1$, we descend to $0=[(b+1)-x_1,(n-a+4)-x_{a-1},(a-2)-x_{a}]$. The possibilities for an entry equal to $1$ are in the middle of the chain or at the ends of the chain. These cases are (d) and (e), respectively. Finally, if $(a-2)-x_{a}=1$, since it is one of the ends, Lemma \ref{easy}(a) implies that there is another entry equal to $1$, but we did consider that case already.
\end{proof}

\begin{proposition} 
Let $a\geq 3$. Consider the H-J continued fractions $[3,a,2]$. Then its associated zero continued fractions and corresponding P-resolutions are: 
{\tiny 
\begin{itemize}
    \item[(a)] $[1,\underbrace{2,\ldots,2}_{a-2},1]$ for $a\geq 3$; $(3)-(a)-A_1$. 
    \item[(b)] $[2,1,2]$ for $a=3$; $[2\binom{2}{1}]-(2)$.
    \item[(c)] $[1,3,1,2]$ for $a=4$; $(3)-[\binom{2}{1}]-(2)$.
    \item[(d)] $[2,2,1,3]$ for $a=4$;  $[\binom{2}{1}]-(1)-[\binom{3}{1}]$. 
    \item[(e)] $[2,3,1,2,3]$ for $a=5$; $[\binom{5}{2}]$.
    \item[(f)] $[1,3,2,1,3]$ for $a=5$; $(3)-\binom{3}{1}$.
\end{itemize}}
\label{typeII}
\end{proposition}

\begin{proof}
This is a simple inspection.
\end{proof}

\begin{proposition} 
Let $a\geq 3$. Consider the H-J continued fractions $[4,a,2]$. Then its associated zero continued fractions and corresponding P-resolutions are:  
{\tiny
\begin{itemize}
    \item[(a)] $[1,\underbrace{2\ldots,2}_{r-1},1]$ for $a\geq 3$; $(4)-(a)-A_1$.
    \item[(b)] $[1,2,3,1,2]$ for $a=4$; $(4)-[\binom{2}{1}]-(2)$.
    \item[(c)] $[1,2,3,2,1,3]$ for $a=5$; $(4)-[\binom{3}{1}]$.
    \item[(d)] $[2,1,3,\underbrace{2,\ldots,2}_{a-3},1]$ for $a\geq 3$; $[\binom{2}{1}]-(a)-A_1$.
    \item[(e)] $[2,2,1,3]$ for $a=3$; $[2\binom{3}{2}]$.
\end{itemize}}
\label{typeIII}
\end{proposition}

\begin{proof}
 This is a simple inspection.
\end{proof}

\begin{proposition} 
Let $a\geq 3$. Consider the H-J continued fractions 
$[3,a,3]$. Then its associated zero continued fractions and corresponding P-resolutions are:  
{\tiny 
\begin{itemize}
    \item[(a)] $[1,\underbrace{2,\ldots,2}_{a-1},1]$ for $a\geq 3$; $(3)-(a)-(3)$.
    \item[(b)] $[1,3,1,2]$ for $a=3$; $(3)-[2\binom{2}{1}]$.
    \item[(c)] $[2,1,3,1]$ for $a=3$; $[2\binom{2}{1}]-(3)$.
    \item[(d)] $[1,3,1,3,1]$ for $a=4$; $(3)-[\binom{2}{1}]-(3)$.
\end{itemize}}
\label{typeIV}
\end{proposition}

\begin{proof}
The case $a=3$ can be checked directly. Assume that $a\geq 4$. Its dual chain is $[2,3,2,\ldots,2,$ $3,2]$ with $(a-3)$ $2$s in the middle. Remark \ref{strategy} tells us that there could be $4$ positions $i$ for $k_i=1$. If $2-x_1=1$, using Lemma \ref{easy} once we have $0=[2-x_2,2-x_3,\ldots,2-x_{a-1},3-x_{a},2-x_{a+1}]$. Note that if $a\geq 5$, then the only possibilities for an entry equal to $1$ are at the ends. If $a=4$, the possibilities are the ends, but $2-x_{a-1}=1$ is also possible. Using Lemma \ref{easy}, we have the first and last case, respectively. If $2-x_3=1$, note that if $a\geq 5$ we will have a contradiction with that lemma. If $a=4$, using Lemma \ref{easy} repeatedly we will have the last case. Due to the symmetry of the dual chain, the other cases are analogous.
\end{proof}

From now on, we will be only concerned about certain elements in $K(\Delta/\Delta-\Omega)$, hence we define $$ K^{\times}(\Delta/\Delta-\Omega) = \{ [k_1,\ldots,k_s] \in K(\Delta/\Delta-\Omega) \ \text{with} \  k_1=b_1 \ \text{and} \ k_s=b_s \}.$$ This means we will be considering only P-resolutions with no ending $A_d$ singularities, including smooth points $A_0$. This is because the next propositions will be used for Lemma \ref{allforsection} part $(2)$, this is, when two complete fibers are used. In the presence of ending $A_d$s, the ``final" $(-1)$-curve in any of these fibers would contradict nefness with $K_W$. See Remark \ref{typeA} for a characterization of this. Let us call them \textit{non-A-ending P-resolutions}.

\begin{proposition}
Let $r\geq 3$, and $b\geq 4$. Consider the H-J continued fraction $[4,r,b,\underbrace{2,\ldots,2}_{b-4}].$ Then there is only one associated zero continued fraction and corresponding non-A-ending P-resolution: $[2,1,3,\underbrace{2,\ldots,2}_{r-3},3,\underbrace{2,\ldots,2}_{b-3},1,b-2]$ for $r\geq 3$ and $b\geq 4$; $[\binom{2}{1}]-(r)-[\binom{b-1}{1}]$.

\label{2arb2}
\end{proposition}

\begin{proof} 
Its dual chain is $[2,2,3,2,\ldots,2,3,2,\ldots,2,b-2]$ with $(r-3)$ $2$s on the left and $(b-3)$ $2$s on the right. As in Remark \ref{strategy}, we have at most $7$ possible positions $i$ for $k_i=1$ (when $r=3$, we only have $5$ of them): $2,3,4,r,r+1,r+2,r+b-2$. If $(2-x_2)=1$, using Lemma \ref{easy} repeatedly, we descend to $0=[2-x_{r+1},2-x_{r+2},\ldots,2-x_{r+b-2},
b-2]$. Since $b\geq 4$, the only possibilities for an entry equal to $1$ are $(2-x_{r+1})=1$ or $(2-x_{r+b-2})=1$. In the first case, we descend to $0=[1-x_{r+b-2},b-2]$, which is absurd; whereas in the second one, we descend to $0=[1-x_{r+1},1-x_{r+b-1}]$, so we obtain the case (a). If $(3-x_3)=1$, we must have $r=3$; otherwise we have a contradiction with Lemma \ref{easy}(b). Using that lemma once yields a contradiction. If $(2-x_4)=1$ (it requires $r>3$), we must have $r\leq 5$. If $r=5$, using Lemma \ref{easy}(b) twice yields a contradiction, whereas when $r=4$, we descend to $0=[2,2-x_2,2-x_3,2-x_{r+1},\ldots,2-x_{r+b-2},b-2]$. The only possibility for an entry equal to $1$ is $2-x_{r+b-2}=1$, but using Lemma \ref{easy}(b) $(b-4)$ times we have a contradiction with that lemma. If $(2-x_{r})=1$ (it requires $r>3$), we must have $r\leq 5$. If $r=4$, we have $x_r=x_4$, then this is the same as the previous case. If $r=5$, using Lemma \ref{easy}(b) twice, we have a contradiction with that lemma. If $(3-x_{r+1})=1$, we must have that $r=3$. Then, using Lemma \ref{easy}(b) twice, we obtain a contradiction with that lemma, since $b\geq 4$. If $2-x_{r+2}=1$, we must have $b\leq 5$. If $b=5$, using Lemma \ref{easy}(b) twice we obtain that $r=3$; in such a case, we descend to $0=[2,1-x_2]$, which is absurd. The case $b=4$ is a particular case of the following case. If $(2-x_{r+b-2})=1$, we descend to $0=[2,2-x_2,2-x_3]$. As the only zero continued fractions of length $2$ are $[2,1,2]$ and $[1,2,1]$, we obtain the case (a).
\end{proof}

\begin{proposition}
Let $r\geq 3$, $b\geq 3$, and $n\geq 2$. Consider the H-J continued fraction $[4,r,b,\underbrace{2,\ldots,2}_{n-2},3,\underbrace{2,\ldots,2}_{b-3}].$ Then its associated zero continued fractions and corresponding non-A-ending P-resolutions are: 
{\tiny
\begin{itemize}
    
    
    
    
    \item[(a)] $[2,1,3,\underbrace{2,\ldots,2}_{r-3},3,\underbrace{2,\ldots,2}_{b-3},1,b-1]$ for $r\geq 3$, $n\geq 2$ and $b\geq 3$; $[\binom{2}{1}]-(r)-[n\binom{b-1}{1}]$.
    
    
    \item[(b)] $[2,2,3,\underbrace{2,\ldots,2}_{r-3},3,1,2,r,4]$ for $r\geq 3$, $b=5$ and $n\geq r-1$; $[\binom{8r-6}{2r-1}]-(1)-[(n-r+1)\binom{4}{1}]$. 
    
    \item[(c)] $[2,2,3,\underbrace{2,\ldots,2}_{r-1},1,r+1,4]$ for $r\geq 3$, $b=5$ and $n\geq r$; $[\binom{4r-5}{r-1}]-(1)-[\binom{4r+3}{r+1}]-(1)-[(n-r)\binom{4}{1}]$.
    
    \item[(d)] $[\underbrace{2,\ldots,2}_{r},3,\underbrace{2,\ldots,2}_{r-1},1,b-2,b-1]$ for $r\geq 3$, $b=r+3$ and $n\geq r$; $[\binom{3}{1}]-(1)-[\binom{r^2+2r+1}{r+1}]-(1)-[(n-r)\binom{r+2}{1}]$.
    
    
    
\end{itemize}}
\label{4rb232}
\end{proposition}

\begin{proof}
Its dual chain is $[2,2,3,2,\ldots,2,3,2,\ldots,2,n+1,b-1]$, with $(r-3)$ $2$s on the left and $(b-3)$ $2$s on the right. Following the strategy in Remark \ref{strategy}, we have at most $8$ possible positions $i$ for $k_i=1$ (when $r=3$, we have $6$ of them, while when $b=r=3$, we have only $4$ of them): $2,3,4,r,r+1,r+2,r+b-2,r+b-1$. 

If $(2-x_2)=1$, then we descend to $0=[2-x_{r+1},2-x_{r+2},\ldots,2-x_{r+b-2},(n+1)-x_{r+b-1},b-1]$. When $(2-x_{r+b-2}=1)$ or $(2-x_{r+1}=1)$, we descend to an absurd situation, whereas when $(n+1)-x_{r+b-1}=1$ we obtain (a). 

If $(3-x_3)=1$, we must have $r=3$. Using Lemma \ref{easy}(b), it yields a contradiction. 

If ($2-x_4)=1$ (it requires $r>3)$, we must have $r\leq 5$. In both cases, we obtain a contradiction.

If $(2-x_{r})=1$ (it requires $r>3)$, we have that $r\leq 5$. In both cases, we obtain a contradiction.

If $(3-x_{r+1})=1$, we must have ($r=3$ and $3\leq b\leq 4$) or ($r\geq 3$ and $b=3$). In any of these cases, we descend to a zero continued fraction of length two, where one the entries is $b-1$, which is not possible. 

If $(2-x_{r+2})=1$ (it requires $b>3$), we must have $b\leq 5$. If $b=5$, we descend to $0=[2,2-x_2,2-x_3,(n-r+2)-x_{r+b-1},4]$. The only possibility is $(n-r+2)-x_{r+b-1}=1$, so we obtain (b). The case $b=4$ is part of the next case. 

If $(2-x_{r+b-2})=1$ (it requires $b>3$), we descend to $0=[2,2-x_2,3-x_3,2-x_4,\ldots,2-x_r,2-x_{r+1},(n'-b+4)-x_{r+b-1},b-1]$. The cases $(2-x_2=1)$, $(3-x_3=1)$, and $(2-x_4=1)$ have already been considered, so we have two cases. If $2-x_{r+1}=1$, we descend to $0=[2,2-x_2,2-x_3,(n-b-r+6)-x_{r+b-1},b-1]$. The only possibility is $(n-b-r+6)-x_{r+b-1}=1$, then we descend to $0=[1,b-4]$, so we have (c). If $(n-b+4)-x_{r+b-1}=1$, we descend to $0=[2,2-x_2,2-x_3,(b-r)]$. The only possibility is $2-x_3=1$, so we obtain (d). 

Finally, if $(n+1)-x_{r+b-1}=1$, we have (a).
\end{proof}

\begin{proposition}
Let $r\geq 3$, $a\geq 4$, and $n\geq 2$. Consider the H-J continued fraction $[\underbrace{2,\ldots,2}_{a-4},a,r,3,\underbrace{2,\ldots,2}_{n-2},3].$ Then its associated zero continued fractions and corresponding non-A-ending P-resolutions are: 
{\tiny 
\begin{itemize}

    \item[(a)] $[a-2,1,\underbrace{2,\ldots,2}_{a-4},3,\underbrace{2,\ldots,2}_{r-3},3,1,2]$ for $r\geq 3$, $n\geq 2$ and $a\geq 4$; $[\binom{a-2}{1}]-(r)-[n\binom{2}{1}]$.

    \item[(b)] $[3,2,2,3,1,2,5,2]$ for $a=5$, $r=4$ and $n\geq 4$; $[\binom{16}{9}]-(1)-[\binom{9}{5}]-(1)-[(n-4)\binom{2}{1}]$.
    
    \item[(c)] $[3,2,2,3,2,1,3,5,2]$ for $a=5$, $r=5$ and $n\geq 4$; $[\binom{25}{14}]-(1)-[(n-4)\binom{2}{1}]$.
    
    
    
    \item[(d)] $[3,\underbrace{2,\ldots,2}_{r},1,r+2,2]$ for $a=5$, $r\geq 3$ and $n\geq r+1$; $[\binom{7}{4}]-(1)-[2\binom{2r+3}{r+2}]-(1)-[(n-r-1)\binom{2}{1}]$. 
    
    
    
    
    
    
\end{itemize}}
\label{2ar323}
\end{proposition}

\begin{proof}
Its dual chain is $[a-2,2,\ldots,2,3,2,\ldots,2,3,n+1,2]$ with $a-3$ $2$s on the left and $r-3$ $2$s on the right. Following Remark \ref{strategy} we have at most $7$ possible positions for an entry $k_i$ equal to $1$: $2,a-2,a-1,a,a+r-4,a+r-3,a+r-2$. 

If $(2-x_2)=1$, we descend to $0=[2-x_{a+r-3},(n+1)-x_{a+r-2},2]$. Since the only zero continued fractions of length $3$ are $[2,1,2]$ and $[1,2,1]$, we obtain the case (a).

If $(2-x_{a-2})=1$, we must have $a\leq 5$. The case $a=4$, is a particular case of the previous one. If $a=5$, we descend to a zero continued fraction of length two with an entry equal to $2$, which is an absurd. 

If $(3-x_{a-1})=1$, we must have $r=3$. Since $a\geq 4$, we obtain a contradiction.

If $(2-x_a)=1$ (it requires $r>3$), we must have $r\leq 5$. If $r=5$, then we obtain a contradiction. If $r=4$, we descend to $0=[a-2,2-x_2,\ldots,2-x_{a-2},2-x_{a-1},2-x_{a+r-3},(n+1)-x_{a+r-2},2-x_{a+r-1}]$. The only possible case is $2-x_{a+r-3}=1$, so we obtain (b). 

If $2-x_{a+r-4}=1$ (it requires $r>3$), we must have that $r\leq 5$. The case $r=4$ is a particular case of the previous one. If $r=5$, we descend to $0=[a-3,(n-a+2)-x_{a+r-2},2]$, so we obtain (c).

If $(3-x_{a+r-3})=1$, we descend to $0=[a-2,2-x_2,\ldots,2-x_{a-2},2-x_{a-1},(n-r+3)-x_{a+r-2},2-x_{a+r-1}]$. The only possibility is $2-x_{a-1}=1$, so we descend to $0=[a-3,(n-r-a+5)-x_{a+r-2},2-x_{a+r-1}]$, so we obtain (d). 

Finally, if $(n+1)-x_{a+r-2}=1$, we obtain the case (a).
\end{proof}

\begin{proposition}
Let $r\geq 3$, $b\geq 3$, and $n,n'\geq 2$ be integers. Consider the H-J continued fraction $[3,\underbrace{2,\ldots,2}_{n-2},3,r,b,\underbrace{2,\ldots,2}_{n'-2},3,\underbrace{2,\ldots,2}_{b-3}].$ Then its associated zero continued fractions and corresponding non-A-ending P-resolutions are:
{\tiny
\begin{itemize}
    \item[(a)] $[2,3,3,1,2,3,3]$ for $r=3$, $b=4$ and $n,n'\geq 2$;  $[(n-2)\binom{2}{1}]-(1)-[2\binom{13}{5}]-(1)-[(n'-2)\binom{3}{1}]$.
   
    \item[(b)] $[2,3,3,1,3,1,4,3]$ for $r=4$, $b=4$, $n\geq 2$ and $n'\geq 3$;  $[(n-2)\binom{2}{1}]-(1)-[\binom{13}{5}]-(1)-[\binom{11}{4}]-[(n'-3)\binom{3}{1}]$.
   
    \item[(c)] $[2,5,2,1,3,2,2,3]$ for $r=4$, $b= 4$, $n\geq 4$ and $n'\geq 2$; $[(n-4)\binom{2}{1}]-(1)-[\binom{9}{4}]-(1)-[\binom{16}{7}]-(1)-[(n'-1)\binom{3}{1}]$.
   
    \item[(d)] $[2,5,3,1,2,3,2,2,3]$ for $r=5$, $b=4$, $n\geq 4$ and $n'\geq 2$; $[(n-4)\binom{2}{1}]-(1)-[\binom{25}{11}]-(1)-[(n'-1)\binom{3}{1}]$.
   
    \item[(e)] $[2,1,3,\underbrace{2,\ldots,2}_{r-3},3,\underbrace{2,\ldots,2}_{b-3},1,b-1]$ for $r\geq 3$, $b\geq 3$ and $n,n'\geq 2$;  $[n\binom{2}{1}]-(r)-[n'\binom{b-1}{1}]$.
   
    \item[(f)] $[2,r+2,1,\underbrace{2,\ldots,2}_{r},3]$ for $r\geq 3$, $b=4$, $n\geq r+1$ and $n'\geq 2$; \\ $[(n-r-1)\binom{2}{1}]-(1)-[2\binom{2r+3}{r+1}]-(1)-[\binom{7}{3}]-(1)-[(n'-1)\binom{3}{1}]$.
   
    \item[(g)] $[2,r+1,1,\underbrace{2,\ldots,2}_{r-3},3,1,3,3]$ for $r\geq 3$, $b=4$, $n\geq r$ and $n'\geq 2$; \\ $[(n-r)\binom{2}{1}]-(1)-[2\binom{2r+1}{r}]-(1)-[\binom{8}{3}]-(1)-[(n'-2)\binom{3}{1}]$.
   
    \item[(h)] $[2,3,\underbrace{2,\ldots,2}_{r-1},1,r+1,3]$ for $r\geq 3$, $b=4$, $n\geq 2$ and $n'\geq r$; \\ $[(n-2)\binom{2}{1}]-(1)-[\binom{5}{2}]-(1)-[\binom{3r-1}{r}]-(1)-[\binom{3r+2}{r+1}]-(1)-[(n'-r)\binom{3}{1}]$.
   
    \item[(i)] $[2,2,3,\underbrace{2,\ldots,2}_{r-3},3,1,2,r,4]$ for $r\geq 3$, $b=5$, $n\geq 2$ and $n'\geq r-1$; \\ $[(n-1)\binom{2}{1}]-(1)-[\binom{8r-6}{2r-1}]-(1)-[(n'-r+1)\binom{4}{1}]$.
   
    \item[(j)] $[2,2,3,\underbrace{2,\ldots,2}_{r-1},1,r+1,4]$ for $r\geq 3$, $b=5$, $n\geq 2$ and $n'\geq r$; \\ $[(n-1)\binom{2}{1}]-(1)-[\binom{4r-5}{r-1}]-(1)-[\binom{4r+3}{r+1}]-(1)-[(n'-r)\binom{4}{1}]$.
   
    \item[(k)] $[\underbrace{2,\ldots,2}_{r},3,\underbrace{2\ldots,2}_{r-1},1,b-2,b-1]$ for $r\geq 3$, $b=r+3$, $n\geq 2$ and $n'\geq r$;\\  $[(n-1)\binom{2}{1}]-(1)-[\binom{3}{1}]-(1)-[\binom{r^2+2r+1}{r+1}]-(1)-[(n'-r)\binom{r+2}{1}]$.
\end{itemize}}
\label{323rb232}
\end{proposition}

\begin{proof}
Its dual chain is $[2,n+1,3,2,\ldots,2,3,2,\ldots,2,n'+1,b-1]$ with $r-3$ $2$s on the left and $b-3$ $2$s on the right. Following Remark \ref{strategy} we have at most $8$ possible positions $i$ for $k_i=1$ (when $r=3$ or $b=3$, we have $6$ cases, while when $b=r=3$, we have only $4$ of them): $2,3,4,r,r+1,r+2,r+b-2,r+b-1$.

If $(n+1)-x_2=1$, we descend to $0=[2-x_{r+1},2-x_{r+2},\ldots,2-x_{r+b-2},(n'+1)-x_{r+b-1},b-1]$. When $(2-x_{r+1})=1$ or $(2-x_{r+b-2})=1$, we descend to an absurd situation, whereas when $(n'+1)-x_{r+b-1}=1$, we obtain (e).

If $(3-x_3)=1$, we descend to $0=[2,(n-r+3)-x_2,2-x_{r+1},\ldots,2-x_{r+b-2},(n'+1)-x_{r+b-1},b-1]$. When $(n-r+3)-x_2=1$ or $(n'+1)-x_{r+b-1}=1$, we descend to an absurd situation. When $2-x_{r+1}=1$ and $2-x_{r+b-2}=1$, we descend to $0=[2,(n-r-b+5)-x_2,n'-x_{r+b-1},b-1]$ and $0=[2,(n-r+2)-x_2,(n'-b+3)-x_{r+b-1},b-1]$, respectively. The only zero continued fraction of length four with ends $2$ and $b-1$ is $[2,2,1,3]$ (and so $b=4$). Therefore, we obtain the cases (f) and (g), respectively.

If $(2-x_4)=1$ (it requires $r>3$), we must have $r\leq 5$. If $r=5$, we descend to $0=[2,(n-b+2)-x_2,n'-x_{r+b-1},b-1]$. It must be $[2,2,1,3]$, so we obtain (d). If $r=4$, we descend to $0=[2,(n+1)-x_2,2-x_3,2-x_{r+1},\ldots,2-x_{r+b-2},(n'+1)-x_{r+b-1},b-1]$. When $(n'+1)-x_{r+b-1}=1$, we descend to an absurd situation, while when $2-x_3=1$ and $2-x_{r+b-2}=1$, we descend to $0=[2,(n-b+2)-x_2,n'-x_{r+b-1},b-1]$ and $0=[2,n-x_2,(n'-b+2)-x_{r+b-1},b-1]$, respectively. The only zero continued fraction of length four with ends $2$ and $b-1$ is $[2,2,1,3]$ (and so $b=4$). We obtain (c) and (b) respectively.

If $(2-x_{r})=1$ (it requires $r>3$), we must have $r\leq 5$. The case $r=4$ was considered in the previous one. If $r=5$, we must have $b=3$. Then, we descend to a zero continued fraction of length four, where both ends are $2$s, which is impossible.

If $(3-x_{r+1})=1$, we must have ($r=3$ and $3\leq b \leq 4$) or ($r\geq 3$ and $b=3$). In either the second case or $b=3$ in the first one, we descend to a zero continued fraction of length four, whose ends are $2$s, which is impossible. If $b=4$ in the first case, then we descend to $0=[2,n-x_2,(n'-1)-x_{r+b-1},3-x_{r+b}]$. It must be $[2,2,1,3]$, so we obtain (a).

If $(2-x_{r+2})=1$ (it requires $b>3$), we must have $b\leq 5$. The case $b=4$ is part of the next case. If $b=5$, we descend to $0=[2,(n+1)-x_2,2-x_3,(n'-r+2)-x_{r+b-1},4]$. If $(2-x_3)=1$, we descend to an absurd situation, whereas when $(n'-r+2)-x_{r+b-1}=1$, we obtain (i).

If $(2-x_{r+b-2})=1$ (it requires $b>3$), we descend to $0=[2,(n+1)-x_2,3-x_3,2-x_4,\ldots,2-x_{r},2-x_{r+1},(n'-b+4)-x_{r+b-1},b-1]$. The cases $n+1-x_2=1$, $3-x_3=1$, and $2-x_4=1$ have already been considered, so we have two cases to analyze. If $2-x_{r+1}=1$, we descend to $0=[2,(n+1)-x_2,2-x_3,(n'-b-r+6)-x_{r+b-1},b-1]$. When $2-x_{3}=1$, we descend to $0=[2,n-x_2,(n'-b-r+5)-x_{r+b-1},b-1]$. The latter must be $[2,2,1,3]$, so we obtain (h). When $(n'-b-r+6)-x_{r+b-1}=1$, we descend to $0=[2,n-x_2,b-3]$. The latter must be $[2,1,2]$, so we obtain (j). If $(n'-b+4)-x_{r+b-1}=1$, we descend to $0=[2,(n+1)-x_2,2-x_3,b-r]$. Among the zero continued fractions of length $4$, the only ones with the first entry equal to $2$ are $[2,2,1,3]$ and $[2,1,3,1]$. Only the first case is possible, so we obtain (k).

If $(n'+1)-x_{r+b-1}=1$, we descend to $0=[2,(n+1)-x_2,2-x_3]$. The latter must be $[2,1,2]$, so we obtain (e).
\end{proof}


\end{document}